\documentclass[english,12pt]{amsart}
\usepackage{graphicx}
\usepackage{amsmath}
\usepackage{amssymb}
\usepackage{amsthm}
\usepackage{xcolor}
\usepackage{enumerate}
\usepackage[normalem]{ulem}

\usepackage[english]{babel}

\newtheorem{theorem}{Theorem}[section]
\newtheorem{definition}[theorem]{Definition}
\newtheorem{maintheorem}{Theorem}
\newtheorem{lemma}[theorem]{Lemma}

\newtheorem{remark}[theorem]{Remark}

\DeclareMathSymbol{\shortminus}{\mathbin}{AMSa}{"39}

\newcommand{\interior}{{\rm int }}

\newcommand{\N}{\mathbb{N}}
\newcommand{\e}{\varepsilon}

\newcommand{\rah}{\stackrel{\Phi}{\rightarrow}}

\begin{document}
\title[Ergodic measures with infinite entropy]{Ergodic measures with infinite entropy}

\author{Eleonora Catsigeras}

\author{Serge Troubetzkoy}

\address{Instituto de Matem\'atica y Estad\'istica \lq\lq Prof.\ Ing.\  Rafael Laguardia\rq\rq \  (IMERL), Universidad de la Rep\'ublica, Av.\ Julio Herrera y Reissig 565, C.P. 11300, Montevideo, Uruguay}
\email{eleonora@fing.edu.uy}
\urladdr{http:/fing.edu.uy/{\lower.7ex\hbox{\~{}}}eleonora}

\address{Aix Marseille Univ, CNRS, Centrale Marseille, I2M, Marseille, France}
\address{postal address: I2M, Luminy, Case 907, F-13288 Marseille Cedex 9, France}
\email{serge.troubetzkoy@univ-amu.fr}
\urladdr{http://www.i2m.univ-amu.fr/perso/serge.troubetzkoy/}

\thanks{
 We thank the anonymous referee for his careful reading of our text, and suggesting many improvements.
We gratefully acknowledge support of the project   \lq\lq Sistemas Din\'{a}micos\rq\rq \  by CSIC, Universidad de la Rep\'{u}blica (Uruguay).
The project leading to this publication has also received funding Excellence Initiative of Aix-Marseille University - A*MIDEX and Excellence Laboratory Archimedes LabEx (ANR-11-LABX-0033), French ``Investissements d'Avenir'' programmes.}

\begin{abstract}
We construct ergodic, probability measures with infinite metric entropy for  generic continuous maps and  homeomorphisms on  compact manifolds. We also construct sequences of such measures that converge to a zero-entropy measure.

\end{abstract}\maketitle

\section{Introduction}\label{s1}
Let $M$ be a $C^1$ compact manifold of finite dimension $m \geq 1$, equipped with  a Riemannian metric dist.   The manifold $M$ may or may not have  boundary.
Let  $C^0(M)$ be  the space of   continuous maps $f: M \rightarrow M$ with the metric:
$$\|f-g\|_{C^0} := \max_{x \in M} {{\rm  dist}}(f(x), g(x)), \qquad \forall \ f,g \in C^0(M).$$
We denote by ${{\rm Hom}}(M)$  the space of homeomorphisms $f: M \rightarrow M$ with the metric:
$$\|f-g\|_{{{\rm Hom}}} := \max \Big \{\|f-g\|_{C^0}, \ \|f^{-1} - g^{-1}\|_{C^0} \Big \} \ \ \   \forall \  f,g \in {{\rm Hom}}(M).$$

We note that the topology induced in Hom$(M)$  by the above metric is the subspace topology induced by $C^0(M)$. Nevertheless, the metrics are different.

  Since the metric spaces $C^0(M)$ and ${{\rm Hom}}(M)$ are complete, the Baire category theorem holds, namely the countable intersection of open dense sets is dense.
A subset ${\mathcal S} \subset C^0(M)$ (or ${\mathcal S} \subset {{\rm Hom}}(M)$) is called \em a $G_{\delta}$-set   \em  if it is the countable intersection of open subsets of $C^0(M)$ (resp.\ $ {{\rm Hom}}(M)$).
We say that  a property  $P$ of the maps $f \in C^0(M)$ (or $f \in {{\rm Hom}}(M)$) is  \em generic, \em or that   generic maps  satisfy $P$, if the set of maps  that satisfy $P$ contains a \em dense   $G_{\delta}$-set \em in $C^0(M)$ (resp.\ $ {{\rm Hom}}(M)$).

The main result of this article is the following theorem.

\begin{maintheorem}
\label{Theorem1}
 The generic map $f \in C^0(M)$  has an ergodic Borel probability measure $\mu$ such that  $h_{\mu}(f) = + \infty$
 and there exists $p \ge 1$ such that $\mu$ is mixing for the map $f^p$.
\end{maintheorem}

\noindent{\bf Remark.} In the case that $M$ is a compact interval,  Theorem \ref{Theorem1} was proved in  {\cite[Theorem 3.9 and p.33, paragraph 2]{CT2017}}. { The statements and proofs of \cite{CT2017}  also hold for continuous maps of the circle $S^1$. In fact, each $f\in C(S^1)$ can be represented by a continuous map $f$ in $[0,1]$ such that $f(0)=f(1)$. In the proof of genericity of the  properties studied in  \cite{CT2017},  no restrictions on the images of the endpoints  $0$ and $1$ are imposed. In particular the proof  of the denseness condition was obtained by perturbing the map only in the interior of a finite number of compact subintervals contained in $[0,1]$. 
Finally,  if the one-dimensional compact manifold $M$ is not connected, the arguments of \cite{CT2017} applied to a recurrent connected component of $M$,  extend the results to $C(M)$.
This is why in this paper we will prove Theorem \ref{Theorem1} only for $m$-dimensional manifolds for $m \geq 2$.}

\vspace{.15cm}

Yano proved that generic continuous maps of compact manifolds with or without boundary have infinite topological entropy \cite{Yano}. Therefore, from the variational principle, there exists  invariant measures with  metric entropies as large as required. Nevertheless, this property alone does not  imply the existence of invariant measures with infinite metric entropy.   In fact, it is well known that  the metric entropy function $\mu \rightarrow h_{\mu}(f)$ is not upper semi-continuous for $C^0$-generic systems.  Moreover, we prove that it is \em strongly \em non upper semi-continuous in the following sense:

\begin{maintheorem}
\label{Theorem2}
For a generic map  $f \in C^0(M)$  there exists a sequence of ergodic measures $\mu_n$ such that for all $n \ge 1$ we have
$h_{\mu_n}(f) = +\infty$ and  $$ {\lim _{n \rightarrow +\infty}\!\!\! }^*\,  \mu_n = \mu  \mbox{  with  } h_{\mu }(f)= 0,$$
  where $\lim^*$ denotes the limit in the space of probability measures endowed with the weak$^*$ topology.
\end{maintheorem}
{Theorem \ref{Theorem2}  holds for any $m$-dimensional manifold, including $m = 1$. In this paper we will prove it for $m \geq 2$, but the proof for $m=1$ is easily obtained by repeating our proof after some trivial substitutions that are explained at the beginning of Section \ref{sectionCleanSequences}.

 Even if we had    a priori some   $f$-invariant measure $\mu$ with infinite metric entropy, we do not know if this property alone implies the existence of ergodic measures with infinite metric entropy as Theorems \ref{Theorem1} and \ref{Theorem2} state. Actually,  if $\mu$ had  infinitely many ergodic components, the  proof   that the metric entropy of at least one of those ergodic components must be larger or equal   than the entropy of $\mu$, uses the upper semi-continuity of the metric entropy function (see for instance \cite[Theorem 4.3.7, p. 75]{Keller}).

 Yano also proved that generic homeomorphisms on manifolds of dimension 2 or larger, have infinite topological entropy \cite{Yano}. Thus one wonders if Theorems \ref{Theorem1} and \ref{Theorem2} hold also for  homeomorphisms.  We give a positive answer to this question for $m \geq 2$.  If $M$ is one-dimensional  then a homeomorphisms of $M$ has zero topological entropy, so the following two theorems do not hold for one-dimensional manifolds.

\begin{maintheorem}
\label{Theorem3}
If  $\text{dim}(M) \ge2$, then the generic homeomorphism $f \in {{\rm Hom}}(M)$  has an ergodic Borel probability measure $\mu$
satisfying   $h_{\mu}(f) = + \infty$ and there exists $p \ge 1$ such that $\mu$ is mixing for the map $f^p$.
\end{maintheorem}

\begin{maintheorem}
\label{Theorem4}
If  $\text{dim}(M) \ge2$,  then for a generic homeomorphism  $f \in {{\rm Hom}}(M)$    there exists a sequence of ergodic measures $\mu_n$ such that
  for all $n \ge 1$ we have
$h_{\mu_n}(f)  = +\infty$ and  $$ {\lim _{n \rightarrow +\infty}\!\!\! }^*\,  \mu_n = \mu  \mbox{  with  } h_{\mu }(f) = 0.$$
\end{maintheorem}

To prove Theorems \ref{Theorem1}, \ref{Theorem3} and \ref{Theorem4} in dimension two or larger, we construct a  family ${\mathcal H}$, called models, of continuous maps in the cube $[0,1]^m$, including some homeomorphisms of the cube onto itself,  which have a complicated behavior on a Cantor set (Definition \ref{DefinitionModel}).
{ To prove Theorem  \ref{Theorem2} in dimension one,  the family ${\mathcal H}$ of model maps in $M$  we use is  the set of continuous maps     that have an \lq\lq atom doubling cascade\rq\rq, according to \cite[Definition 35]{CT2017}.}

In any dimension $m\geq 1$, a periodic shrinking box is a compact set $K \subset M$  that is homeomorphic to the cube $[0,1]^m$ and such that for some $p \geq 1$:  $K, f(K), \ldots, f^{p-1}(K)$ are pairwise disjoint and $f^p(K) \subset {{\rm int}}(K)$ (Definition \ref{DefinitionPerShrBox}).

The main steps of the proofs of Theorems \ref{Theorem1} and \ref{Theorem3} are the following results.

\begin{description}
\item[\bf Lemma \ref{LemmaMain}]  \em  For $m \geq 1$, any model $\Phi \in {\mathcal H}$ in the cube $[0,1]^m$ has a $\Phi$-invariant mixing measure $\nu$ such that $h_{\nu}(\Phi) = +\infty$. \em

\item[\bf Lemmas \ref{Lemma1} and \ref{Lemma1b}]
\em For $m \geq 1$, generic maps in $C^0(M)$,  and generic homeomorphisms of $M$,  have a periodic shrinking box. \em

\item[\bf Lemmas \ref{Lemma2} and \ref{Lemma2b}] \em If $m \geq 1$  generic maps  $f \in C^0(M)$,  and if $m \geq 2$ also generic homeomorphisms of $M$, have a periodic shrinking box $K$ such that the return map $f^p|_K$  is topologically conjugated to a model $\Phi \in {\mathcal H}$. \em
\end{description}

{   We prove and use Lemma \ref{LemmaMain} only for $m \geq 2$ since the case $m=1$ was proven  in \cite[Theorem 38]{CT2017}. The other results in the above list will be fully proven here even in the case $m=1$ independently of \cite{CT2017}.}

\noindent A \em good sequence of periodic shrinking boxes \em is a sequence $  \{K_n\}_{n \geq 1} $  of   periodic shrinking boxes which accumulate (with  the Hausdorff distance)  on a periodic point $x_0$, and  moreover their iterates $f^{j}(K_n) $ also accumulate on the periodic orbit of $x_0$, uniformly for $j \geq 0$ (see Definition \ref{definitionCleanSequence}).
The main tools used in the proofs of  Theorems \ref{Theorem2} and \ref{Theorem4} are the statements of Theorems \ref{Theorem1} (for $m \geq 1)$ and \ref{Theorem3}, Lemmas \ref{Lemma1}, \ref{Lemma1b}, \ref{Lemma2} and \ref{Lemma2b},  together with

\begin{description}
\item[\bf Lemma \ref{lemma4}]
\em  For $m \geq 1$  a generic  map  $f \in C^0(M)$,  and for $m \geq 2$ a generic homeomorphism $f$  has a  good sequence $\{K_n\}$ of  boxes, such that the return map $f^{p_n}|_{K_n}$  is topologically conjugated to a   model $\Phi_n \in {\mathcal H}$. \em

\end{description}

\section{Construction of the family of  models.} \label{SectionModels}

We call a compact set $K \subset D^m:=[0,1]^m$ or more generally $K \subset M$ (where $M$ is an $m$-dimensional manifold with $m \geq 1$) \em a box \em  if it is homeomorphic to  $D^m$.
Models are certain continuous maps of   $D^m$ that we will define in this section.

 We denote by ${{\rm Emb}}(D^m)$ the space of embeddings $\Phi: D^m \rightarrow D^m$ (i.e., $\Phi$ is a homeomorphism onto its image  included in $D^m$),  with the  topology induced as a subspace of $C^0(D^m)$.

 \begin{definition}
 \label{definitionModelDim1}
 { \em For $m=1$  a model is a map that has  an \lq\lq atom doubling   cascade\rq\rq \,  according to \cite[Definition 35]{CT2017} and  the family  ${\mathcal H}$  is the set of all model maps.}\end{definition}

From here to the end of this section we  construct the family ${\mathcal H}$ of model map for $m \geq 2$.
 \begin{definition}
 {\bf  ($\Phi$-relation from a box to another).} \em \label{definition-hRelation}\hspace*{\fill} \\
  Let $\Phi \in C^0(D^m)$. Let $B,C \subset {{\rm int}}(D^m) $ be two  boxes. We write  $$B \rah C \mbox{ if }
  \Phi(B) \cap {{\rm int}}(C) \neq \emptyset.$$  Observe that this condition is open in $C^0(D^m)$.

  Let ${\mathcal A}$ be a finite family of boxes. Denote
    $${\mathcal A}^{2*}: = \{(B, C) \in {\mathcal A}^2: \ \ B \rah C\}, $$
$${\mathcal A}^{3*}: = \{(D,B, C) \in {\mathcal A}^3: \ \ D \rah B, \ \ \ B \rah C\}. $$
 \end{definition}

For all $n \geq 0$ we will define atoms of   generation $n $ for a map $\Phi    \in {C}^0(D^m)$ by Definition \ref{definitionAtomsGeneration-n}. 

\begin{figure}[bh]
\centering
\hspace{-0.2in}\vspace{0.2in}\includegraphics[scale=.40]{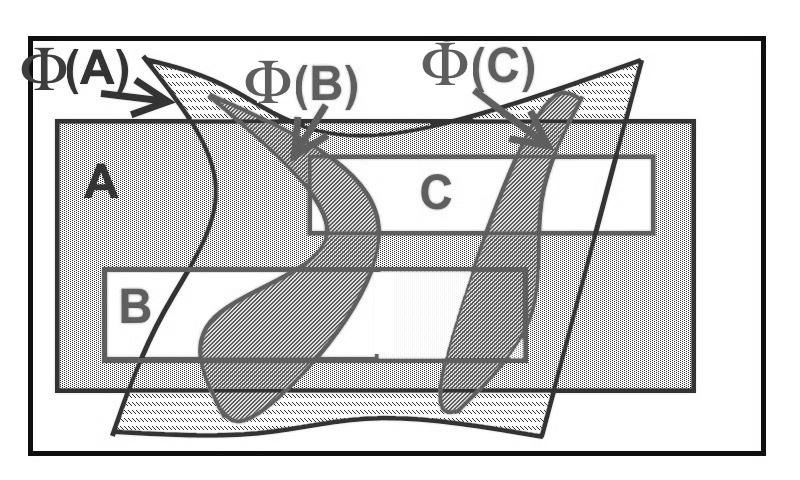}
\caption{The atom $A$ of generation 0 and two atoms $B, C$ of generation 1 for a map $\Phi$ of $D^2$. \label{FigureAtomsGen-0-1}}
\end{figure}

  \begin{definition} \label{definitionAtomsGeneration-n}
 {\bf Atoms of  generations $\mathbf{0 \le n \le k}$} \em (See Figure  \ref{FigureAtomsGen-n}) \hspace*{\fill} \\
Fix $\Phi \in C^0(D^m)$ and collections of boxes ${\mathcal A}_{0}, \mathcal A_1, \dots \mathcal A_k$ contained in the interior of $D^m$.
 For $n \ge 1$ and for $(D, B, C )  \in {\mathcal A}^{3 *}_{n-1}$  we define
 $$\Omega_{n}(B): = \{G \in   {\mathcal A}_{n} \colon G \subset {{\rm int}}(B)\},$$
 $$\Omega_{n}(D, B) := \{G \in   \Omega_{n}(B) \colon D \rah G\},$$
$$\Gamma_{n}(D, B, C) := \{G \in \Omega_{n}(D, B) \colon G \rah C\}.$$
Suppose that the following conditions hold for all $1 \le n \le k$;
\begin{enumerate}[i)]
\item  The family ${\mathcal A}_n$ consists of $2^{n^2}$  pairwise disjoint boxes.
\item  For all $B \in {\mathcal A}_n$ we have:
     $$ \#\{C \in {\mathcal A}_n \colon B \rah C\}= 2^n, \ \ \ \ \ \  \#\{D \in {\mathcal A}_n \colon D \rah B\}= 2^n.$$
\end{enumerate}

\begin{enumerate}[a)]
\item   $\#\Omega_n(B) = \#\Omega_n(B')   \  \forall \, B, B' \in {\mathcal A}_{n-1}, \ \ \mbox{and }$
 $
 {\mathcal A}_n = \bigcup_{B \in {\mathcal A}_{n-1}} \Omega_n(B).$

\item  For all  $(D,B) \neq (D',B') \in {\mathcal A}^{2*}_{n-1}$,

 $\#\Omega_n(D, B) = \#\Omega_n(D', B')   \  \mbox{and} \ \Omega_n(D, B) \cap \Omega_n(D', B')= \emptyset . \ \ $

 Besides, $
 \Omega_n(B) = \bigcup_{D: (D,B) \in {\mathcal A}^{2*}_{n-1}} \Omega_n(D, B)\  $ for all $B  \in {\mathcal A}_{n-1}. $

\item For all $(D,B,C) \neq (D', B', C') \in {\mathcal A}_{n-1}^{3*}$,
   $$\#\Gamma_n (D,B, C)= \#\Gamma_n (D',B', C') \ \mbox{ and } \ \Gamma_n (D,B, C)\cap \Gamma_n (D',B', C') = \emptyset,$$
and  for all $(D, B) \in {\mathcal A}_{n-1}^{2*}$,

   \begin{equation}
\label{eqn99a}\Omega_n(D,B) = \bigcup_{C \colon (B,C) \in {\mathcal A}_{n-1}^{2*}} \Gamma_n(D,B,C),\end{equation}

\item For each   $(D, B,C) \in {\mathcal A}_{n-1}^{3*}$ and  for each $G \in \Gamma_n(D, B, C)$
 $$\Phi(G) \cap E  = \emptyset \ \forall \ E  \in {\mathcal A}_n \setminus \Omega_n(B,C).$$
We  call the members of $\mathcal{A}_n$  \em atoms of generation $n$ \em or \em $n$-atoms, \em
 \end{enumerate}
  \end{definition}

\begin{figure}[t]
\centering
\hspace{-0.2in}\includegraphics[scale=.55]{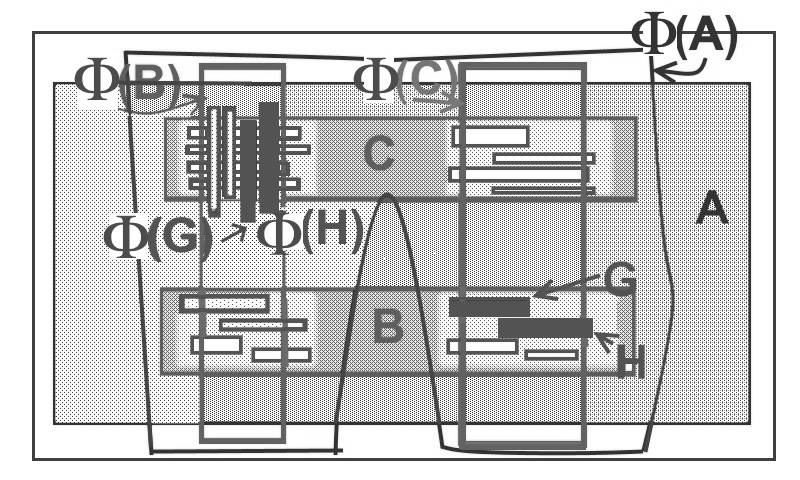}
\caption{An atom $A$ of generation 0, two atoms $B,C$ of generation 1, and  $16$ atoms of generation~2. In particular the two atoms $G, H$ of generation 2 satisfy $\Gamma_2(C,B,C) = \{G, H\}$.} \label{FigureAtomsGen-n}
\end{figure}

  \begin{remark}
  \label{remarkAtomsGeneration-n} \em
  From the   conditions (i), (ii) and (a) to (d) of Definition \ref{definitionAtomsGeneration-n} we deduce the following properties of the families of atoms for $\Phi \in C^0(D^m)$:

  \noindent{$\bullet$} $\#\Omega_n(B) =2^{2n-1}$ for all  $B \in {\mathcal A}_{n-1}$.  In fact, the families   $\Omega_n(B)$  are pairwise disjoint because any two different atoms of generation $n$ are disjoint. Therefore, from  condition a), we obtain  $$
  \#   {\mathcal A}_n  = (\#{\mathcal A}_{n-1}) (\#\Omega_n(B)) =2^{(n-1)^2}  (\#\Omega_n(B))  =     2 ^{n^2},$$ $$ \mbox{ hence} \ \#\Omega_n(B))=2^{2n-1}.  $$

 \noindent{$\bullet $}  $\# \Omega_n(D, B) = 2^n, $ for all $(D,B) \in {\mathcal A}_{n-1}^{2*}.$
   In fact, from condition b),
$$\Omega_n(B) = \bigcup_{D \colon (D,B) \in {\mathcal A}_{n-1}^{2*}} \Omega_n(D,B), \qquad  \forall \ B \in {\mathcal A}_{n-1}, $$
where the families of atoms in the above union are pairwise disjoint. Thus, for any $B \in {\mathcal A}_{n-1}$ we have
$$\#\Omega_n(B) =\big( \# \{D \in {\mathcal A}_{n-1} \colon D \rah B\} \big) \cdot \big (\# \Omega_n(D,B)\big) =
$$ $$
2^{n-1} \cdot \big (\# \Omega_n(D,B)\big)  = 2^{2n-1},
$$
hence $\big (\# \Omega_n(D,B)\big) = 2^n.$

\noindent{$\bullet$} $\#\Gamma_n (D,B, C) = 2\ \ \forall  \ (D,B,C) \in {\mathcal A}_{n-1}^{3*}.$ In fact, from conditions (ii) and (c), for each 2-tuple $(D,B) \in {\mathcal A}_{n-1}^{2*}$ the collection  $\Omega_n(D, B)$ is partitioned into $2^{n-1}$ pairwise disjoint sub-collections $\Gamma_n(D,B,C)$, where the atoms $C \in {\mathcal A}_{n-1}$ are such that $B \rah C$. Since $\#\Omega_n(D, B) = 2^n$ (proved above), we deduce  that
    $\#\Gamma_n (D,B, C) = 2$.

 For example, in Figure \ref{FigureAtomsGen-n} we have $ \Gamma_2(C,B,C) = \{F,G\}  $.

 \noindent{$\bullet$} As a straightforward consequence of conditions a), b) and c) we obtain \begin{equation}
\label{eqn99}
{\mathcal A}_n = \bigcup_{(D, B, C) \in {\mathcal A}_{n-1}^{3*} } \Gamma_n(D, B,C),\end{equation}
where the families of atoms in the union  are pairwise disjoint.

 \noindent{$\bullet$ } For each   $(D, B,C) \in {\mathcal A}_{n-1}^{3*}$,  for each $G \in \Gamma_n(D, B, C)$ and for all $E \in {\mathcal A}_n$, either
  $      G  \rah E$, and this occurs if and only if  $  E \in \Omega_n(B,C), $ or
  $ \Phi(G) \cap E  = \emptyset $, and this occurs if and only if  $E  \not \in  \Omega_n(B,C).$

 In fact, from condition d), if $\Phi(G) \cap E  \neq \emptyset$ then $  E \in \Omega_n(B,C) $. So, recalling condition (ii) of Definition \ref{definitionAtomsGeneration-n}, we obtain
 $$2^n = \#\{E \in {\mathcal A}_n \colon G  \rah E\} \leq $$
 $$\#\{E \in {\mathcal A}_n \colon \Phi(G) \cap E \neq \emptyset\}\leq \#\Omega_n(B,C)= 2^n.$$
 Hence, all the above inequalities are equalities and the assertion is proved.

\noindent{$\bullet$} For any pair $(G,E) \in {\mathcal A}_n^2$:
$$   G \rah E  \mbox{ if and only if } \ \exists \  (D, B,C) \in {\mathcal A}_{n-1} ^{3*} \mbox{ such that  }$$ $$ G \in \Gamma_n(D,B, C), \ E \in \Omega_n(B,C), $$
This is a   restatement of the above assertion.

\noindent{$\bullet$} $\#{\mathcal A}_n^{3*} = 2^{n^2 + 2n}$

In fact, all the $3$-tuples  $(D,B,C) \in {\mathcal A}_n^{3*}$ can be constructed by choosing freely $D \in {\mathcal A}_n$, later  choosing $B \in {\mathcal A}_n$ such that $D\stackrel{\Phi_n}{\rightarrow}B$, and finally   choosing  $C \in {\mathcal A}_n$ such that $B\stackrel{\Phi_n}{\rightarrow}C$. Taking into account   equalities  i) and ii) of Definition \ref{definitionAtomsGeneration-n},  we deduce
$$\#{\mathcal A}_n^{3*}= \#\{(D,B,C)\in {\mathcal A}_n^3\colon D\stackrel{\Phi_n}{\rightarrow}B, B \stackrel{\Phi_n}{\rightarrow}C \}= $$
$$(\#{\mathcal A}_n) \cdot (\#\{B \in {\mathcal A}_n \colon D\stackrel{\Phi_n}{\rightarrow}B\}) \cdot (\#\{C \in {\mathcal A}_n \colon B \stackrel{\Phi_n}{\rightarrow}C\}) = $$ $$ 2^{n^2} 2^n 2^n = 2^{n^2 + 2n}. $$
  \end{remark}

  \noindent {\bf Notation.}     In certain statements  we refer to families of atoms for several different maps.
When necessary we will use the following   notation:  $A \in {\mathcal A}_n(\Phi)$ or $(B, D) \in {\mathcal A}_n^{2*}(\Phi)$, etc., where $\Phi$ is the map to which the family of atoms is referred.

\begin{definition} {\bf (Models)}
\label{DefinitionModel} \em

We call  $\Phi \in C^0(D^m)$  \em a model \em  if
 $\Phi(D^m) \subset {{\rm int}}(D^m),$
and there exists a sequence  $\{{\mathcal A}_n\}_{n \geq 0}$
of finite families of pairwise disjoint boxes contained in ${{\rm int}}(D^m)$ that are atoms of generations $n \ge 0$ for $\Phi$ such that
\begin{equation}
\label{eqnLimitDiamAtoms=0}
\lim_{n \rightarrow + \infty} \max_{A \in {\mathcal A}_n} {{\rm diam}}A = 0.\end{equation}
Denote by ${\mathcal H}$  the family of  all the models in $C^0(D^m)$.

\end{definition}

  For each fixed $n \geq 1$ the four conditions a) to d) of Definition
   \ref{definitionAtomsGeneration-n}, are open conditions. So, for fixed $n \geq 0$, and for any given map $\Phi$ having families ${\mathcal A}_0, {\mathcal A}_1, \ldots, {\mathcal A}_n$ of atoms of generations $0, 1, \ldots, n$, the set of maps that have the same families of atoms of generation up to $n$ as $\Phi$ (for that fixed $n$ and not necessarily  for all $n$) is an open set.
   Moreover,  the condition   $\Phi(D^m) \subset {{\rm int}}(D^m)$ is open.
   \begin{definition}
   \label{Definition HsubPhi} \em
 For any  $\Phi \in \mathcal H$, we denote by   ${\mathcal H}_{\Phi} $ the family of maps in $C^0(D^m)$ that have the same atoms of all generations as $\Phi$. Note that ${\mathcal H}_{\Phi} \subset {\mathcal H} $.
   \end{definition}

  \begin{remark}
\label{RemarkHisG_delta} \em
    We deduce that,
   for any given $\Phi \in \mathcal H$, the family $\mathcal H_{\Phi} \subset {\mathcal H} $ is a nonempty $G_{\delta}$-set  in $C^0(M)$. In other words, $\mathcal H$ contains a nonempty $G_{\delta}$-set if it is nonempty.

On the other hand, if there exists $\Phi \in {\mathcal H} \cap {{\rm Emb}}(D^m)$ then $H_{\Phi}$ is not necessarily composed by embeddings of $D^m$. Nevertheless, it contains $H_{\Phi} \cap {{\rm Emb}}(D^m)$, which is a nonempty $G_{\delta}$-set in ${{\rm Emb}}(D^m)$. Thus
${\mathcal H} \cap {{\rm Emb}}(D^m)$   contains a nonempty  $G_{\delta}$-set in ${{\rm Emb}}(D^m)$ if ${\mathcal H} \cap {{\rm Emb}}(D^m) \neq \emptyset$.

Note that the nonempty $G_{\delta}$-set ${\mathcal H}_{\Phi}$ (if $\Phi \in {\mathcal H} \neq \emptyset$) is not necessarily dense in $C^0(D^m)$!
\end{remark}

\noindent {\bf Construction of models.} \label{SectionFamilyModelsIsGood}\\
The rest of this section is dedicated to the proof of the following lemma.

\begin{lemma}
    \label{LemmaModelHEmbNonempty}
The family ${\mathcal H}$ of models contains the nonempty $G_{\delta}$-set ${\mathcal H}_{\Phi}$ (defined in
Definition \ref{Definition HsubPhi}) in $C^0(D^m)$ for any chosen $\Phi \in {\mathcal H}$.

 The family ${\mathcal H} \cap {{\rm Emb}}(D^m)$ contains the nonempty  set ${\mathcal H}_{\Phi}\cap {\rm Emb}(D^m)$, which is a $G_{\delta}$-set in ${\rm Emb}(D^m)$, for any chosen $\Phi \in {\mathcal H}\cap {{\rm Emb}}(D^m)$.
\end{lemma}

{
\begin{proof}
Lemma \ref{LemmaConstruccionModelPsifisPhi} stated below implies that  ${\mathcal H}\cap {{\rm Emb}}(D^m) \neq \emptyset$,
and so ${\mathcal H} \neq \emptyset$. Repeating the argument of
 Remark \ref{RemarkHisG_delta}  proves Lemma \ref{LemmaModelHEmbNonempty}.
 \end{proof}
 }

  \begin{lemma}
    \label{LemmaConstruccionModelPsifisPhi}

    For all $f \in {{\rm Emb}}(D^m)$ such that $f (D^m) \subset {{\rm int}}(D^m),$  there exists $\psi \in {{\rm Hom}}(D^m)$ and $\Phi \in {\mathcal H} \cap {{\rm Emb}}(D^m)$ such that \em
  $$\psi|_{\partial D^m} = \mbox{id}|_{\partial D^m} \ \ \mbox{ \em and } \ \ \psi \circ f = \Phi.$$
  \end{lemma}

We now outline the strategy of the proof of Lemma \ref{LemmaConstruccionModelPsifisPhi}. The homeomorphisms  $\psi$ and $\Phi$ are constructed as   limits of respective convergent sequences $\psi_n \in {{\rm Hom}}(D^m)$ and $\Phi_n \in {{\rm Emb}}(D^m)$, such that $\psi_n \circ f= \Phi_n$ for all $n \geq 0$. The embedding $\Phi_n$, by an inductive hypothesis, has atoms of generations $0, 1, \ldots, n$, and $\Phi_{n+1} $ will be constructed so it coincides with $\Phi_n$
outside the interiors of all the atoms of generation $n$ for $\Phi_n$.  Hence,   the collections of atoms of generation   $0,1, \ldots, n$ for $\Phi_n$ is also
a collection of atoms of the same generations for $\Phi_{n+1}$ (see the proof of assertion a) of Lemma \ref{LemmaConstruction(n+1)atoms}).

 To change $\Phi_n$ in the interior of each atom $A$   of generation $n$ for $\Phi_n$, we will change the homeomorphism $\psi_n$ only inside  some adequately defined boxes $f(R)$ such that $R \subset {{\rm int}}(  A) $  is a box (recall that $f$ is an embedding). We will so construct $\psi_{n+1}|_{f(R)}$ such that $ \psi_{n+1} |_{\partial f(R)} = \psi_{n}|_{\partial f(R)}$, and finally extend  $\psi_{n+1}(x) := \psi_n(x)$ for all $x$ in the complement   of the union of  all the boxes $f(R)$.

  Such new homeomorphism $\psi_{n+1}$, if adequately constructed inside the boxes $f(R)$, will   allow us to define the atoms  of generation $n+1$ for $\Phi_{n+1} = \psi_{n+1} \circ f$. These  atoms of generation $n+1$ for $\Phi_{n+1}$  will be many little boxes in the interior of each box  $f^{-1} (f(R) ) =R \subset A$, where $A$ is an atom of generation $n$  both for $\Phi_n$ and for $\Phi_{n+1}$.

Lemma \ref{LemmaConstruccionModelPsifisPhi}  will be proved by induction in several technical lemmas. One inductive hypothesis in the proof is that for a fixed $n \ge 0$ we have constructed an embedding $\Phi_n$ along with associated atoms of generations $0,1, \ldots, n$.
For each $(P,Q) \in {\mathcal A}_n^{2*}$, we
 will   choose  a  connected component  $S(P,Q)$ of $\Phi_n(P) \cap Q $.
 For each $(D,B,C) \in {\mathcal A}_n^{3*}$ we choose     two disjoint boxes $G_0(D,B,C), \ G_1(D,B,C)$ contained in    $\interior(S(D,B) \cap \Phi_n^{-1} S(B,C))$.~By an additional inductive hypothesis on $\Phi_n$ a choice of the connected components $S(D,B)$ and $S(B,C)$ is assumed to exist such that the interior of this intersection is nonempty.

\begin{definition}
\label{remark S(P,Q), G_i}
\em
We provisionally adopt an abusive notation for the  families of such boxes $G_{\cdot}(\cdot, \cdot, \cdot)$.   Even if   such  boxes   are not probably atoms of generation $n+1$ for $\Phi_n$, we  use the same  notation  as if they   were. This is due to   the purpose, realized later in the proof of Lemma \ref{LemmaConstruccionModelPsifisPhi}, of modifying $\Phi_n$  to construct  a new embedding $\Phi_{n+1}$  for which the same atoms up to  generation $n$ for $\Phi_n$ are also atoms up to generation $n$ for $\Phi_{n+1}$, and besides the boxes $G_{\cdot}(\cdot, \cdot, \cdot)$ are the atoms of generation $n+1$   for $\Phi_{n+1}$. In brief, we first choose the boxes, candidates to be the atoms of generation $n+1$ for a new embedding $\Phi_{n+1}$, and later we construct $\Phi_{n+1}$. Let
$${\mathcal A}_{n+1} := \{G_j(D,B,C): j \in \{0,1\}, (D,B,C) \in {\mathcal A}_n^{3*} ( \Phi_n)\};$$
$$\Omega_{n+1}(B) := \{G_j(D,B,C): j \in \{0,1\}, (D,B,C) \in {\mathcal A}_n^{3*} (\Phi_n) \}    $$   for each fixed $B \in {\mathcal A}_n(\Phi_n)$;
 $$\Omega_{n+1}(D, B) := \{G_j(D,B,C): j \in \{0,1\}, B \stackrel{\Phi_n}{\rightarrow} C  \}$$   for each fixed $(D,B) \in {\mathcal A}_n^{2^*} (\Phi_n)$;
\begin{equation}
\label{eqn300}
\Gamma_{n+1}(D, B,C) := \{G_j(D,B,C): j \in \{0,1\} \} \end{equation} for each fixed $(D,B,C) \in {\mathcal A}_n^{3^*}(\Phi_n)$.
We will use this abusive notation in Lemmas \ref{lemma2.10-f} and \ref{LemmaConstruction(n+1)atoms} as well as in Remark \ref{remarkConditionsabc}.
\end{definition}

\begin{lemma}
\label{lemma2.10-f}
For all $(B,C) \in {\mathcal A}_n^{2*}(\Phi_{n})$ and for all $E \in {\mathcal A}_{n+1}$,
  $E \subset \Phi_n^{-1}(S(B,C))$  if and only if $E \in \Gamma_{n+1}(D, B, C)$ for some $D \in {\mathcal A}_n$ such that $D\stackrel{\Phi_{n}}{\rightarrow} B$.
\end{lemma}

\begin{proof}
  By the construction in Remark \ref{remark S(P,Q), G_i}, for all $E \in {\mathcal A}_{n+1}$ we have    $E \in \Gamma_{n+1}(D,B,C)= \{\Gamma_0(D,B,C), \Gamma_1(D,B,C)\}$ for some $(D,B,C) \in {\mathcal A}_n^{3^*}$. This means that $E=G_j(D,B,C) \subset {{\rm int}} (S(D,B) \cap \Phi_{n}^{-1}(S(B,C)))$ for some $j = 0,1$. Therefore, $E \subset \Phi_n^{-1} (S(B,C))$ if and only if there exists $D \in {\mathcal A}_n$ such that $D \stackrel{\Phi_n}\rightarrow B$  and $E \in \Gamma_{n+1}(D,B,C)$.
\end{proof}

\begin{lemma}
\label{LemmaConstruction(n+1)atoms}

Suppose that
$$\interior(S(D,B) \cap \Phi_n^{-1} S(B,C)) \neq \emptyset \ \ \forall \ (D,B,C) \in {\mathcal A}_n^{3*} (\Phi_n).$$
Let $\Phi_{n+1} \in {{\rm Emb}}(D^m) $ be such that $\Phi_{n+1} (x) = \Phi_n(x)$ for all $x \not \in \cup_{\{(B,C) \in {\mathcal A}_n^{2*}(\Phi_n)\}}\ \interior( \Phi_n^{-1} S(B,C)) $.
Then,

\vspace{.15cm}

\noindent a)
For all $0 \leq j \leq n$ and for any two atoms $B, C \in {\mathcal A}_j^2(\Phi_n)$, we have $\Phi_n(B) = \Phi_{n+1}(B)$; hence $B \stackrel{\Phi_n}{\rightarrow} C$ if and only if $B \stackrel{\Phi_{n+1}}{\rightarrow} C$.

\vspace{.15cm}

\noindent b)  $\#{\mathcal A}_{n+1} = 2^{(n+1)^2}$ and $E \cap F = \emptyset$ for all $E, F \in {\mathcal A}_{n+1}$ such that $E \neq F.$

\vspace{.15cm}

\noindent c) The family ${\mathcal A}_{n+1}$ is partitioned into the   pairwise disjoint subfamilies $\Omega_{n+1}(B)$ where $B \in {\mathcal A}_n$. Besides     $\# \Omega_{n+1}(B) = 2^{2n+1}  $   and $  \Omega_{n+1}(B)= \{G \in {\mathcal A}_{n+1} \colon G \subset \rm{int}(B)\} $ for all $B \in {\mathcal A}_n$.

\vspace{.15cm}

\noindent d) For all $B \in {\mathcal A}_n$ the family of boxes $\Omega_{n+1}(B)$ is partitioned into  the  pairwise disjoint subfamilies $\Omega_{n+1} (D,B)$ where $D \in {\mathcal A}_n$ is such that $D \stackrel{\Phi_{n+1}}{\rightarrow} B$. Besides, for all $(D,B) \in {\mathcal A}_n^{2*}(\Phi_{n+1})$, we have
$\#\Omega_{n+1}(D, B) = 2^{n+1}$ and $\Omega_{n+1}(D, B) =\{G \in \Omega_{n+1}(B) \colon D \stackrel{\Phi_{n+1}}{\rightarrow} G\} $.

\vspace{.15cm}

\noindent e) For all $(D, B) \in {\mathcal A}_n^{2*}(\Phi_{n+1})$, the family of boxes $\Omega_{n+1}(D,B)$ is partitioned into the   pairwise disjoint subfamilies
$\Gamma_{n+1} (D,B, C)$, where $C \in {\mathcal A}_n$ is such that $B\stackrel{\Phi_ {n+1}}{\rightarrow} C$.

Besides, for all $(D,B,C) \in {\mathcal A}_n^{3*}(\Phi_{n+1})$, we have
 $\#\Gamma_{n+1}(D,B,C)= 2$  and
 $\Gamma_{n+1}(D,B,C) = \{G \in \Omega_{n+1}(D,B) \colon G \stackrel{\Phi_{n+1}}{\rightarrow} C\}.$

\vspace{.15cm}

\noindent f) For all   $(D, B,C) \in {\mathcal A}_{n}^{3*}(\Phi_{n+1})$,   for all $G \in \Gamma_{n+1}(D, B, C)$, and for all $E \in {\mathcal A}_{n+1}$,
 $$\Phi_{n+1}(G) \cap E  \neq \emptyset \ \mbox{ only if } \ E  \in   \Omega_{n+1}(B,C).$$
\end{lemma}

\begin{proof}
a) Let us prove assertion a) under the more general hypothesis $\Phi_{n+1} (x) = \Phi_n(x)$ for all $x \not \in \cup_{B  \in {\mathcal A}_n}\ \interior(B) $. (Note the $x \not \in \interior(S(D,B) \cap \Phi_n^{-1} S(B,C)) $ implies $x \not \in B$.)

By hypothesis $\Phi_n, \Phi_{n+1} \in \mbox{End} (D^m)$ and $\Phi_n|_{\partial A} = \Phi_{n+1}|_{\partial A }$ for the boxes $A \in {\mathcal A_j}$ for all $0 \leq j \leq n$ (recall that, from condition a) of definition \ref{definitionAtomsGeneration-n}, each atom  of generation $n$ for $\Phi_n$ is contained in the interior of an atom  of generation $0 \leq j \leq n$). Then $\Phi_{n+1}(A) = \Phi_{n}(A)$ for all $ B \in \cup_{0 \leq j \leq n}{\mathcal A_j}$. 
Part  a) follows immediately.

 b) By construction, $E= G_j(D,C,B), \ F= G_{j'}(D', B', C')$. If  $E \neq F$ then,
either $( D, C, B)= ( D', C', B')$ and $j \neq j'$, or $( D, C, B)\neq ( D', C', B')$. In the first case, by construction $$G_0(D, C, B) \cap G_1(D, C, B) = \emptyset,$$ in other words $E \cap F = \emptyset$. In the second case, either $D \neq D'$ or $B \neq B'$ or $C \neq C'$. By construction $G_j(D, B, C) \subset \Phi_n(D)\cap B \cap \Phi_n^{-1}(C)$ and $G_{j'}(D', B', C') \subset  \Phi_n(D')\cap B' \cap \Phi_n^{-1}(C')$. Since members of ${\mathcal A}_n $ are pairwise disjoint, and $\Phi_n \in {{\rm Emb}}(D^m)$, we deduce that $G_{j }(D , B , C ) \cap G_{j'}(D', B', C') = \emptyset$, hence $E \cap F = \emptyset$  as required.

By the construction in Remark \ref{remark S(P,Q), G_i}:
$${\mathcal A}_{n+1}= \bigcup_{(D, C, B)\in {\mathcal A}_n^{3*}}  \Gamma_{n+1}(D,B,C),$$
where the families in the union are pairwise disjoint and each one has   2 different boxes of ${\mathcal A}_{n+1}$. Therefore, taking into account the last assertion of  Remark \ref{remarkAtomsGeneration-n}, we deduce that $$\#{\mathcal A}_{n+1} = 2 \cdot(\#{\mathcal A}_n^{3*}) = 2 \cdot 2^{n^2 + 2n}= 2^{(n+1)^2}.$$

 c)  Using the notation at the end of Remark \ref{remark S(P,Q), G_i} of   ${\mathcal A}_{n+1}$ and  $\Omega_{n+1}(B)$, we have
$${\mathcal A}_{n+1} = \bigcup_{B \in {\mathcal A}_n} \Omega_{n+1}(B).$$
Besides, for all $G \in {\mathcal A}_{n+1}$, $G \subset \mbox{int} (B)$ if and only if $G \in \Omega_{n+1} (B)$, because by construction,  $G \subset \interior (S(D,B)) \subset \interior(B)$ for some $B \in {\mathcal A}_n$. Since members of ${\mathcal A}_n$ are pairwise disjoint, we deduce that $\Omega_{n+1}(B) \cap \Omega_{n+1}(B') = \emptyset$ if $B \neq B'$. We conclude that the above union of different subfamilies $\Omega_{n+1}(B)$ is a partition of ${\mathcal A}_{n+1}$, as required.

 Note that
 $$\Omega_{n+1}(B) = \bigcup_{D \in {\mathcal A}_n, D \stackrel{\Phi_n}{\rightarrow}B} \ \ \ \bigcup_{C \in {\mathcal A}_n, B \stackrel{\Phi_n}{\rightarrow}C} \Gamma_{n+1} (D,B,C),$$ where the families in the union are pairwise disjoint and each of them has two different boxes. Therefore, taking into account that ${\mathcal A}_n$ is a family of atoms for $\Phi_n$ (by hypothesis), equality ii) of Definition \ref{definitionAtomsGeneration-n} implies:
 $$\#\Omega_{n+1}(B) = 2 \cdot (\#\{ D \in {\mathcal A}_n, D \stackrel{\Phi_n}{\rightarrow}B\}) \cdot \#\{C \in {\mathcal A}_n, B \stackrel{\Phi_n}{\rightarrow}C\} =
 $$ $$2 \cdot 2^{n} \cdot 2^{n} = 2^{2n+1}.  $$

 d)
 By the construction at the end of Remark \ref{remark S(P,Q), G_i}, $$\Omega_{n+1}(B) = \bigcup_{D \in {\mathcal A}_n, D \stackrel{\Phi_n}{\rightarrow} B}\ \ \Omega_{n+1}(D, B).$$ Besides, $\Omega_{n+1}(D,B) \cap \Omega_{n+1}(D', B) = \emptyset$ if $D \neq D'$ in ${\mathcal A}_n$, since different atoms of generation $n$ are pairwise disjoint, and $G \in \Omega_{n+1}(D, B) $ implies $G \subset \Phi_n(D)$ which is disjoint with $\Phi_n(D')$ since $\Phi_n$ is an embedding.

 By the construction in Remark \ref{remark S(P,Q), G_i}, $$\Gamma_{n+1}(D,C,B) = \{G_0(D,C,B), \ G_1(D,C,B)\},$$ where the two boxes $G_{\cdot}(D,C,B)$ inside the family $\Gamma_{n+1}(D,C,B)$ are disjoint, hence different. Thus the cardinality of $\Gamma_{n+1}(D,C,B)$ is 2.

 Also,
 $\Omega_{n+1} (D,B) = \bigcup_{C \in {\mathcal A}_n, B \stackrel{\Phi_n}{\rightarrow} C}\ \ \Gamma_{n+1}(D, B, C)$. Besides,   $$\Gamma_{n+1}(D,B,C) \cap\Gamma_{n+1}(D,B,C')  = \emptyset $$ if $C \neq C'$ in ${\mathcal A}_n$, because  two different atoms of generation $n$ are pairwise disjoint and $G \in \Gamma_{n+1}(D,B,C)$ implies $ G \subset \Phi_{n}^{-1}(C)$.

 From the above assertions and from equality (ii) of the definition of atoms of generation $n$, we deduce that
 $$\#\Omega_{n+1} (D,B) = 2 \cdot (\#    \{C \in {\mathcal A}_n\colon B \stackrel{\Phi_n}{\rightarrow} C\}) = 2 \cdot 2^{n} = 2^{n+1}.$$

Finally, for all $G \in \Omega_{n+1}(B)$  there exists (unique) $D \in {\mathcal A}_n$ such that $G \subset S(D, B) \subset \Phi_n(D) = \Phi_{n+1}(D)$. Hence   $D  \stackrel {\Phi_{n+1}} \rightarrow G$ if and only if $G \in \Omega(D,B)$.

e) Above we proved that $\Omega_{n+1}(D,B)$ is partitioned into the   pairwise disjoint subfamilies
$\Gamma_{n+1} (D,B, C)$, where $C $ is such that $(B, C) \in {\mathcal A}_n^{*2}$.

We have also noticed that $\#\Gamma_{n+1}(D,B,C)= 2$. Finally, by the construction of Remark \ref{remark S(P,Q), G_i}, for all $G \in \Omega_{n+1}(D,B)$ there exists $C \in {\mathcal A}_{n}$ such that $G \in  S(D,B) \cap \Phi^{-1}_n(S(B,C) $.
Therefore $\Phi_{n+1}(G) \subset \Phi_{n+1}(\Phi^{-1}_n(S(B,C)) $.
This latter set   coincides with $\Phi_{n}(\Phi^{-1}_n(S(B,C))$ because, by hypothesis,  $\Phi_n$ and $\Phi_{n+1}$ are embeddings and coincide outside the interiors of all the sets $ \Phi^{-1}_n(S(B,C)$. We deduce that $$\Phi_{n+1}(G) \subset   \Phi^{-1}_n(S(B,C)) \subset \Phi_n (\Phi_{n}^{-1} (S(B,C)))$$ $$ \subset S(B,C) \subset \Phi_n(B) \cap C \subset C .$$
Thus, the interior of $\Phi_{n+1}(G)$, which is nonempty because  $G$ is a box and $\Phi_{n+1} $ is an embedding, is contained in the interior of $C \in {\mathcal A}_n$. Since members of ${\mathcal A}_n$ are pairwise disjoint,  we conclude that, for all $G \in \Omega(D,B)$, $G \in \Gamma_{n+1}(D,B,C)$ if and only if $G \stackrel {\Phi_{n+1}}  \rightarrow C$, as required.

 f) If $G \in \Gamma_{n+1}(D, B, C)$ then $G \subset   (S(D,B)) \subset  \Phi_{n}(D) \cap B$. Therefore $$ \Phi_{n+1}(G) \subset \Phi_{n+1}(B)= \Phi_n(B).$$ Besides, we have proved above that $$\Phi_{n+1}(G) \subset C.$$
 Assume that $\Phi_{n+1}(G) \cap E  \neq \emptyset $ for some $E \in {\mathcal A}_{n+1}$. Since   $E \in \Omega_{n+1}(B', C')$ for some $(B', C') \in {\mathcal A}_n^{2*}$, we have $$E \subset S(B', C') \subset \Phi_n(B') \cap C' .$$
 Since $\Phi_{n+1}(G) \cap E \neq \emptyset$, we deduce that $\Phi_n(B)\cap \Phi_n(B') \cap C \cap C' \neq \emptyset$. 
 Since But distinct atoms of generation $n$ are disjoint and $\Phi_n$ is one to one, we conclude that $B= B', C= C'$ and
   $ E  \in   \Omega_{n+1}(B,C)$.
\end{proof}

\begin{remark} \label{remarkConditionsabc}
 {
 \em Lemma \ref{LemmaConstruction(n+1)atoms}a) immediately implies that for $0 \le j \le n$:\\
 - the families ${\mathcal A}_j^{2*}$ and ${\mathcal A}_j^{3*}$ for $\Phi_n$ and for $\Phi_{n+1}$, coincide, and \\
 - the members of the same families ${\mathcal A}_j$ are also atoms of respective generations $0,1, \ldots, n$ for $\Phi_{n+1}$.
 }

\em
 Parts b) to e) of Lemma \ref{LemmaConstruction(n+1)atoms} ensure that the family ${\mathcal A}_{n+1}$ of boxes   constructed in Remark \ref{remark S(P,Q), G_i}, satisfy conditions i), a), b), c) and d) of Definition \ref{definitionAtomsGeneration-n} for $\Phi_{n+1}$. Thus, the members of ${\mathcal A}_{n+1}$ are good candidates to be atoms of generation $n+1$ for $\Phi_{n+1}$.

  To  actually obtain   atoms of generation $n+1$ for $\Phi_{n+1}$ we will further modify the map  in the interior of the sets $S(D,B) \cap \Phi_n^{-1} S(B,C)) $ for all $(D,B,C) \in {\mathcal A}_n^{3*}(\Phi_n)$, in such a way that  for   the new embedding $\Phi_{n+1}$   the  boxes of ${\mathcal A}_{n+1}$ also satisfy condition ii) of Definition \ref{definitionAtomsGeneration-n}.
\end{remark}

\begin{lemma}
\label{LemmaPermutation}
Still keeping the notation of Remark \ref{remark S(P,Q), G_i}, let
$\widetilde L_{n+1} \subset D^m$ be a finite set with cardinality  $2^{(n+1)^2} 2^{n+1}$, with a unique point $\widetilde e_i (E)\in \widetilde L_{n+1} $  for each $(i,E) \in \{1,2, \ldots 2^{n+1}\}\times {\mathcal A}_{n+1}$.
Assume that $$\widetilde e_i(E) \in \interior(E)   \ \ \forall \   (i,E) \in \{1,2, \ldots 2^{n+1}\} \times {\mathcal A}_{n+1}.$$
Then, there exists a permutation $\theta: \widetilde L_{n+1} \rightarrow \widetilde L_{n+1}$ such that

\begin{enumerate}[a)]

\item For all $  (i,E) \in \{1,2, \ldots 2^{n+1}\} \times \Gamma_{n+1}(D, B, C) $ for some

$(D, B, C) \in {\mathcal A}_n^{3*}(\Phi_n)$,
$$\theta(\widetilde e_i(E)) = \widetilde e_{i'}(E') $$ for a  unique  $ i' \in \{1, 2, \ldots, 2^{n+1}\} $ and a  unique  $  E' \in \Omega_{n+1}(B, C).$

\item For all $(D,B,C) \in {\mathcal A}_n^{3*}(\Phi_{n})r
$, for all $   E \in \Gamma_{n+1}(D, B, C) $ and for all $F \in \Omega_{n+1}(B, C)$
there exists  unique $$(i, i') \in \{1,2, \ldots 2^{n+1}\}^2$$ such that  $\theta(\widetilde e_i(E)) = \widetilde e_{i'}(F) $.

\item For all $(B,C) \in {\mathcal A}_n^{2*}(\Phi_n)$ $$\theta \Big(\Big \{\widetilde e_i(E)\colon  E\in  \hspace{-0.3cm} \bigcup_{D \in {\mathcal A}_n, \ (D,B) \in {\mathcal A}_n^{2^*}} \ \  \Gamma_{n+1}(D,B,C), \ \ i\in  \{1,2, \ldots 2^{n+1}\}\Big\} \Big) = $$ $$\Big \{\widetilde e_{i'}(F): F \in \Omega_{n+1}(B,C), i'\in \{1,2,, \ldots 2^{n+1}\} \Big \} = \widetilde L_{n+1} \cap S(B,C).$$
\end{enumerate}
\end{lemma}

\begin{proof}

From the construction  of the family ${\mathcal A}_{n+1}$ (see Remark \ref{remark S(P,Q), G_i}), we deduce that  for all $E \in {\mathcal A}_{n+1}$ there exists unique $j \in \{0,1\}$ and unique $(D,B,C) \in {\mathcal A}_{n}^{3*}$ such that $$E= G_j(D,B,C)\in \Gamma_{n+1}(D,B,C) $$ (recall Equality (\ref{eqn300})) and thus
we will write
$$\widetilde e_i(G_j(D,B,C))  = \widetilde e_i(E)$$
 for all $(i,E) \in \{1, 2, \ldots, 2^{n+1}\} \times {\mathcal A}_{n+1}$.

By hypothesis ${\mathcal A}_n$ is the family of  atoms of generation $n$ for $\Phi_n$, thus we can apply the equalities ii) of Definition \ref{definitionAtomsGeneration-n}. So, for each $B \in {\mathcal A}_n$, we can index the different atoms $D \in {\mathcal A}_n$ such that $D \stackrel{\Phi_n}{\rightarrow} B$ as follows:
\begin{equation}
\label{eqn105D-}
\{D \in {\mathcal A}_n \colon D \stackrel{\Phi_n}{\rightarrow} B\} = \{D_1^-(B), D_2^-(B), \ldots D_{2^n}^-(B)\},
\end{equation}
where $D^-_{k_1}(B) \neq D^-_{k_2}(B)$ if $k_1 \neq k_2$ (they are disjoint atoms of generation $n$).

Analogously
\begin{equation}
\label{eqn105C+}
\{C \in {\mathcal A}_n \colon B\stackrel{\Phi_n}{\rightarrow} C\} = \{C_1^+(B), C_2^+(B), \ldots C_{2^n}^+(B)\},
\end{equation}
where $C^+_{l_1}(B) \neq C^+_{l_2}(B)$ if $l_1 \neq l_2$.

Now, we index the distinct points of $\widetilde L_{n+1}$ as follows:
$$\widehat e_{i, j}(k, B, l):= \widetilde e_i(G_j(D, B, C))=\widetilde e_i(G_j(D_k^-(B), B, C_l^+(B))),$$
 $$\mbox{for all } (i,j,B, k, l) \in \{1,2,\ldots, 2^{n+1}\} \times \{0,1\} \times {\mathcal A}_n \times \{1,2, \ldots, 2^n\}^2. $$
Define the following correspondence $\theta: \widetilde L_{n+1} \rightarrow \widetilde L_{n+1}$:
$$\theta (\widehat e_{i,j}(k, B, l)) =\widehat e_{i',j'}(k', B', l'), \mbox{ where} $$
\noindent{$\bullet$}  $ B' := C_l^+(B),$

\noindent{$\bullet$} $k'$ is such that $B= D^-_{k'}(C)$ (such $k'$ exists and is unique because $B \stackrel{\Phi_n}{\rightarrow} C$, using  (\ref{eqn105D-})),

\noindent{$\bullet$} $l' = i  \pmod{2^n},$

\noindent{$\bullet$} $j' = 0$ if $i \leq 2^n$ and $j'= 1$ if $i > 2^n,$

\noindent{$\bullet$} $i' = k + j \cdot 2^n.$

Let us prove that $\theta$ is surjective; hence it is a permutation of the finite set $\widetilde L_{n+1}$.

Let $\widehat e_{i',j'}(k', B', l') \in \widetilde L_{n+1}$ be given, where
 $$(i',j',B', k', l') \in \{1,2,\ldots, 2^{n+1}\} \times \{0,1\} \times {\mathcal A}_n \times \{1,2, \ldots, 2^n\}^2.$$

Construct

\noindent{$\bullet$} $i: = l' + j' \cdot 2^n$. Then
  $l'= i \pmod{2^n}$, $j' = 0$ if $i \leq 2^n$ and $j'=1$ if $i > 2^n$.

  \noindent{$\bullet$} $B := D^{-}_{k'}(B')$. Then $B \stackrel{\Phi_n}{\rightarrow}B'$. So, there exists $l$ such that $B' = C^+_{l}(B)$.

  \noindent {$\bullet$}  $k:= i' \pmod{2^n}$, $j:= 0$ if $i' \leq 2^n$ and $j:=1$ if $i' > 2^n$. Therefore $i' = k + 2^n j$.

  By  the above equalities we have constructed some $\theta^{-1}$ such that $\theta \circ \theta^{-1} $ is the identity map. So, $\theta$ is surjective, hence also one-to-one in the finite set $\widetilde L_{n+1}$, as required.

 Now, let us prove that $\theta$ satisfies assertions a), b), c) of Lemma \ref{LemmaPermutation}.

 a)  Fix $\widetilde e_i(E) \in {\widetilde L_{n+1}}$. By construction $\theta(\widetilde e_i(E)) = \widetilde e_{i'}(E') \subset {{\rm int}}(E')$ for some $(i,E) \in \{1,2, \ldots, 2^{n+1}\} \times {\mathcal A}_{n+1}$. Since members of ${\mathcal A}_{n+1}$ are pairwise disjoint (recall Lemma \ref{LemmaConstruction(n+1)atoms}-b)), the box $E'$ is unique. Besides,  by hypothesis, $\widetilde e_{i'}(E') \neq \widetilde e_{j'}(E') $ if $i' \neq j'$. So, the index $i'$ is also unique. Therefore, to finish the proof of a), it is enough to check that $E ' \in \Omega_{n+1}(B,C)$ if $E \in \Gamma_{n+1}(D,B,C)$.

  By the definition of the family $\Gamma_{n+1}(D,B,C) $ in Remark \ref{remark S(P,Q), G_i}, if $E \in \Gamma_{n+1}(D,B,C)$, there exists $j \in \{0,1\}$ such that $E = G_j(D,B,C)$. Thus, using the notation at the beginning
  $\widetilde e_i(E) = \widetilde e_i (G_j(D,B,C)) = \widehat e_{i,j}(k, B, l)$, where $D= D^-_{k}(B)$ and $C= C^+_{l}(B)$. Then, using the definition of the permutation $\theta$, and the computation of its inverse $\theta^{-1}$, we obtain $\widetilde e_{i'}(E) = \theta (\widetilde e_i(E)) = \widehat e_{i', j'}(k', B', l')$, where
  $$ B' = C_l^+ (B) = C, \ \ \ D' =D^-_{k'}(B') = B. $$

We have proved that $\widetilde e_{i'}(E') = \widetilde e_{i'}(G_{j'} (B,C, C'))$.
Finally, from the definition of the family $\Omega_{n+1}(B,C)$ in Remark \ref{remark S(P,Q), G_i} we conclude that $E' \in \Omega_{n+1}(B,C) $ as asserted in part a).

b) Fix $(D, B, C) \in {\mathcal A}_{n}^{3*}$ and $E \subset \Gamma_{n+1}(D,B,C)$. Then, using the definition of the family
$\Gamma_{n+1}(D,B,C)$ in Remark \ref{remark S(P,Q), G_i}, we have unique
$(j, k, l) \in \{0,1\} \times \{1,2, \ldots, 2^n\}^2$ such that $E= G_j(D,B,C)$,  $D= D_k^-(B)$,  $C= C^+_l(B)$.
Consider the finite set $Z$ of $2^{n+1}$ distinct points $\widetilde e_i(E) = \widehat e_{i,j}(k, B, l)$,
with $j,k,B,l $ fixed as above and $i\in \{1,2, \ldots, 2^{n+1}\}$.
Let $i' =: k + 2^n j$, then
the image of each point in $Z$ by the permutation $\theta$    is $\theta(\widetilde e_i(E)) = \widetilde e _{i'}(G_{j'}(B, C, C') $
(here we use assertion a). Since $k,j$ are fixed, we deduce that there exists
a unique $i'$ such that    all the points of $\theta(Z) $ are of the form $\widetilde e_{i'}(F)$, $F= G_{j'}(B,C,C')$ with $j'\in \{0, 1\}$,
$C' = C^+_{k'}(C), \ k' \in \{1,2, \ldots, 2^{n+1}\}$. We have proved that the permutation $\theta|_Z$ is  equivalent to
$$  i \in \{1,2, \ldots, 2^{n+1}\} \rightarrow (j', k') \in \{0,1\} \times \{1,2, \ldots, 2^n\}$$ such that
$\theta (\widetilde e_i(E)) = \widetilde e_{i'}(G_{j'}(B, C, C^+_{k'}(C) ))$ with $i'$ fixed.

Since $\#\{1,2, \ldots, 2^{n+1} \}= $ $\#( \{0,1\} \times \{1,2, \ldots, 2^n\})$, from the injectiveness of $\theta$  we deduce that $\theta(Z) = \{0,1\} \times \{1,2, \ldots, 2^n\}$.  In other words, for every $F \in \Omega(B,C)$ there exists unique
$i$ such that $\theta(\widetilde e_i(E)) = \widetilde e_{i'}(F)$ (where $i'$ is uniquely defined given $E$). This ends the proof of assertion b).

c) For fixed $(B,C) \in {\mathcal A}_n^{2*}$, denote $$P:= \Big\{\widetilde e_i(E) \colon  \ \ E \in \bigcup_{ D \in {\mathcal A}_n, D \stackrel{\Phi_n}{\rightarrow} B} \ \    \Gamma_{n+1}(D,B,C), \ \ i\in\{1, 2, \ldots, 2^{n+1}\}\Big \},$$
$$Q:=  \{\widetilde e_{i'}(F) \colon  \ \ F\in   \Omega_{n+1}( B,C), \ \ i'\in\{1, 2, \ldots, 2^{n+1}\} \} \subset \widetilde L_{n+1}. $$
Applying assertion a) we deduce that $\theta(P) \subset Q$. So, to prove that $\theta(P)=Q$ it is enough to prove that $\#P = \#Q$. In fact, applying   Lemma \ref{LemmaConstruction(n+1)atoms} for the family of boxes ${\mathcal A}_{n+1}$ for the family of atoms ${\mathcal A}_n$, we obtain

\begin{align*}
\#P &= 2^{n+1} \cdot (\#\Gamma_{n+1}(D,B,C)) \cdot (\#\{D \in {\mathcal A_n} \colon D \stackrel{\Phi_n}{\rightarrow}B \}) = 2^{n+1} \cdot 2 \cdot 2^n \\
\#Q & = 2^{n+1} \cdot (\#\Omega_{n+1} )(B,C))= 2^{n+1}\cdot  2^{n+1},
\end{align*}
which proves that $\#P = \#Q$ and thus that  $\theta(P)=Q$.

Finally, let us prove that $Q=\widetilde L_{n+1} \cap S(B,C).$
On the one hand, if $F \in \Omega_{n+1}(B,C)$, then $F = G_j(B,C,C')$ for some $(j, C')$. Applying the construction of the boxes of ${\mathcal A}_{n+1}$ in Remark \ref{remark S(P,Q), G_i}, we obtain     $  F\subset S(B,C) $, hence  $\widetilde e_{i'} (F) \in \widetilde L_{n+1} \cap {{\rm int}}(F) \subset  \widetilde L_{n+1} \cap S(B,C) $. This proves that $Q \subset \widetilde L_{n+1} \cap S(B,C)$.

On the other hand, if $\widetilde e_{i'}(F) \in \widetilde L_{n+1} \cup S(B,C) $, then  $F \in {\mathcal A}_{n+1}$. We obtain $F = G_j(D', B', C') \subset S(D', B')$ for some $(D', B', C') \in {\mathcal A}_n^{3*}$. Since $S (D',B') \subset \Phi_n(D') \cap B'$ and $S(B,C) \subset \Phi_n(B) \cap C$, we obtain $S(D',B') \cap S(B,C) = \emptyset$ if $(D', B') \neq (B,C)$. But $\widetilde e_{i'} (F) \in {{\rm int}}(F) \cap S(B,C)\subset S(D', B') \cap S(B,C) $. We conclude that $(D',B')= (B,C)$, thus $F = G_j(B, C, C')\in \Omega_{n+1}(B,C)$,  hence $\widetilde e_{i'}(F) \in Q$. We have proved that $ \widetilde L_{n+1} \cap S(B,C) \subset Q$. 
\end{proof}

\begin{lemma}
\label{LemmaSON(n+1)-atomos}
Assume the hypothesis  of Lemmas \ref{LemmaConstruction(n+1)atoms} and \ref{LemmaPermutation}. Let $\Phi_{n+1} \in {{\rm Emb}}(D^m)$ be such that, besides the conditions in the hypothesis of Lemma \ref{LemmaConstruction(n+1)atoms}, satisfies the following:
$$ \Phi_{n+1}(\widetilde e)= \theta(\widetilde e) \ \ \forall \ \widetilde e \in \widetilde L_{n+1},$$
where $\theta$ is the permutation of $\widetilde L_{n+1}$ constructed in Lemma \ref{LemmaPermutation}.
Then,
\begin{enumerate}[a)]
\item  ${\mathcal A}_0, {\mathcal A}_1, \ldots,  {\mathcal A}_{n+1} $ are collections of atoms up to generation $n+1$ for $\Phi_{n+1}$.
\item For each $(E,F) \in {\mathcal A}_{n+1}^2$ such that  $E\stackrel{\Phi_{n+1}}{\rightarrow}F$, there exists exactly one  point  $\widetilde e_i(E)  \in \widetilde L_{n+1} \cap {{\rm int}}(E) $, and exactly one point $\widetilde e_{i'}(F) \in \widetilde L_{n+1} \cap {{\rm int}}(F)$,  such that $$  \Phi_{n+1}(\widetilde e_i(E)) =  \widetilde e_{i'}(F).$$
\end{enumerate}
\end{lemma}

\begin{proof}
a) By Remark \ref{remarkConditionsabc}, it is enough to establish the truth of condition ii) of Definition \ref{definitionAtomsGeneration-n} with $n+1$ instead of $n$.

Take $E \in {\mathcal A}_{n+1}$. There exists $(D,B,C) \in {\mathcal A}_{n}^{3^*}$ such that $E \in \Gamma_{n+1}(D,B,C)$. Take $F \in \Omega_{n+1}(B,C)$. Applying Lemma \ref{LemmaPermutation}-b), there exists unique $(i, i')$  such that $\theta(\widetilde e_i(E))  = \widetilde e_{i'}(F)$. Therefore $$\Phi_{n+1}(\widetilde e_i(E)) )= \widetilde e_{i'}(F).$$
 Since $\widetilde e_i(E) \in {{\rm int}}(E)$ and $\widetilde e_{i'}(F) \in {{\rm int}}(F) $, we conclude that $\Phi_{n+1}(E)  \cap {{\rm int}}(F) \neq \emptyset $, namely, $E \stackrel{\Phi_{n+1}}{\rightarrow}F$. We have proved that
 $$E \stackrel{\Phi_{n+1}}{\rightarrow}F \ \ \forall E \in \Gamma_{n+1}(D,B,C), \ \forall F\in \Omega_{n+1}(B,C).$$

 Combining with the assertion g) of Lemma \ref{LemmaConstruction(n+1)atoms}, we deduce that, for all $(D,B,C) \in {\mathcal A}_n^{3*}$,   for all $E \in \Gamma_{n+1}(D,B,C)$, for all $F \in {\mathcal A}_{n+1}$
 \begin{equation}\label{eqn2000}
  E \stackrel{\Phi_{n+1}}{\rightarrow}F  {\mbox{ if and only if }} F\in \Omega_{n+1}(B,C).
 \end{equation}

Given $E \in {\mathcal A}_n$, let us count how many $F \in {\mathcal A}_n$ satisfy $E \stackrel{\Phi_{n+1}}{\rightarrow}F$. Given $E$, there exists unique $(D,B,C) \in {\mathcal A}_n^{3*}$ such that $E \in \Gamma_{n+1} (D,B,C)$. Applying   (\ref{eqn2000}) and assertion d) of Lemma \ref{LemmaConstruction(n+1)atoms}, we deduce
 $$ \# \{F \in {\mathcal A}_n \colon E \stackrel{\Phi_{n+1}}{\rightarrow}F \} = \Omega_{n+1}(B,C)= 2^{n+1}.$$

 Finally, given $F \in {\mathcal A}_n$, let us count how many $E \in {\mathcal A}_n$ satisfy $E \stackrel{\Phi_{n+1}}{\rightarrow}F$. Given $F$,   there exists unique $(B,C) \in {\mathcal A}_n^{2*}$ such that $F \in \Omega_{n+1} (B,C)$. Applying   (\ref{eqn2000}), assertion e) of Lemma \ref{LemmaConstruction(n+1)atoms}, and assertion ii) of Definition \ref{definitionAtomsGeneration-n} for the atoms of generation $n$ (for $\Phi_n$ and for $\Phi_{n+1}$), we obtain:
 $$ \# \{E \in {\mathcal A}_n \colon E \stackrel{\Phi_{n+1}}{\rightarrow}F \} = $$ $$ \# \{E \in {\mathcal A}_{n+1} \colon \exists D \in {\mathcal A}_n \mbox{ such that }D \stackrel{\Phi_{n+1}}{\rightarrow} B , E \in \Gamma_{n+1}(D,B,C)\}   =$$ $$ (\# \{D \in {\mathcal A}_n \colon D \stackrel{\Phi_{n+1}}{\rightarrow} B\})  \cdot (\#\Gamma_{n+1}(D,B,C)) = 2^n \cdot 2=  2^{n+1}.$$
 We have proved that the boxes of ${\mathcal A}_{n+1}$ satisfy equalities ii) of Definition \ref{definitionAtomsGeneration-n} for $\Phi_{n+1}$. The proof of assertion a) is complete

b)  Take $(D,B,C) \in {\mathcal A}_n^{3*}$ and $E \in \Gamma_{n+1}(D,B,C)$. Take $F \in {\mathcal A}_{n+1}$. From Remark \ref{remarkAtomsGeneration-n} (putting $n+1$ instead of $n$), we know that
$ \Phi_{n+1}(E)\cap F \neq \emptyset$ if and only if $F \in \Omega_{n+1}(B,C)$.
Applying Lemma \ref{LemmaPermutation}-b) there exists a unique $(i, i') \in \{1,2, \ldots, 2^{n+1}\}^2$ such that
 $\Phi_{n+1}(\widetilde e_i(E)) = \theta(\widetilde e_i(E)) = \widetilde e_{i'}(F)$.
 The proof of part b) is complete.
\end{proof}

\begin{lemma}
\label{LemmaHomeosEspecificandoFinitosPuntos2}
Let $\psi \in {{\rm Emb}}(D^m)$, $r \geq 1$, $P_1, P_2, \ldots P_r \subset D^m$ be pairwise disjoint boxes, and
 $Q_j :=\psi(P_j)$ for all $j\in \{1, 2, \ldots, r\}$. For $k \geq 1$ and  $j \in \{1,2, \ldots, r\}$, let
$p_{1,j}, \ldots, p_{k,j} \in {{\rm int}} (P_j)$ be distinct points and $q_{1,j}, \ldots, q_{k,j} \in {{\rm int}} (Q_j)$   also
be distinct points.
Then, there exists a $\psi^* \in {{\rm Emb}}(D^m)$ such that
$$\psi^*(x) = \psi(x) \ \ \forall \ x \not \in \bigcup_{j=1}^{r}{{\rm int}} (P_j) \mbox{ and }$$
$$\psi^*(p_{i,j}) = q_{i,j} \ \ \forall \ (i,j) \in \{1, \ldots, k\}\times \{1, 2, \ldots, r\}.$$
\end{lemma}
\begin{proof} It is straightforward.
\end{proof}

\begin{proof} {\em of Lemma \ref{LemmaConstruccionModelPsifisPhi}}
We  divide the construction of $\psi$ and $\Phi \in {\mathcal H}$ into several steps:

\noindent {\bf Step 1. Construction of the atom  of generation 0. }
Since $f(D^m) \subset {{\rm int}}(D^m)$, there exists a  box $A_0 \subset {{\rm int}}(D^m)$ such that $f(D^m) \subset {{\rm int}}(A_0)$.
  The box $A_0$ is the {\em atom of generation 0 } for the  embedding $\Phi_0 := f$ which satisfies $\Phi_0= \psi_0 \circ f$ where $\psi_0$  is the identity map.
Applying the Brower Fixed Point Theorem, there exists a point
 $ e_0 \in  {{\rm int}} (\Phi_0(A_0)) $
such that   $\Phi_0(e_0)= e_0$.
Define $S(A_0,A_0)$ to be the connected component of $A_0 \cap \Phi_0(A_0)$ containing $e_0$.

(The notation above is   too complicated, because simply $A_0 \cap \Phi_0(A_0) = \Phi(A_0)$, which is connected. But we introduced that complicated notation to make obvious that the inductive hypothesis that we will assume in the following step, is satisfied for $n=0$.)

\noindent {\bf Step 2.  Construction of the atoms of generation {\em n+1}}.
Inductively assume that we have constructed  families
${\mathcal A}_0, {\mathcal A}_1, $ $ \ldots, {\mathcal A}_{n}$ of atoms up to generation $n$ for $\Phi_n  = \psi_n \circ f$, where $\psi_n\in {{\rm Hom}}(D^m)$, satisfying:

\begin{enumerate}[I)]
\item $\psi_n|  _{\partial D^m}$ is the identity map,
\item $\label{eqn19-n}  \max_{B \in {\mathcal A}_i} \max\{{{\rm diam}} (B), {{\rm diam}}(f(B))\}< \frac{1}{2^i} \ \ \forall i \in \
\{0,1,\ldots n\};$

\item  $\Phi_i(x) = \Phi_{i-1}(x), \qquad \forall \ x \in D^m \setminus \bigcup_{B \in {\mathcal A}_{i-1}} B, \ \forall  i\in \{1,\ldots, n\};$

\item for all  $(D,B,C) \in  {\mathcal A}_n ^{3*}(\Phi_n)$ there exists a point  $e (D,B,C)  $   such that
   $$  L_n:= \{e(D,B,C) \colon (D,B,C) \in {\mathcal A}_n^{3*}\}$$ is $\Phi_n$-invariant, and
\begin{equation} \label{eqn34}  e(D,B,C)  \in {{\rm int}}\big (S(D,B) \cap \Phi_n^{-1}(S(B,C)) \big),\end{equation} where $S(D,B)$ and $S(B,C)$ are  (adequately chosen)  connected components of $B \cap \Phi_n(D)$ and of $C \cap \Phi_n(B)$ respectively. (Recall the notation in Remark \ref{remark S(P,Q), G_i}).

\vspace{.3cm}

\noindent  Note that the sets $S(B,C)$   and $S(B',C')$  are disjoint if $(B,C) \ne (B', C')$, because two different atoms of generation $n$ for $\Phi_n$ are  disjoint (recall Definition \ref{definitionAtomsGeneration-n}) and $\Phi_n$ is one to one.
\end{enumerate}

Let us construct the family ${\mathcal A}_{n+1}$ of boxes,  candidates to be atoms of generation $n+1$ for a new embedding   $\Phi_{n+1}$ (to be constructed as in Remark \ref{remark S(P,Q), G_i}), and let us construct the homeomorphism $\psi_{n+1}$ such that $\Phi_{n+1}= \psi_{n+1} \circ f$.

 First, for each   $(B,C) \in {\mathcal A}_n^{2^*}(\Phi_n)$, we choose a box $R(D,B)$ such that
\begin{equation}
 \label{eqn31-n}e(D,B,C) \in {{\rm int}}(R(B,C)), \ \ R(B,C) \subset {{\rm int}}\big (\Phi_n^{-1}(S (B,C) ) \big) \end{equation} $$\forall \ D \in {\mathcal A}_n \mbox{ such that } D {\stackrel{\Phi_n}{\rightarrow}} B.$$
  Note that such boxes $R(\cdot, \cdot)$ are pairwise disjoint, because they are contained in pairwise disjoint sets.

Recall that $e(D,B,C) \in L_n$ and the set $L_n$ is $\Phi_n$-invariant. Consider assertions (\ref{eqn34}) and (\ref{eqn31-n}). Thus,

 $$e(D,B,C)  \in {{\rm int}} \Big(R(B,C) \cap \Phi_n (R(D,B))\Big) \neq \emptyset.$$

Next, for each $(D,B,C) \in {\mathcal A}_n^{3*}$ we choose two pairwise disjoint  boxes, $G_0(D,B,C)$ and $ G_1(D,B,C)$,
contained in  the interior of $   R(B,C) \cap \Phi_n (R(D,B)) $, satisfying
\begin{equation} \label{eqn26a}\max\{{{\rm diam}}(G_i(D,B,C)) ,  {{\rm diam}}(f(G_i(D,B,C))) \} < \frac{1}{2^{n+1}}\end{equation} for $i= 0,1$.
Now, we use the notation of Remark \ref{remark S(P,Q), G_i},  to construct
   the family  ${\mathcal A}_{n+1}$ of all the boxes $G_i(D,B,C)$.    The boxes of the family ${\mathcal A}_{n+1}$ will be the $(n+1)$-atoms of two new embeddings $\widetilde \Phi_{n+1}$ and $\Phi_{n+1}$ that we will construct as follows.

First, in the interior of each box $E \in {\mathcal A}_{n+1}$ we choose $2^{n+1}$ distinct points $\widetilde e_i(E), i= 1, 2 \ldots, 2^{n+1}$, and denote $$\widetilde L_{n+1} := \{\widetilde e_i(E)\colon E \in {\mathcal A}_{n+1}, \ 1 \leq i \leq 2^{n+1}\}.$$ Second, we build a permutation $ \widetilde \theta  $ of $\widetilde L_{n+1}$  satisfying the properties of Lemma \ref{LemmaPermutation}.

Third, applying Lemma \ref{LemmaHomeosEspecificandoFinitosPuntos2}, we construct $\widetilde \psi_{n+1} \in {{\rm Hom}}(D^m) $  satisfying the following   constraints.

\noindent{(a)} For all $(B,C) \in {\mathcal A}_n^{2^*}(\Phi_n)$:
$$\widetilde \psi_{n+1}|_{f(R(B,C))}: f(R(B,C))  \rightarrow \psi_n \circ f (R(B,C)) = \Phi_n(R(B,C,)),$$
 \noindent{(b)}    $$ \widetilde \psi_{n+1}  (x) = \psi_n(x) , $$ $$\forall  x \not \in  \bigcup_{(B,C)\in {\mathcal A}_n^{2*}  \mbox{ for }\Phi_n } f ( R(B,C)), $$

  \noindent{(c)} $$ \widetilde \psi_{n+1}(f(\widetilde e)) = \widetilde \theta(\widetilde e), \qquad \forall \ \widetilde e \in \widetilde L_{n+1}.$$
 To prove the existence of such a homeomorphism  $\widetilde \psi_{n+1}$  we must verify the hypotheses of Lemma \ref{LemmaHomeosEspecificandoFinitosPuntos2}.
 On the one hand, the boxes $R(B, C)$ where $(B,C) \in {\mathcal A}_n^{2*}$ are pairwise disjoint. So  their  images  by
 the embedding $f$ are also  pairwise disjoint boxes.
 On the other hand, for each $(B,C) \in {\mathcal A}_n^{2*}(\Phi_n)$,
 the finite set $$\{f(\widetilde e): \ \ \ \widetilde e \in \widetilde L_{n+1}\,  \cap \,  {{\rm int}}(  R(B,C))\} = $$ $$\{f(\widetilde e): \ \ \ \widetilde e \in \widetilde L_{n+1}\,  \cap \,  {{\rm int}}(  \Phi_n^{-1}(S(B,C)))\}$$ is contained in  the interior of  $f   (R(B,C)).$
     Besides, it coincides with $$ \{f(\widetilde e_i(E)) \colon E \in \Gamma_{n+1}(D, B, C)   \mbox{ for some } D \in {\mathcal A}_n, \  i = 1, \ldots, 2^{n+1}\}  $$
      (recall   Lemma \ref{lemma2.10-f}).
     So, the  image  by the permutation    $\widetilde \theta $ of
         such points $\widetilde e_{ \cdot } (\cdot)$ is $$\{\widetilde \theta(\widetilde e_i(E))\colon E \in \Gamma_{n+1}(D, B, C) \ \ \mbox{ for some } D \in {\mathcal A}_n, \ i=1,2, \ldots 2^{n+1}\}$$
Applying  Lemma   \ref{LemmaPermutation}-c), the latter set is
$$ \{\widetilde e_k(F) \colon F \in \Omega_{n+1}(B,C), k = 1,2, \ldots 2^{n+1} \} = \widetilde L_{n+1} \cap S(B,C), $$ which is contained  in the interior of $\Phi_n(R(B,C)) = \widetilde \psi_n (f(R(B,C)))$

We have proved that the points $f(\widetilde e(\cdot))$   are contained in the interior of the boxes
$f(R(\cdot, \cdot))$, and that their required images $\theta( \widetilde e(\cdot))$ by the homeomorphism $\widetilde \psi_{n+1}$ (to be constructed), are in the interior of the images by $\widetilde \psi_n$ of those boxes.
So, the hypothesis of Lemma \ref{LemmaHomeosEspecificandoFinitosPuntos2} is satisfied.

 We construct  $$\widetilde \Phi_{n+1} :=\widetilde  \psi_{n+1} \circ f.$$
 Since $$\widetilde \Phi_{n+1}(x) = \widetilde \psi_{n+1} \circ f = \widetilde \psi_n \circ f= \Phi_n (x) \ \   \forall    x \not  \in $$ $$ \bigcup_{ (B, C) \in  {\mathcal A}_n^{2*} \mbox {\footnotesize{ for }}\Phi_n} \mbox{int}(R(B,C)) \subset \bigcup_{ (B, C) \in  {\mathcal A}_n^{2*} \mbox {\footnotesize{ for }}\Phi_n} \mbox{int} ( \Phi_n^{-1}(S(B,C))),$$  the hypothesis of Lemma \ref{LemmaConstruction(n+1)atoms} is satisfied. Therefore the same atoms  up to     generation $n$ for $\Phi_n$ are still atoms up to  generation   $n$ for $\widetilde \Phi_{n+1}$.
     But moreover,  applying Lemma \ref{LemmaSON(n+1)-atomos}-a), the  boxes of the new family ${\mathcal A}_{n+1}$  are now  $(n + 1)$-atoms for $\widetilde \Phi_{n+1}$.

\noindent {\bf Step 3.  Construction of $\Phi_{n+1}$ and $\psi_{n+1}$.}
To argue by induction, we will not use the embedding $\widetilde \Phi_{n+1}$ and the homeomorphism $\widetilde \psi_{n+1}$, even if $\widetilde \Phi_{n+1} = \widetilde \psi_{n+1} \circ f$ already has families ${\mathcal A}_0, \ldots, {\mathcal A}_n, {\mathcal A}_{n+1}$ of atoms up to generation $n+1$, as required. Rather, we
 need to modify them  to obtain a new embedding $\Phi_{n+1}$ and a new homeomorphism $\psi_{n+1}$ such that the inductive hypothesis (IV) and Assertion (\ref{eqn34}) also holds for
 $n+1$ instead of $n$.
We will modify $\widetilde \psi_{n+1}$   only in the interiors of the boxes $f(G)$ for all the atoms $G \in {\mathcal A}_{n+1}$ for $\widetilde \Phi_{n+1}$, we will
construct a new homeomorphism  $\psi_{n+1}$ such that  $\Phi_{n+1} := \psi_{n+1} \circ f$  has the same atoms up to generation $n+1$
of $\widetilde \Phi_{n+1}$ (see the proof of part a) of Lemma \ref{LemmaConstruction(n+1)atoms}), and besides satisfies the inductive hypothesis (IV) with $n+1$ instead of $n$.

From the above construction of $\widetilde \psi_{n+1}$ and $\widetilde \Phi_{n+1}$, and from Lemma \ref{LemmaSON(n+1)-atomos}-b), we know that for each $(G,E) \in {\mathcal A}_{n+1}^{2*}$ for $\widetilde \Phi_{n+1}$, there exists a unique point  $\widetilde e_i(G)  \in {{\rm int}}(G) $, and a unique point $\widetilde e_k(E)$,  such that $$\widetilde \Phi_{n+1}(\widetilde e_i(G)) = \widetilde \psi_{n+1}  \circ f (\widetilde e_i(G)) = \widetilde\theta(\widetilde e_i(G))= \widetilde e_k(E) \in {{\rm int}}(E).$$  Therefore
$$ \widetilde e_k(E) \in {{\rm int}} ( E \cap \widetilde  \Phi _{n+1}(G) ).$$

\noindent Denote by $$S(G,E)$$ the connected component of  $ E\,  \cap \, \widetilde \Phi_{n+1}(G)  $ that contains the point   $\widetilde e_k(E) $.
 Choose $2^{n+1}$ distinct points $$e_i(G, E)  \in {{\rm int}}(S(G,E)), \ \  i= 1, \ldots, 2^{n+1} $$
 and a  permutation $\theta $   of the finite set \begin{equation}\label{eqnLn+1}   L_{n+1}:= \{e_i(G,E): \ \ (G,E) \in  {\mathcal A}_{n+1}^{2*} \mbox{ for } \widetilde \Phi_{n+1}, \ \  i= 1, \ldots, 2^{n+1}\}\end{equation}   such that
  for each fixed  $(G,E,F) \in {\mathcal A}_{n+1}^{3*}$ for $\widetilde \Phi_{n+1} $, there exists a unique point $e_i(G,E)$, and a unique point $e_k(E,F)$, satisfying
  \begin{equation}
  \label{eqn2001}
  \theta (e_i(G,E)) = e_k(E,F).\end{equation}
  The proof of the existence of such permutation is similar (but simpler) than the the proof of Lemma \ref{LemmaPermutation}.

Applying Lemma \ref{LemmaHomeosEspecificandoFinitosPuntos2},   construct a homeomorphism $$ \psi_{n+1}\in {{\rm Hom}}(D^m)$$
such that $$\psi_{n+1}|_{f(G)}:  f(G)\rightarrow  \widetilde \psi_{n+1}(f(G)) = \widetilde \Phi_{n+1}(G) \ \ \forall \ G \in {\mathcal A}_{n+1} \mbox{ for } \widetilde \Phi_{n+1},$$
$$ \psi_{n+1}(x)  =  \widetilde \psi_{n+1}(x) \ \ \forall \ x \not \in \bigcup_{G \in {\mathcal A}_{n+1}} f(G), $$
\begin{equation}
\label{eqn2002}
   \psi_{n+1} (f(e_i(G,E)) = \theta  (e_i(G,E)) \end{equation} $$ \forall \ (E, G)\in {\mathcal A}_{n+1}^2 \mbox{ such that } G \stackrel{\widetilde \Phi_{n+1}} \rightarrow E, \qquad \forall \ i= 1, \ldots, 2^{n+1}, $$
and  extend $  \psi_ {n+1}$ to the whole box $D^m$ by defining  $\psi_{n+1}(x) = \widetilde \psi_{n+1}(x), $   $\forall \ x \in {D^m \setminus \bigcup_{G \in {\mathcal A}_{n+1}} f(G)}.$
In particular $$  \psi_{n+1}|_{\partial D^m} = \widetilde \psi_{n+1}|_{\partial D^m} = \mbox{id}|_{\partial D^m}.$$
Define \begin{equation}
\label{eqn2003}
  \Phi_{n+1} :=  \psi_{n+1} \circ f.\end{equation} As said above, the property that $\Phi_{n+1} $ coincides with $   \widetilde \Phi_{n+1}$ outside all the atoms of ${\mathcal A}_{n+1}$ for $\widetilde \Phi_{n+1}$ implies that the boxes of the families ${\mathcal A}_0, \ldots, {\mathcal A}_{n+1}$, which are the family of atoms up to generation $n$ for $\widetilde \Phi_{n+1}$, are also atoms   up to generation $n+1$ for $\Phi_{n+1}$. But now, due to equalities (\ref{eqn2001}), (\ref{eqn2002}) and (\ref{eqn2003}), they have the following additional property:
there exists a one-to-one correspondence between  the 3-tuples
 $(G, E, F)  \in {\mathcal A}_{n+1}^{3*}$  (for $\widetilde \Phi_{n+1}$ and also for $ \Phi_{n+1}$ )  and the  points of the set  $L_{n+1}$ of Equality (\ref{eqnLn+1}), such that
\begin{equation} \label{eqn30z}e(G,E,F) := e_i(G, E) \in {{\rm int}}\big (S(G,E) \cap \Phi_{n+1}^{-1}(S(E,F)) \big)., \end{equation}
Recall that   $S(G,E)$ and $S(E,F)$ are the  connected components of $E \cap \Phi_n(G)$ and of
$F \cap \Phi_n(E)$  respectively, that were chosen after $\widetilde \Phi_{n+1}$ was constructed.

 Besides, by construction, the finite set $L_{n+1}$ is $\Phi_{n+1}$-invariant. In fact, $\Phi_{n+1}(L_{n+1}) = \psi_{n+1} (f(L_{n+1})) = \theta(L_{n+1}) = L_{n+1}$.
Therefore, the inductive hypothesis I), II), III) and IV)   holds for $n+1$  and the inductive construction is complete.

\noindent{\bf Step 4. The limit homeomorphisms. }
From the above construction we have:
 $$\psi_{n+1}(x) = \widetilde \psi_{n+1}(x) = \psi_n(x) \mbox{ if }  x \not \in  \bigcup_{ B,C } \psi_n^{-1}  ( R(B,C)) \subset \bigcup_{B} f(B) $$

 $$\psi_{n+1}\circ \psi_n^{-1}(R(B,C))  = \widetilde \psi_{n+1} \circ \psi_n^{-1} (R(B,C))  =$$ $$ \psi_n \circ \psi_n^{-1}( R(B,C))  =  R(B,C) \subset C.$$
Therefore,
$$\mbox{dist}(\psi^{-1}_{n+1}(x),  \psi^{-1}_n(x)) \leq   \max_{B \in {\mathcal A}_n} {{\rm diam}}(f(B)) < \frac{1}{2^{n}}, \qquad \forall \ x \in D^m; $$
$$\mbox{dist}(\psi_{n+1}(x),  \psi_n(x)) \leq   \max_{C \in {\mathcal A}_n} {{\rm diam}}(C) < \frac{1}{2^n}, \qquad \forall \ x \in D^m, $$
\begin{equation} \label{eqn27a} \| \psi_{n+1} - \psi_n\|_{{{\rm Hom}}}   < \frac{1}{2^n}.\end{equation}

From Inequality (\ref{eqn27a}) we deduce that the sequence $\psi_n $ is Cauchy in ${{\rm Hom}}(D^m)$. Therefore, it converges to a homeomorphism   $\psi $. Moreover, by construction  $\psi_n|_{\partial D^m} =     \mbox{id}|_{\partial D^m}$ for all $n \geq 1$. Then  $\psi |_{\partial D^m} =     \mbox{id}|_{\partial D^m}.$

The convergence of $\psi_n$ to $\psi$ in ${{\rm Hom}}(D^m)$ implies that
   $\Phi_n =   \psi_n \circ f  \in {{\rm Emb}}(D^m)  $  converges  to $\Phi = \psi \circ f \in {{\rm Emb}}(D^m)$ as $n \rightarrow + \infty$.  Since $f(D^m) \subset {{\rm int}}(D^m)$ and $\psi \in {{\rm Hom}}(D^m)$, we deduce that  $\Phi(D^m) \subset {{\rm int}}(D^m).$
Moreover, by construction ${\mathcal A}_0, {\mathcal A}_1, \ldots, {\mathcal A}_n$ are families of atoms up to generation $n$ for $\Phi_n$, and  $\Phi_j(x)= \Phi_n (x) $ for all $ x \in D^m \setminus \bigcup_{B \in {\mathcal A}_n} B$ and for all $ j \geq n .$   Since  $\lim_j \Phi_j = \Phi$,   the boxes of the family   $ {\mathcal A}_n $ are $n$-atoms for $\Phi$ for all $n \geq 0$.
Finally, from Inequality (\ref{eqn19-n}, the diameters of the $n$-atoms converge uniformly to zero as $n \rightarrow + \infty$. Thus $\Phi $ is a model according to Definition \ref{DefinitionModel}.
\end{proof}

\section{Infinite metric entropy and mixing property of the models.} \label{SectionMainLemma}
The purpose of this section is to prove the following Lemma.
\begin{lemma} 
\label{LemmaMain}
Let ${\mathcal H} \subset C^0(D^m)$ be a family of models with $m \ge 2$. \em
For each $\Phi \in {\mathcal H}$ \em  \em  there exists a $\Phi$-invariant mixing (hence ergodic) measure $\nu$  supported on a $\Phi$-invariant Cantor set $\Lambda \subset D^m $ such that
$  h_{\nu}(\Phi) = + \infty.$
\end{lemma}

Throughout this section we assume $m \ge 2$ and we suppose there is a  given $\Phi \in {\mathcal H}$, with a given sequence of families ${\mathcal A}_n \ (n \ge 0)$ of atoms of generations $n \ge 0$ respectively for $\Phi$. When we refer to the atoms of generation $n$, we omit writing $\Phi$ and the families of atoms of previous generation, which are the previously given map and families.

\begin{remark} \em
\label{RemarkMainLemma} Lemma \ref{LemmaMain} holds, in particular, for ${\mathcal H} \cap {{\rm Emb}}(D^m)$.
\end{remark}

To prove Lemma \ref{LemmaMain} we need to   define   the paths of atoms and to discuss their properties. We also need to define the invariant Cantor set $\Lambda$ that will support the measure $\nu$  and prove some of its topological dynamical properties.

\begin{definition} {\bf (Paths of atoms)}
  \label{definitionPathOfAtoms} \em

  Let $\Phi \in {\mathcal H} \subset C^0(D^m)$, $l \geq 2$ and   let $(A_1, A_2, \ldots, A_l)$ be a $l$-tuple of atoms for $\Phi$ of the same generation $n$, such that
$$A_i \rah A_{i+1}, \qquad \forall \ i \in \{1,2, \ldots, l-1\}.$$
We call $(A_1, A_2, \ldots, A_l)$ an \em $l$-path of $n$-atoms from $A_1$ to $A_l$. \em
Let  ${\mathcal A}_n ^{l*} $ denote  the family of all the $l$-paths of atoms of generation~$l$.
\end{definition}

\begin{lemma}
\label{lemmaPathsOfAtoms}
For all $n \geq 1$, for all $l \geq 2 n$, and for all   $ A_1,A_2  \in {\mathcal A}_n$ there exists an $l$-path   of $n$-atoms from $A_1$ to $A_2$.
\end{lemma}
\begin{proof}
 For $n= 1$, the result is trivial.
Let us assume by induction that the result holds for some $n-1 \geq 1$ and let us prove it for $n$.

Let  $E, F \in {\mathcal A}_{n}$. From equality (\ref{eqn99}) of Remark \ref{remarkAtomsGeneration-n}, there exists unique atoms $B_{-1}, B_0, B_1 \in {\mathcal A}_{n-1} $ such that   $E \in \Gamma_n(B_{-1}, B_0, B_1).$   Then
   $B_{-1}  \rah  B_0, $ $ E  \subset B_0 $ and, by Remark  \ref{remarkAtomsGeneration-n}:
\begin{equation} \label{eqn20}E   \rah  E_1, \qquad \forall \ E_1 \in \Omega_n(B_0, B_1).\end{equation}
Analogously, there exists unique atoms  $B_{*}, B_{*+1} \in {\mathcal A}_{n-1} $ such that  $F \in \Omega_n(B_*, B_{*+1}).$  Then  $B_*  \rah B_{*+1}, $ $ F  \subset B_{*+1} $ and     \begin{equation} \label{eqn21}E_{*}  \rah  F, \qquad \forall \ E_* \in \bigcup_{\substack {B_{*-1} \in {\mathcal A}_{n-1}:   \\ B_{*-1} \rah B_*}} \Gamma_n(B_{*-1}, B_*, B_{*+1})\end{equation}

Since $B_1, B_{*} \in {\mathcal A}_{n-1}$ the induction hypothesis ensures that that for all  $l \geq 2n-2  $ there exists an $l$-path $(B_1, B_2, \ldots, B_{l})$ from $B_1$ to $B_l= B_*$. We write $B_{*-1} = B_{l-1}$, $B_{*} = B_l, \   B_{* + 1} = B_{l+1}$.
So 
(\ref{eqn21}) becomes
 \begin{equation} \label{eqn21b}E_{l}  \rah  F, \qquad \forall \ E_l  \in \Gamma_n(B_{l-1}, B_l, B_{l+1})\end{equation}

Taking into account that $B_{i-1} \ \rah B_i$ for $1 <i \leq l$,  and applying  Remark \ref{remarkAtomsGeneration-n}, we deduce  that, if     $  E_{i-1}  \in \Gamma_n( B_{i-2} , B_{i-1}, B_{i} ) \subset {\mathcal A}_n,  $  then
\begin{equation}
\label{eqn23}E_{i-1}  \rah E_{i}, \qquad \forall \ E_i \in \Omega_n(B_{i-1}, B_{i}), \qquad \forall \ 1 <i \leq l.\end{equation}
Combining  (\ref{eqn20}), (\ref{eqn21b}) and (\ref{eqn23}) yields an $(l+2)$-path  $(E, E_1, \ldots, E_{l}, F) $ of atoms of generation $n$ from $E$ to $F$, as required.
\end{proof}

\begin{lemma}
{\label{lemmaPathsOfAtoms2}}
Let   $n, l  \geq 2$. For each   $l$-path $(B_1, \ldots, B_{l})$   of $(n-1)$-atoms
 there exists an $l$- path   $(E_1, E_2, \ldots, E_{l})$  of $n$-atoms   such that
$E_i \subset {{\rm int}}(B_i) $ for all $i= 1, 2, \ldots,   l.$

\begin{proof}
In the proof of Lemma \ref{lemmaPathsOfAtoms} for each $l$-path  $(B_1, B_2, \ldots, B_l)$  of $(n-1)$-atoms we have constructed the $l$-path $(E_1, E_2, \ldots, E_l)$ of $n$-atoms as required.
\end{proof}
\end{lemma}

\begin{definition}{\bf (The $\mathbf \Lambda$-set)} \em  \label{definitionLambdaSet}
Let $\Phi \in {\mathcal H} \subset C^0(D^m)$ be a model map. Let
 ${\mathcal A}_0, {\mathcal A}_1, \ldots , {\mathcal A}_n, \ldots$  be its sequence
of  families of  atoms. The subset
$$\Lambda := \bigcap_{n \geq 0} \bigcup_{A \in {\mathcal A}_n} A$$
 of ${{\rm int}}(D^m)$ is called \em the $\Lambda$-set \em of the map $\Phi$.
 \end{definition}

From Definition \ref{definitionAtomsGeneration-n}, we know that, for each fixed $n \geq 0$, the set
 $\Lambda_n :=  \bigcup_{A \in {\mathcal A}_n} A,$
is nonempty, compact, and ${{\rm int}}(\Lambda_n) \supset \Lambda_{n+1}.$ Therefore,  $\Lambda$ is  also non\-empty and compact. Moreover, $\Lambda_n$ is composed of a finite number of connected components $A \in {\mathcal A}_n$, which by Definition \ref{DefinitionModel}, satisfy
 $\lim_{n \rightarrow + \infty} \max_{A \in {\mathcal A}_n} {{\rm diam}}A = 0.$
Since  $\Lambda := \bigcap_{n \geq 0}\Lambda_n,$
we deduce that the $\Lambda$-set is \em a Cantor set \em contained in ${{\rm int}}(D^m)$.

\begin{lemma}
\label{lemmaAtomCapLambda}
Let $n,l \geq 1$ and $A_1, A_2 \in {\mathcal A}_n$. If  there exists an $l+1$-path   from $A_1$ to $A_2$, then $\Phi^l(A_1 \cap \Lambda) \cap (A_2 \cap \Lambda) \neq \emptyset$.
\end{lemma}

\begin{proof}
Assume that there exists an $(l+1)$-path from $A_1$ to $A_2$. So, from Lemma
 \ref{lemmaPathsOfAtoms2},  for all $j \geq n$ there exists atoms $B_{j,1}, B_{j,2} \in {\mathcal A}_j$ and an $(l+1)$-path from $B_{j,1}$  to $B_{j,2}$ (with constant  length $l +1$)  such that
 $$B_{n, i}= A_i, \ \ \ B_{j+1, i} \subset B_{j,i}, \quad \forall  \ j \geq  n, \ \  \ i= 1,2.$$
Construct the following two points $x_1$ and $x_2$:
$$\{x_i\}= \bigcap_{j \ge n_0}  B_{j,i}, \ \ \ \ i= 1,2.$$
By Definition \ref{definitionLambdaSet},  $x_i \in A_i \cap \Lambda.$
So, to finish the proof of Lemma \ref{lemmaAtomCapLambda} it is enough to prove  that $\Phi^l(x_1) = x_2$.

Recall that $l $ is fixed.
Since $\Phi$ is uniformly  continuous, for any $\e >0$ there exists $\delta > 0$ such that if
$(y_0, y_1, \ldots, y_{l}) \in (D^m)^l $ satisfies $d(\Phi(y_i), y_{i+1}) < \delta$ for $0 \le i \le l-1$,  then the  points $y_0$ and $y_l$ satisfy
 $d(\Phi^l(y_0), y_l) < \e$.
We choose $\delta$ small enough  such that additional  $d(\Phi^l(x), \Phi^l(y)) < \e$ if
 $d(x,y) < \delta$.

From (\ref{eqnLimitDiamAtoms=0}), there exists $j  \geq n$ such that
${{\rm diam}}(B_{j,i}) < \delta.$
Since there exists an $(l+1)$-path from $B_{j,1}$ to $B_{j,2}$, there exists a $(y_0,\dots,y_l)$ as in the previous paragraph with
 $y_0 \in B_{j,1}$ and $ y_l \in B_{j,2}.$
Thus
$$d(\Phi^l(x_1), x_2)  \leq d(\Phi^l(x_1), \Phi^l(y_0)) + d(\Phi^l(y_0), y_l) + d(y_l, x_1) $$ $$<{{\rm diam}}(\Phi^l(B_{j,1})) + \e + {{\rm diam}}(B_{j,2}) < 3 \e.$$
Since $\e>0$ is arbitrary, we obtain
  $\Phi^l(x_1) = x_2$,  as required.
\end{proof}

\begin{lemma}{\bf (Topological dynamical properties of $\mathbf \Lambda$)}
\label{lemmaLambda}
\begin{enumerate}[a)]
\item The $\Lambda$-set of a model map $\Phi \in {\mathcal H}$ is $\Phi$-invariant, i.e., $\Phi(\Lambda) = \Lambda$.

\item The map $\Phi$ restricted to the $\Lambda$-set is topologically mixing.

\item  In particular,
 $\Phi^l(A_1 \cap \Lambda) \cap (A_2 \cap   \Lambda) \neq \emptyset,$
for all $n \geq 1$, for any two atoms $A_1, A_2 \in {\mathcal A}_n$ and for all $l \geq 2n -1$.
\end{enumerate}
\end{lemma}
\begin{proof}
{a) } Let $x \in \Lambda$ and let $\{A_n(x)\}_{n \geq 0}$ the unique sequence of atoms   such that
 $x \in  A_n(x)$ and $ A_n(x) \in {\mathcal A}_n$ for all $    n \geq 0. $  Then,  $ \Phi(x) \in \Phi(A_n(x))$ for all $ n \geq 0.$
From Definition \ref{definitionAtomsGeneration-n}, for all $n \geq 0$ there exists an atom $B_n \in {\mathcal{A}_n}$ such that  $A_n(x) \rah B_n.$ Therefore $ \Phi(A_n(x)) \cap B_n \neq \emptyset$. Let $d$ denote the Hausdorff distance between subsets of $D^m$,
we deduce
$$d(\Phi(x), B_n) \leq {{\rm diam}}\big (\Phi(A_n(x))\big) + {{\rm diam}}\big(B_n\big) .$$
Moreover,  Equality (\ref{eqnLimitDiamAtoms=0}) and the continuity of $\Phi$ imply
 $$\lim_{n \rightarrow + \infty}  \max\big \{{{\rm diam}}\big(\Phi(A_n(x))\big), {{\rm diam}}\big (B_n\big )\big \} = 0.$$
Then, for all $\e >0$ there exists $n_0 \geq 0$
such that
 $d(\Phi(x), B_n) < \e  $   for some atom  $B_n \in {\mathcal A}_n  $ for all $ n \geq n_0. $
Since any atom of any generation intersects $\Lambda$, we deduce that
  $ d(\Phi(x), \Lambda)  < \e$ for each $\e >0$. Since $\Lambda$ is compact,  this implies
  $\Phi(x) \in \Lambda$. We have proved that
  $\Phi(\Lambda) \subset \Lambda.$

  Now, let us prove the other inclusion.
Let $y \in  \Lambda $   and let $\{B_n(y)\}_{n \geq 0}$ the unique sequence of atoms   such that
 $y  \in  B_n(y)$ and $ B_n(y) \in {\mathcal A}_n$ for all $n \geq 0. $
From Definition \ref{definitionAtomsGeneration-n}, for all $n \geq 0$ there exists an atom $A_n \in {\mathcal{A}_n}$ such that  $A_n \rah B_n(y).$ Therefore  $\Phi(A_n) \cap B_n(y) \neq \emptyset.$
We deduce that, for all $n \geq 0$, there exists a point  $x_n \in A_n \in {\mathcal A}_n$ such that $\Phi(x_n)  \in B_n(y)$. Since any atom $A_n$ contains points of $\Lambda$, we obtain
$$d(x_n,\Lambda) \leq {{\rm diam}}(A_n)  \hbox{ and }
d(\Phi(x_n), y) \leq \mbox{ diam} (B_n(y)), \qquad \forall \ n \geq 0.$$
Let $x$ be the limit of a convergent subsequence of $\{x_n\}_{n \geq 0}$.  Applying Equality (\ref{eqnLimitDiamAtoms=0})  and the continuity of $\Phi$, we deduce that
 $d(x,\Lambda) = 0$ and $d(\Phi(x), \, y) = 0$.
This means that $y =\Phi(x) $ and $x \in \Lambda$. We have proved that $y \in \Phi(\Lambda)$ for all $y \in \Lambda$; namely $\Lambda = \Phi(\Lambda)$, as required.

c)  We will prove a stronger assertion: for any two atoms, even of different generation, there exists $l_0 \geq 1$ such that
\begin{equation}\label{next} \Phi^l(A_1 \cap \Lambda) \cap (A_2 \cap \Lambda) \neq \emptyset \ \ \forall \ l \geq l_0.\end{equation}
It is not restrictive to assume that $A_1$  and $A_2$ are atoms of the same generation $n_0$  (if not, take $n_0$ equal to the largest of both generation and substitute $A_i$ by an atom of generation $n_0$ contained in $A_i$).
Applying Lemma \ref{lemmaPathsOfAtoms}, for all $l \geq 2n_0-1$ there exists an $(l+1)$-path from $A_1$ to $A_2$. So, from Lemma \ref{lemmaAtomCapLambda} $\Phi^l(A_1 \cap \Lambda) \cap (A_2 \cap \Lambda) \neq \emptyset$, as required.

b) The intersection  of $\Lambda$ with the  atoms of all the generations  generates its topology,  thus Equation \eqref{next}
implies that  $\Lambda$ is topologically mixing.
\end{proof}

For fixed $ (A_0, A_l) \in {\mathcal A}_n^{2}$ we set
$${\mathcal A}_n^{l+1\,*}(A_0, A_l):=\{(A_0, A_1, \ldots, A_{l-1}, A_l) \in {\mathcal A}_n^{l+1\,*}\}.$$

\begin{lemma}
\label{LemmaCardinalA_n^(l+1)}
Let $l,n \geq 1$. Then
\begin{enumerate}[a)]
\item $ \label{eqn51}\# {\mathcal A}_n^{l+1\,*} = 2^{nl} \cdot (\#{\mathcal A}_n).$
\item
$\# {\mathcal A}_n^{l+1\,*}(A_0, A_l) = \frac{2^{nl}}{\#{\mathcal A}_n} \ \ \forall \ (A_0, A_l) \in {\mathcal A}^2_n$,
for all $l \geq 2 n-1$.

\end{enumerate}
\end{lemma}
\begin{proof}
a)  Each $(l+1)$-path $(A_0, A_1, \ldots, A_l)   $ of $n$-atoms is determined by a free choice of the atom $A_0 \in {\mathcal A}_{n}$, followed by the choice of the atoms $A_j \in {\mathcal A}_n$ such that $A_j \rah A_{j-1}$ for all $j =1, \ldots, l$.    From equality ii) of Definition \ref{definitionAtomsGeneration-n},  we know that for any fixed $A \in {\mathcal A}_n$ the number of atoms $B \in {\mathcal A}_n$ such that $B \rah A$ is   $2^n$. This implies  \ref{eqn51}, as required.

 b)  We argue by induction on $n$. Fix $n=1$ and $l \geq 1$. Since any two atoms $A_j, A_{j+1} \in {\mathcal A}_1$ satisfies $A_j \rah A_{j+1}$, the number of $(l+1)$-paths $$(A_0, A_1, \dots,A_j, A_{j+1}, \ldots A_{l-1}, A_l)$$ of $1$-atoms with $(A_0, A_l)$ fixed, equals $\#({\mathcal A}_1)^{l-1} = 2^{l-1}= 2^{l}/2 = 2^{nl}/(\#{\mathcal A}_n)$ with $n=1$.

Now, let us assume that assertion b) holds for some $n \geq 1$ and let us prove it for $n+1$.
Let $l \geq 2(n+1) -1 = 2n + 1 \geq 3$ and let $(B_0, B_l) \in {\mathcal A}_{n+1}^{2}$. From equality (\ref{eqn99}) and conditions a) and b)   of Definition \ref{definitionAtomsGeneration-n}, there exists unique $(A_{-1}, A_0, A_1) \in {\mathcal A}_n^{3*}$ and unique $(A_{l-1}, A_l) \in {\mathcal A}_n^{2*}$ such that
$$B_0 \in \Gamma_{n+1}(A_{-1}, A_0, A_1), \ \ \ \ B_l \in \Omega_{n+1}(A_{l-1}, A_l).$$
As $(A_1, A_{l-1}) \in {\mathcal A}_n^2$  and $l-2 \geq 2n-1$, the induction hypothesis ensures that the number of $(l-1)$-paths $(A_1, A_2, \ldots, A_{l-1})$ from $A_1$ to $A_{l-1}$ is
\begin{equation} \label{eqn113}\#{\mathcal A}_n^{l-1\,*}(A_1, A_{l-1}) = \frac{2^{n(l-2)}}{\#{\mathcal A}_n} = \frac{2^{n(l-2)}}{2^{n^2}}= 2^{n l-2n-n^2}. \end{equation}

Let ${\mathcal C}(B_0,B_l)$ be the set
$$
\hspace{-2cm}  \bigcup_{\hspace{2cm}(A_1, \ldots, A_{l-1}) \in  {\mathcal A}_n^{l-1\,*}(A_1, A_{l-1})} \hspace{-2.5cm}
\big \{(B_0, B_1, \ldots,  B_l)\in{\mathcal A}_{n+1}^{l+1} \colon
 B_j \in \Gamma_{n+1} (A_{j-1}, A_j, A_{j+1}) \ \forall j  \big \}, \nonumber
$$
where the families in the above union are pairwise disjoint.
 It is standard to  check that the families   in the union ${\mathcal C}(B_0,B_l)$  are pairwise disjoint, because for $A \neq \widetilde A$ in ${\mathcal A}_n$, the families $\Gamma_{n+1}(\cdot,A,\cdot)$ and $\Gamma_{n+1}(\cdot,\widetilde A,\cdot)$  are disjoint.

A straightforward verification shows that

\begin{equation}
\label{eqn114ToBeProved}
{\mathcal A}_{n+1}^{l+1\,*}(B_0, B_l) = {\mathcal C}(B_0,B_l).
\end{equation}

Now, applying  (\ref{eqn113}) and (\ref{eqn114ToBeProved}), we obtain

\begin{eqnarray*}
&&\hspace{-0.9cm}\#{\mathcal A}_{n+1}^{l+1\,*}(B_0, B_l) =\\
& =& \hspace{-0.8cm} \sum_{(A_1, \ldots, A_{l-1}) \in  {\mathcal A}_n^{l-1\,*}(A_1, A_{l-1})}  \hspace{-0.8cm}   \#\big  \{(B_0, B_1, \ldots,  B_l)\in{\mathcal A}_{n+1}^{l+1} \colon
 B_j \in \Gamma_{n+1} (A_{j-1}, A_j, A_{j+1}) \ \forall j  \big \}\\
& = &  \hspace{-0.8cm}
 \sum_{ (A_1, \ldots, A_{l-1}) \in  {\mathcal A}_n^{l-1\,*}(A_1, A_{l-1})}  \prod_{j=1}^{l-1}    \#\Gamma_{n+1} (A_{j-1}, A_j, A_{j+1})  \\
& = & (\#{\mathcal A}_n^{l-1\,*}(A_1, A_{l-1})) \cdot 2^{l-1} = 2^{nl-2n-n^2+l-1}= 2^{(n+1)l-(n+1)^2} =
 \frac{2^{(n+1)l}}{\#{\mathcal A}_{n+1}}
\end{eqnarray*}
as required.
\end{proof}

Let
 $\vec{A}_n^l := (A_0, A_1, \ldots, A_l)$
be an $(l+1)$-path of $n$-atoms,
and
${\mathcal F}_{n, l} (\vec{A}_n^l )  := \Big \{G \in {\mathcal A}_{n+l} \colon  G \cap \Lambda \subset \bigcap_{j=0}^l \Phi^{-j}(A_j) \Big\}.   $

\begin{lemma} {\bf (Intersection of $\mathbf \Lambda$ with $\mathbf l$-paths)}
\label{lemmaLambda2}
Fix $l,n \geq 1$. Then

\begin{enumerate}[a)]

\item  For any $G \in {\mathcal A}_{n+l}$, there exists a unique $(l+1)$-path $(A_0, A_1, \ldots, A_l)$ of $n$-atoms such that
$G \cap \Lambda \subset \bigcap_{j=0}^{l} \Phi^{-j}(A_j).$

\item For any atoms $G \in {\mathcal A}_{n+l}$, $A \in
{\mathcal A}_n$ and  $j \in \{ 0,1, \ldots, l\}$:
$$  (G \cap \Lambda) \cap \Phi^{-j}(A) \neq \emptyset \ \ \Leftrightarrow \ \ G \cap \Lambda \subset \Phi^{-j}(A).$$

\item For any  $(l+1)$-path $\vec{A}_n^l = (A_0, A_1, \ldots, A_l)$ of $n$-atoms,
 \begin{equation}
\label{eqn39b}
\Lambda \cap \bigcap_{j=0}^l \Phi^{-j}(A_j) = \bigcup_{G \in {\mathcal F}_{n, l} (\vec{A}_n^l)} G \cap \Lambda,  \end{equation}

\item For any atom $G \in {\mathcal A}_{n+l}$ and any path $\vec A_n^l \in {\mathcal A}_n^{l+1 \, *}$:\\
 $ G\in {\mathcal F}_{n,l}(\vec A_n^l) $ if and only if there exists $(G_0, G_1, \ldots, G_l) \in {\mathcal A}_{n+l}^{l+1 \, *} $ such that $G_0= G$  and  $ G_j \subset A_j $ for all  $  j= 0, 1, \ldots, l.$

\item For any $(l+1)$-path $(A_0, A_1, \ldots, A_l)$ of $n$-atoms,
   $$\displaystyle {\#{\mathcal F}_{n, l } (\vec{A}_n^l )  =\frac{1}{2^{nl}} \cdot  \frac{\#{\mathcal A}_{n+l}}{\#{\mathcal A}_{n}}}.$$
\end{enumerate}
\end{lemma}

\begin{proof} a) From equalities (\ref{eqn99a}) and (\ref{eqn99}), for any atom $G$ of generation $n+l$ there exist two unique atoms $B, C$ of generation $n+ l -1$ such that  $B \rah C$,  $ G \subset B$ and $G \rah E  $ for all $ E \in \Omega_{n+l}(B,C) $. Moreover, from Remark \ref{remark S(P,Q), G_i}, we have
\begin{equation}\label{eqn40}
\Phi(G) \cap   F \neq \emptyset   \mbox{ if and only if } \ F \in  \Omega_{n+l}(B,C).\end{equation}
We claim that
\begin{equation}
\label{eqn41} \Phi(G \cap \Lambda) \subset {{\rm int}}(C).\end{equation}
 Since  $\Lambda$ is $\Phi$-invariant, for any $x \in G \cap \Lambda$, we have $\Phi(x) \in \Phi(G) \cap \Lambda$. Therefore $\Phi(x)$ is in the interior of some atom $E(x)$ of generation $n+l$  (see Definition \ref{definitionLambdaSet}). From (\ref{eqn40}), $E(x) \in \Omega_{n+l}(B,C)$. Thus  $E(x) \subset {{\rm int}}(C)$ and $\Phi(x) \in {{\rm int}}(C)$ for all $ x \in G \cap \Lambda $ proving (\ref{eqn41}).

 So, there exists $C_1 \in {\mathcal A}_{n+l-1}$ such that $\Phi(G \cap \Lambda) \subset {{\rm int}}(C_1) \cap \Lambda$. Applying the same assertion to $C_1$ instead of $G$, we deduce that there exists $C_2 \in {\mathcal A}_{n+l-2}$ such that $\Phi(C_1 \cap \Lambda) \subset {{\rm int}}(C_2) \cap \Lambda$. So, by induction, we construct atoms
 $$C_1, C_2, \ldots, C_l \ \ \mbox{ such that } \ \ C_j \in {\mathcal A}_{n+l-j} \ \mbox{ and }$$ $$ \ \Phi^j(G \cap \Lambda) \subset {{\rm int}} ( C_j) \cap \Lambda, \qquad  \forall \ j= 1, \ldots, l.$$
 Since any atom of  generation larger than $n$ is contained in a unique atom of generation $n$, there exists $A_0, A_1, \ldots, A_l \in {\mathcal A}_n$ such that $ A_0 \supset  G$ and $ A_i \supset C_i, \qquad  \forall \ i= 1, \ldots, l.$ We obtain
 $$\Phi^j (G \cap \Lambda) \subset {{\rm int}}(A_j), \qquad  \forall \ j= 0,1, \ldots, l. $$
 Besides, $(A_0, A_1, \ldots, A_l)$ is an $(l+1)$-path since
  $\emptyset \neq \Phi^j(G \cap \Lambda)  \subset   \Phi(A_{j-1}) \cap {{\rm int}}(A_j)$;   hence  $A_{j-1} \rah A_j$ for all $ j= 1, \ldots, l.$  Then,
  $G \cap \Lambda \subset   \Phi^{-j} (A_j)$ for all $ j= 0,1, \ldots, l$; proving the existence statement in a).

To prove uniqueness assume that $(A_0, A_1, \ldots, A_l)$  and $(A'_0, A'_1, \ldots, A'_l)$ are paths of $n$-atoms such that $$G \cap \Lambda \subset \Phi^{-j}(A_j) \cap \Phi^{-j}(A'_j)\ \ \forall \ j \in \{0,1, \ldots, l\}.$$
  Then $A_j \cap A'_j \neq \emptyset$ for all $j \in \{0,1, \ldots, l\}$. Since two different
 atoms of the same generation are pairwise disjoint, we deduce that $A_j= A'_j$ for all $j \in \{0,1, \ldots, l\}$ as required.

  b) Trivially, if $G \cap \Lambda \subset \Phi^{-j}(A)$, then $(G \cap \Lambda) \cap \Phi^{-j}(A) \neq \emptyset$. Now, let us prove the converse assertion.  Fix $G \in {\mathcal A}_{n+l}$ and $A \in {\mathcal A}_n$ satisfying $(G \cap \Lambda) \cap \Phi^{-j}(A) \neq \emptyset$. Applying part a) there exists $\widetilde A \in {\mathcal A}_n$ such that $G \cap \Lambda \subset \Phi^{-j}(\widetilde A)$. Therefore $G \cap \Lambda \cap \Phi^{-j}(A) \subset \Phi^{-j}(\widetilde A \cap A) \neq \emptyset$. Since $A$ and $\widetilde A$ are atoms of generation $n$, and two different atoms of the same generation are disjoint, we conclude that $\widetilde A = A$, hence $G \cap \Lambda \subset \Phi^{-j}(A)$, as required.

c)  For the $(l+1)$-path  $\vec{A}_n^l = (A_0, A_1, \ldots, A_l)$ of $n$-atoms, construct
\begin{equation}
\label{eqn112}
\widetilde {\mathcal F}_{n, l}(\vec{A}_n^l) := \big\{G \in {\mathcal A}_{n+l} \colon G \cap \Lambda \cap  \Phi^{-j}(A_j ) \neq \emptyset \ \ \forall j \in \{0,1, \ldots, l\} \big \}.\end{equation}
From the definitions of  the  families ${\mathcal F}_{n,l}$ and $\widetilde {\mathcal F}_{n,l}$, and taking into account that $\Lambda$ is contained in the union of $(n+l)$-atoms, we obtain:
$$ \bigcup_{G \in {\mathcal F}_{n,  l}(\vec{A}_n^l)} G \cap \Lambda   \  \  \subset \  \ \Lambda \cap \big(\bigcap_{j=0}^l \Phi^{-j}(A_j)\big) \  \  \subset \  \    \bigcup_{G \in \widetilde {\mathcal F}_{n,  l}(\vec{A}_n^l)} G \cap \Lambda . $$
 Therefore, to prove Equality (\ref{eqn39b}) it is enough to show that \begin{equation}
 \label{eqn112b}
 \widetilde {\mathcal F}_{n, l}(\vec{A}_n^l) = {\mathcal F}_{n, l}(\vec{A}_n^l), \end{equation}
 but this equality immediately follows  from the construction of the families  ${\mathcal F}_{n, l}(\vec{A}_n^l)$ and $\widetilde{\mathcal F}_{n, l}(\vec{A}_n^l)$ by assertion b).

 d)  For each $(l+1)$-path $\vec A_n^l = (A_0, A_1 \ldots, A_l)$ of $n$-atoms construct the  family  ${\mathcal G}_{n,l}(\vec A_n^l) := $  $$\big\{G_0 \in {\mathcal A}_{n+l}\colon \ \exists (G_0, G_1, \ldots, G_l) \in {\mathcal A}_{n+l}^{l+1 \,*} \mbox{ such that } G_j \subset A_j \ \forall j \big\} $$

We will first prove that ${\mathcal G}_{n,l}(\vec A_n^l) \supset {\mathcal F}_{n,l}(\vec A_n^l) $.
In fact, take $G  \in {\mathcal F}_{n,l}(\vec A_n^l) $, and take any point $x \in G \cap \Lambda$.
We have $\Phi^j(x) \in A_j \cap \Lambda$ for all $j \in \{0,1, \ldots, l\}$ (recall that $\Lambda$ is
$\Phi$-invariant). Since any point in $\Lambda$ is contained in the interior of some atom of any generation, there exists an atom $G_j$ of generation $n+l$ such that
$\Phi^j(x) \in {{\rm int}}(G_j)$. Recall that each atom of generation $n+l$ is contained in a unique atom of generation $n$. As $\Phi^j(x) \in G_j \cap A_j \neq \emptyset$, and different atoms
of the same generation are disjoint, we conclude that $G_j \subset A_j$. Besides $G_0= G$ because
$x \in G \cap G_0$. Finally $(G_0, G_1, \ldots, G_l)$ is a $(l+1)$-path because
$\Phi^{j+1}(x) = \Phi (\Phi^{j}(x)) \in \Phi(G_j) \cap {{\rm int}}(G_{j+1})$ for all
$j \in \{ 0, 1, \ldots, l-1\}$; namely $G_j \rah G_{j+1}$. We have proved that
$G \in {\mathcal G}_{n,l}(\vec A_n^l)$, as required.

Now, let us prove that ${\mathcal G}_{n,l}(\vec A_n^l) \subset {\mathcal F}_{n,l}(\vec A_n^l) $. Assume that $G_0 \in {\mathcal A}_{n+l}$ and $(G _0, G_1, \ldots, G_l) \in {\mathcal A}_{n+l}^{l+1\,*} $ satisfies $G_j \subset A_j$ for all $j \in \{0,1, \ldots, l\}$.  Therefore $(G_0, G_1, \ldots, G_j)$ is a $(j+1)$-path of $(n+l)$-atoms for all $j \in \{1, 2, \ldots, l\}$. Applying Lemma \ref{lemmaAtomCapLambda}, we obtain
$ G_0  \cap \Lambda \cap \Phi^{-j}( G_j )\neq \emptyset$. Therefore, taking into account that $G_j \subset A_j$, we deduce that
$$G_0  \cap \Lambda \cap \Phi^{-j}( A_j )\neq \emptyset \ \ \forall \ j \in \{0,1, \ldots, l\}.$$
Therefore $G_0 \in \widetilde {\mathcal F}_{n,l}(\vec A_n^l) = {\mathcal F}_{n,l}(\vec A_n^l)$ (recall  (\ref{eqn112}) and (\ref{eqn112b})). This holds for any $G_0 \in {\mathcal G}_{n,l}(\vec A_n^l)$, thus ${\mathcal G}_{n,l}(\vec A_n^l) \subset {\mathcal F}_{n,l}(\vec A_n^l) $,  as  required.

e)   From Assertion a)   we obtain:
\begin{equation} \label{eqn50} {\mathcal A}_{n+l} =    \bigcup_{\vec{A}_n^l\in {\mathcal A}_{n}^{l+1 \,*}} {\mathcal F}_{n, l}(\vec{A}_n^l) ,\end{equation}
where the families in the above union are pairwise disjoint, due to the uniqueness property of assertion a).

Recall the characterization of the  family ${\mathcal F}_{n, 1}(\vec{A}_n^l)$ given by Assertion d). From  Definition \ref{definitionAtomsGeneration-n}- condition a) and equality ii), the number of atoms of each generation larger than $n$ that are contained in each $A_j \in {\mathcal A}_n$, and also the number of atoms  $G_{j} \in {\mathcal A}_{n+1}$ such that $G_j \rah G_{j+1}$, are constants that  depend only on the generations but not on the chosen  atom. Therefore, there exists a constant $k_{n, l}$ such that $\#{\mathcal F}_{n, l}(\vec{A}_n^l)= \#{\mathcal G}_{n, l}(\vec{A}_n^l) =  k_{n, l}$ for all the $(l+1)$-paths of $n$-atoms. So, from Equality (\ref{eqn50})  we obtain:
$$ \#{\mathcal A}_{n+l} = (\#  {\mathcal{A}  }_{n}^{l+1\,*}) \cdot ( \# {\mathcal F}_{n,l}(\{A_j\})),$$
 and applying Lemma \ref{LemmaCardinalA_n^(l+1)}, we conclude
 $$ \#{\mathcal A}_{n+l} = 2^{nl} \cdot (\#  {\mathcal{A}  }_{n})  \cdot ( \# {\mathcal F}_{n,l}(\{A_j\})),$$
 as required.
\end{proof}

We turn to the proof of Lemma \ref{LemmaMain}. We will first construct the measure $\nu$ and then prove that it has the required properties.

We start by defining an additive pre-measure  on  the $\Lambda$-set of $\Phi$ by
$$\nu^*(A \cap \Lambda) := \frac{1}{\#{\mathcal A}_n}, \qquad  \forall \ A \in {\mathcal A}_n, \qquad \forall \ n \geq 0.$$
Since $\nu^*$ is a pre-measure defined in a family of sets that generates the Borel $\sigma$-algebra of $\Lambda$,
there exists a unique Borel probability measure $\nu$ supported on $\Lambda$ such that
\begin{equation}
\label{eqnCosntructionOfNu}
\nu (A \cap \Lambda) := \frac{1}{\#{\mathcal A}_n}, \qquad \forall \ A \in {\mathcal A}_n, \qquad \forall \ n \geq 0.\end{equation}

In the following lemmas we will prove that
 $\nu$ is $\Phi$-invariant,
 mixing, and that the metric entropy $h_{\nu}(\Phi)$   is infinite.

\begin{lemma} \label{lemmaNuInvariante}
$\nu$ is invariant by $\Phi$.
\end{lemma}

\begin{proof}
Since the atoms of all generation intersected with $\Lambda$ generates the Borel $\sigma$-algebra of $\Lambda$, it is enough to prove that
\begin{equation}
\label{eqn37ToBeProved}
\nu(C \cap \Lambda) = \nu(\Phi^{-1}(C  \cap \Lambda)), \qquad \forall \ C \in {\mathcal A}_n, \qquad \forall \ n \geq 0.\end{equation}
From  (\ref{eqn99}),   taking into account that $\Lambda$ is invariant and that any point in $\Lambda$  belongs to an atom of generation $n+1$, we obtain:

$$\Phi^{-1}(C \cap \Lambda) =  \bigcup_{\substack {B \in {\mathcal A}_n \\ B \rah C}}\ \ \ \ \bigcup_{\substack {D \in {\mathcal A}_n \\ D \rah B}} \ \ \    \bigcup_ {G \in   \Gamma_{n+1}(D,B,C)}  (G \cap \Lambda),$$
where both unions are of pairwise disjoint sets. Using equalities ii) of Definition \ref{definitionAtomsGeneration-n}, we obtain
\begin{align}
\nu (\Phi^{-1}(C \cap \Lambda)) &=& \sum_{\substack {B \in {\mathcal A}_n \\ B \rah C}}\ \ \ \ \sum_{\substack {D \in {\mathcal A}_n \\ D \rah B}} \ \ \ \sum_  {G \in   \Gamma_{n+1}(B,C,D)}  \nu (G \cap \Lambda)\nonumber \\
&=& N_C \cdot N_B \cdot  (\#\Gamma_{n+1}(B,C,D)) \cdot \frac{1}{\#{\mathcal A}_{n+1}},
\end{align}
where
 $N_X:= \#\{Y \in {\mathcal A}_n  \colon Y \rah X    \}) = 2^n$  for all $X \in {\mathcal A}_n$.   Since $           \#\Gamma_{n+1}(B,C,D)) = 2 $ (see Remark \ref{remarkAtomsGeneration-n}) and $  {\#{\mathcal A}_{n+1}}= {2^{(n+1)^2}}  $, we conclude $$\nu (\Phi^{-1}(C \cap \Lambda)) = 2^{n} \cdot 2^{n} \cdot 2 \cdot \frac{1}{2^{(n+1)^2}} = \frac{1}{2^{n^2}} = \frac{1}{\#{\mathcal A}_{n }} = \nu(C \cap \Lambda),
$$ proving Equality (\ref{eqn37ToBeProved}) as required.
\end{proof}

\begin{lemma} \label{lemmaNuMixing}
$\nu$ is mixing.
\end{lemma}

\begin{proof}
The family of atoms of all generations intersected with $\Lambda$ generates the Borel $\sigma$-algebra of $\Lambda$, thus it is enough to prove that for any pair $(C_0, D_0) $ of atoms (of equal or different generations) there exists $l_0 \geq 1$ such that
\begin{equation}
\label{eqn200ToBeProved}\nu(\Phi^{-l}(D_0 \cap \Lambda) \cap (C_0\cap \Lambda))= \nu(C_0 \cap \Lambda) \cdot \nu(D_0 \cap \Lambda) \ \ \forall \ l \geq l_0.\end{equation}

\vspace{.3cm}

Let us first prove this in the case that $C_0$ and $D_0$  are atoms of the same generation $n$. Take $l \geq 2n-1$. Applying  Lemma \ref{lemmaLambda}-c), we have
$\Phi^{-l}(D_0 \cap \Lambda) \cap (C_0\cap \Lambda) \neq \emptyset \ \ \forall \ l \geq 2n-1. $

Fix $l \geq 2n-1$. We will use the notation
$$\vec A^l_n :=(C_0, A_1, \ldots, A_{l-1}, D_0) \in  {\mathcal A}_n^{l+1\,*}(C_0, D_0)$$
to denote any one of the
$2^{nl}/(\#{\mathcal A}_{n})$ different $l+1$-paths of $n$-atoms from $C_0$ to $D_0$
(see  Lemma \ref{LemmaCardinalA_n^(l+1)}-b)).

We assert that
\begin{equation}
\label{eqn201ToBeProved}
  \Phi^{-l} (D_0 \cap \Lambda) \cap (C_0 \cap \Lambda) = T := \bigcup_{\vec A^l_n \in {\mathcal A}_n^{l+1\,*}(C_0, D_0)} \ \  \bigcup_{B \in {\mathcal F}_{n,l}(\vec A_n^l)} (B \cap \Lambda), \end{equation}
where the family ${\mathcal F}_{n,l}(\vec A_n^l)$ of $(n+l)$-atoms is defined in  Lemma \ref{lemmaLambda2}-c).

First, let us prove that
$\Phi^{-l} (D_0 \cap \Lambda) \cap (C_0 \cap \Lambda) \subset T $. Fix $x \in  (D_0 \cap \Lambda) \cap (C_0 \cap \Lambda) $. Then
$C_0, D_0$ are the unique atoms of generation $n$ that contain  $x$ and $\Phi^l(x) \in \Phi^l(\Lambda) = \Lambda$ respectively. Since
$x \in \Lambda$, there exists a unique atom $B$ of generation $n+l$ that contains $x$. Applying  Lemma \ref{lemmaLambda2}-a) there
exists a unique $(A_0, A_1, \ldots, A_l)\in{\mathcal A}_n^{l+1 \,*}  $ such that $B \cap \Lambda\subset \Phi^{-j}(A_j)$ for all $j\in \{0,1,\ldots, l\}$. Since the $n$-atom  that contains $x$ is $C_0$, and two different $n$-atoms   are disjoint, we deduce that $A_0= C_0$. Analogously, since the $n$-atom that contains $\Phi^l(x)$ is $D_0$ and the preimages of two different $n$-atoms are disjoint, we deduce that $A_l = D_0$.
Therefore we have found $\vec A_n^l= (C_0, A_1, \ldots, A_{l-1}, D_0)$ and $B \in {\mathcal F}_{n,l}(\vec A_n^l)$ such that $x \in B \cap \Lambda$. In other words, $x \in T$, as required.

Next, let us prove that
$\Phi^{-l} (D_0 \cap \Lambda) \cap (C_0 \cap \Lambda) \supset T $. Take $B \in {\mathcal F}_{n,l} (\vec A_n^l)$ for some $\vec A_n^l = (C_0, A_1, \ldots, A_{l-1}, D_0)$. From the definition of the family ${\mathcal F}_{n,l} (\vec A_n^l)$ in Lemma \ref{lemmaLambda2}-c), we have
 $B \cap \Lambda \subset (C_0 \cap \Lambda) \cap  \Phi^{-l}(D_0)$. Besides $B \cap \Lambda \in \Phi^{l}(\Lambda)$ because $\Phi^l(\Lambda)= \Lambda$. We conclude that
 $B \cap \Lambda \subset (C_0 \cap \Lambda) \cap   \Phi^{-l}(D_0 \cap \Lambda)$, proving that
 $T \subset \Phi^{-l} (D_0 \cap \Lambda) \cap (C_0 \cap \Lambda) $, as required.
 This ends the proof of equality (\ref{eqn201ToBeProved}).

By definition $n$-atoms are pairwise disjoint, thus the sets in the union constructing $T$ are pairwise disjoint.  Therefore, from  (\ref{eqn201ToBeProved}), and applying Lemma \ref{LemmaCardinalA_n^(l+1)}-b) and  Lemma \ref{lemmaLambda2}-e), we deduce

 \begin{align*}
 \nu ((C_0 \cap \Lambda) \cap \Phi^{-l}(D_0 \cap \Lambda)) & =
   \sum_{\vec A_n^l \in {\mathcal A}_n^{l+1\,*}(C_0, D_0)} \ \ \  \sum_{B \in {\mathcal F}_{n,l}(\vec A_n^l)} \nu(B \cap \Lambda)\\
& =  (\# {\mathcal A}_n^{l+1\,*}(C_0, D_0)) \cdot (\# {\mathcal F}_{n,l}(\vec A_n^l)) \cdot \frac{1}{\#{\mathcal A}_{n+l}}
\end{align*}
\begin{align*}
& =\frac{2^{nl}}{\#{\mathcal A}_n} \cdot \frac{1}{2^{nl}} \cdot \frac{\#{\mathcal A}_{n+l}}{\#{\mathcal A}_{n}} \cdot \frac{1}{\#{\mathcal A}_{n+l}}= \frac{1}{\#{\mathcal A}_{n}} \cdot \frac{1}{\#{\mathcal A}_{n}} \\
& = \nu(C_0 \cap \Lambda) \cdot \nu(D_0 \cap \Lambda).
 \end{align*}
This ends the proof of equality (\ref{eqn200ToBeProved}) in the case that $C_0$ and $D_0$ are atoms of the same generation $n$, taking $l_0= 2n-1$.

\vspace{.3cm}

Now, let us prove equality (\ref{eqn200ToBeProved}) when $C_0$ and $D_0$  are atoms of different generations. Let $n$ equal the
maximum of both generations. Take $l \geq 2n-1$. Since $\Lambda$ is contained in the union of the atoms of any generation, we have
$$C_0 \cap \Lambda = \bigcup_{C \in {\mathcal A}_n, C \subset C_0} C \cap \Lambda, $$ where the sets in the union are pairwise
disjoint.
Analogously $$\Phi^{-l}(D_0 \cap \Lambda)  = \bigcup_{D \in {\mathcal A}_n, D \subset D_0} \Phi^{-l}(D \cap \Lambda),$$ where also the
sets in this union  are pairwise disjoint.
So,
\begin{align*}
(C_0 \cap \Lambda)\cap \Phi^{-l}(D_0 \cap \Lambda) = \bigcup_{C \in {\mathcal A}_n, C \subset C_0}  \  \bigcup_{D \in {\mathcal A}_n, C \subset D_0}\ \hspace{-0.5cm} (C \cap \Lambda) \cap \Phi^{-l}(D \cap \Lambda).
\end{align*}
Since the sets in the union are pairwise disjoint, we deduce
$$\nu((C_0 \cap \Lambda)\cap \Phi^{-l}(D_0 \cap \Lambda)) = \sum_{C \in {\mathcal A}_n, C \subset C_0} \  \sum_{D \in {\mathcal A}_n, C \subset D_0} \hspace{-0.3cm}  \nu((C \cap \Lambda) \cap \Phi^{-l}(D \cap \Lambda)).$$

As $C,D$ are atoms of the same generation $n$, and $l \geq 2n-1$, we can apply the first case proved above, to deduce that
\begin{equation}
\label{eqn202}
\nu((C_0 \cap \Lambda)\cap \Phi^{-l}(D_0 \cap \Lambda)) = $$  $$ \#\{C \in {\mathcal A}_n, C \subset C_0\} \cdot \#\{D \in {\mathcal A}_n, C \subset D_0\}\cdot \frac{1}{(\#{\mathcal A}_n)^2}. \end{equation}
The number of atoms of generation $n$ contained in an atom $C_0$ of generation $n_1$ larger or equal than $n$, does not depend of the chosen atom $C_0$. Therefore,
$$\#\{C \in {\mathcal A}_n, C \subset C_0\}= \frac{\#{\mathcal A}_n}{\#{\mathcal A}_{n_1}}=(\#{\mathcal A}_{n }) \cdot \nu (C_0 \cap \Lambda). $$
Analogously
$$\#\{D \in {\mathcal A}_n, D \subset D_0\}=  (\#{\mathcal A}_{n }) \cdot \nu (D_0 \cap \Lambda). $$
Finally, substituting in equality (\ref{eqn202}) we conclude that\\
$\displaystyle \nu(\Phi^{-l}(D_0 \cap \Lambda) \cap (C_0\cap \Lambda))= \nu(C_0 \cap \Lambda) \cdot \nu(D_0 \cap \Lambda) \ \ \forall \ l \geq 2n-1.$
\end{proof}

\begin{lemma} \label{lemmaNuEntropyInfty}
$h_\nu(\Phi) = +\infty$.
\end{lemma}

 \begin{proof}
For $n \geq 1$ we consider the   partition $\mathcal A_n$ of $\Lambda $ consisting of all the $n$-atoms intersected with $\Lambda$.  By the definition of metric entropy

\begin{equation} \label{eqn52a}h_\nu (\Phi) := \sup_{\mathcal P} h({\mathcal P}, \nu) \geq h({\mathcal A}_n, \nu) , \mbox{ where }\end{equation}

\begin{equation} \label{eqn52b}h({\mathcal A}_n, \nu):= \lim_{l \rightarrow + \infty} \frac{1}{l} H\Big (\bigvee_{j=0}^l  (\Phi^{-j}{\mathcal A}_n), \nu\Big), \end{equation}
$$ {\mathcal Q}_l:= \bigvee_{j=0}^l  \Phi^{-j}{\mathcal A }_n : =\Big \{\bigcap_{j=0}^l \Phi^{-j}A_j \cap \Lambda \neq \emptyset: \ \ \ A_j \in {\mathcal A}_n \Big\}, $$

\begin{equation} \label{eqn52} H ({\mathcal Q}_l, \nu) := - \sum_{X \in  {\mathcal Q}_l} \nu(X) \log \nu(X). \end{equation}
For any nonempty
 $X := \Lambda \cap  \Big(\bigcap_{j=0}^l \Phi^{-j}A_j \Big)  \in \nolinebreak {\mathcal Q}_l ,$    Lemma \ref{lemmaLambda2}-c)
 yields
$$\nu(X) = \nu \Big( \bigcap_{j=0}^l \Phi^{-j}A_j \cap \Lambda    \Big) = \sum_{G \in {\mathcal F}_{n, l}(\vec A_j^l)}  \nu(G \cap \Lambda).$$
Since $G   $ is an atom of generation $n+l$,  we have $\nu(G \cap \Lambda) = {1}/( \#{\mathcal A}_{n+l})$, thus applying Lemma \ref{lemmaLambda2}-e), yields
$$\nu(X) =   \frac{ \#{\mathcal F}_{n, l}(\{A_j\})}{\#{\mathcal A}_{n+l}} = \frac{1}{2^{ nl } \cdot \#{\mathcal A}_{n}}. $$
Combining this with (\ref{eqn52}) yields  $H ({\mathcal Q}_l) =  \log (\#{\mathcal A}_n) + nl \cdot \log 2. $
Finally, substituting in   Equality  (\ref{eqn52b}), we conclude

$$ h({\mathcal A}_n, \nu):=  \lim_{l \rightarrow + \infty} \frac{1}{l} H\Big ( {\mathcal Q}_l, \nu\Big) = n \log 2.$$

Combining with  \eqref{eqn52a} yields
$h_{\nu}(\Phi)\geq n \log 2$, for all $n \geq 1$; hence $h_{\nu}(\Phi)= + \infty.$
 \end{proof}

\begin{proof}[Proof of Lemma \ref{LemmaMain}]
As proved in Lemmas \ref{lemmaNuInvariante}, \ref{lemmaNuMixing} and \ref{lemmaNuEntropyInfty}, the probability measure $\nu$ constructed by equality (\ref{eqnCosntructionOfNu}) is $\Phi$-invariant, mixing and has infinite metric entropy, as required.
\end{proof}

\section{Periodic Shrinking Boxes}   \label{sectionPeriodicShrinkingBoxes}

In this section we will prove Theorems \ref{Theorem1} and \ref{Theorem3} for  $m \ge 2$. The proofs are based  on the properties of the models proved in the previous sections, and on the existence of the periodic shrinking boxes which we construct here.

Throughout this section we consider $m \ge 1$, unless the condition $m \ge 2$ is explicitly stated.

\begin{definition} \em \label{DefinitionPerShrBox}
{\bf (Periodic shrinking box)}
Let $f \in C^0(M)$ and $K \subset M$ be a box.
We call  $K $   \em periodic shrinking with period $p \geq 1$ \em for $f$, if
 $K, \ f(K), \ f^2(K), \ \ldots, f^{p-1}(K)$   are pairwise disjoint, and
 $f^p(K) \subset {{\rm int}}(K)$.
 If so, we call
  $f^p|_{K} : K \rightarrow \mbox{ int} (K)$ the \em return map. \em
\end{definition}

Recall that the manifold $M$ is compact. This assumption is important to obtain the following Lemmas \ref{Lemma1} and  \ref{Lemma1bis}.   We will construct periodic shrinking boxes whose return maps are homeomorphisms onto their images. Although this latter condition is unnecessary for the construction of the periodic shrinking boxes, it will be used later in the proofs of Lemmas \ref{Lemma2} and \ref{Lemma2b} where the return maps must be topologically conjugated to model maps.

\begin{lemma}
\label{Lemma1} For any $\delta>0$, there exists an open and dense set of maps $f \in C^0(M)$ that have a periodic shrinking box $K_f$ with ${{\rm diam}}(K_f)< \delta$.  For a dense set of $f \in C^0(M)$ the return map to $K_f$ is one to one.
\end{lemma}
The proof of this lemma  uses the following  technical result.

\begin{lemma}
\label{Lemma1bis}
Let $f \in C^0(M)$ and $x_0 \in M$. For all $\e>0$, there exists $g \in C^0(M)$ and a neighborhood $H$ of $x_0$ such that $\|g-f\|_{C^0} < \e$, $g|_H$ is a homeomorphism  onto its image, and coincides with $f$ off a  neighborhood of $x_0$.
\end{lemma}
\begin{proof} Since the assertion is of local character  we may assume that $M=\mathbb{R}^n$. Composing with a  translation we may also assume that $x_0= f(x_0)=0$. Let $0<\delta <\epsilon $ be so small that the ball $\|x\|< \delta$ is mapped under $f$ to a set of diameter smaller than $\epsilon$. Let $\lambda \colon \mathbb{R}^n \to [0,1]$ be a continuous function such that $\lambda (x)=0$ if $\|x\| \leq \delta/2$ and $\lambda (x)=1$ if $\|x\| \geq \delta$. We define $g$ by the formula $g(x) := \lambda(x) f(x) + (1 - \lambda(x))x$ if $\|x\| \leq \delta$ and $g(x) = f(x)$ if $\|x\| \geq \delta$.
\end{proof}

\begin{proof}[Proof of Lemma \ref{Lemma1}]
 According to Definition \ref{DefinitionPerShrBox}, the same periodic shrinking box $K_f$ for $f$ is also a periodic shrinking box with the same period for all $g \in C^0(M)$ near enough $f$, proving the openness assertion.

We turn to the denseness assertion. Let   $f \in C^0(M)$ and $\e>0$. We will  construct $g \in C^0(M)$ and a periodic shrinking box $K_g$ for $g$ with ${{\rm diam}}(K)< \delta$,  such that  $\|g-f\|_{C^0} < \e.$
We suppose $\delta>0$   to be smaller than the $\e$-modulus of continuity of $f$.

By the Krylov-Bogolyubov theorem invariant measures exist (recall that the manifold $M$ is compact), and thus by the Poincar\'{e} Lemma, there exists a recurrent point $x_0 \in M$ for $f$. First assume that $x_0 \not \in \partial M$. So, there exists a   box $   B  \subset M$   with  ${{\rm diam}}( B)  <  \delta$   such that $x_0 \in  {{\rm int}}( B)$.
Since $x_0$ is a recurrent point, there exists a smallest $p \in \N$ such that $f^p(x_0) \in {{\rm int}}(B)$.
{  Taking ${B}$ slightly smaller if necessary, we can assume that $f^j(x_0) \not \in  B$ for all $j=1,2, \ldots, p-1$.}
 So, there exists a small compact box  $ U \subset {{\rm int}}(B) $  as in   Figure \ref{FigureLemma1},  such that  $x_0 \in {{\rm int}}( U), $    the sets $U, f( U), \ldots, f^{p-1}( U)  $ are pairwise disjoint, and
$f^p( U) \subset {{\rm int}}(B)$.

\begin{figure}
 \begin{center}
\includegraphics[scale=.5]{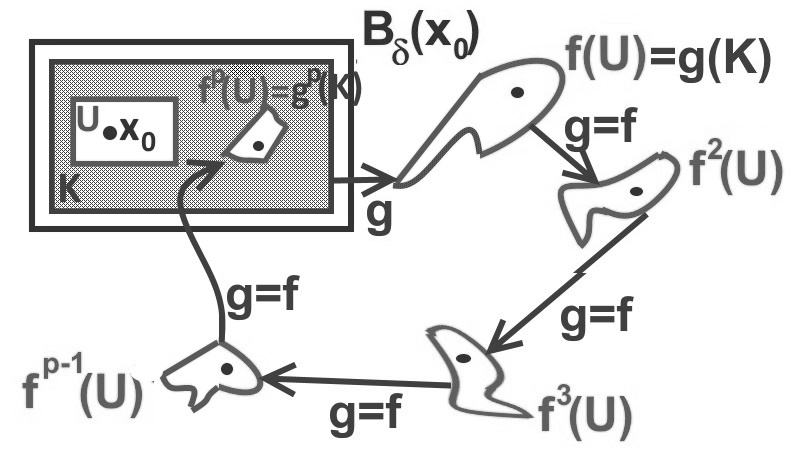}
\caption{Construction of   $g$ near $f$ with a periodic shrinking box $K$ for $g$.\label{FigureLemma1}}
\end{center}
\end{figure}

 Since $ U, f^p( U) \subset {{\rm int}}( B)$, there exists a box ${K}$ such that  $ U, f^p( U) \subset {{\rm int}}({K}) \subset {K} \subset {{\rm int}}( B),$
 and there exists   a homeomorphism  $\psi:  B \rightarrow  B$  such that
 $\psi(x)= x$ for all $x \in \partial { B}$, and $\psi (K) =  U$.

Finally, we construct $g \in C^0(M)$ as follows:
$$g(x):= f(x), \qquad \forall \ x \not \in  B, \ \ \ \ g(x) = f \circ \psi (x), \qquad \forall \ x \in  B.$$
 By construction, $K$ is a periodic shrinking box of $g$, say $K=K_g$; by the choice of  $\delta$ we have $\|g-f\| < \e$.

 Now, let us study the case for which $M$ is a compact manifold with boundary  and all the recurrent points of  $f$ belong to  $\partial M$. Choose one  such recurrent point $x_0 \in \partial M$. For any $\delta>0$, there exists a compact box $   B  \subset M$, with  ${{\rm diam}}( B)  \leq  \delta$  such that  $x_0 \in  \partial M \cap  B$. Since $x_0$ is recurrent, there exists a smallest natural number $p \geq 1$ such that $f^p(x_0) \in  B$. But $f^p(x_0)$ is also recurrent. So,  $f^p(x_0) \in \partial M \cap  B$.
The previous proof does not work as  is.
To overcome the problem, we choose  a new point $\widetilde x_0 \neq x_0$, near enough $x_0$, such that  $\widetilde x_0 \in {{\rm int}}( B) \setminus \partial M$ and $f^p(\widetilde x_0) \in  B.$  By applying Lemma \ref{Lemma1bis} and slightly perturbing $f$, if necessary, we can assume that the restriction of $f $ to a small neighborhood of $\widetilde x_0$  is  a local homeomorphism onto its image. Hence,  $f^p(\widetilde x_0) \in {{\rm int}}( B) \setminus \partial M.$
To conclude, we repeat the construction of $g$ and $K_g$ above
 replacing the recurrent point $x_0$ by $\widetilde x_0$.

Now, let us show that we can construct densely for  $g \in C^0(M)$ a periodic shrinking box $K_g$ such that the return map $g^p|_{K_g}$ is a homeomorphism onto its image. We repeat the beginning of the proof, up to the construction  of the points
 $x_0, f (x_0),  \ldots, f^p(x_0)$ such that
$ x_0, f^p(x_0) \in  {{\rm int}}( B)$ and $ f^j(x_0)     \not \in  B.$
Apply Lemma \ref{Lemma1bis}, slightly perturb $f$, if necessary,  inside  small open neighborhoods
$W_0, W_1, \ldots, W_{p-1}$ of the points $x_0, f (x_0),  \ldots, f^{p-1}(x_0)$ respectively, so that
$f|_{\overline W_i} $ is a homeomorphism onto its image for all $i=0, 1, \ldots, p-1$.
Finally, construct the box  $U$  (Figure \ref{FigureLemma1}),  but small enough so $f^j(U) \subset W_j$ for all
$j= 0,1,\ldots, p-1$,  and repeat the construction of $K=K_g$ and $g$ as above.
 \end{proof}
\begin{remark}
\label{RemarkPerturbInsideB}
\em Note that to obtain the dense property in the proof  of the first sentence  of  Lemma \ref{Lemma1},  we only need to perturb the map $f$ in the interior of the initial box $ {B}$ with diameter smaller than $\delta$.
\end{remark}

The following lemma is the homeomorphism version of  Lemma \ref{Lemma1}.

\begin{lemma}
\label{Lemma1b}
For any $\delta>0$, there exists an open and dense set of maps $f \in {{\rm Hom}}(M)$ that have a periodic shrinking box $K$ with ${{\rm diam}}(K)< \delta$.
\end{lemma}

\begin{proof}
The proof of Lemma \ref{Lemma1} also works in the case that $f \in {{\rm Hom}}(M)$: in fact,    the $\e$-perturbed map $g$ constructed there  is a homeomorphism, and to obtain  $\|g-f\|_{{{\rm Hom}}(M)} < \e$ it is enough to reduce $\delta>0$  to be smaller than the $\e$-continuity modulus of $f$ and $f^{-1}$.
\end{proof}

\begin{remark}
\label{RemarkRecurrentPointIsPeriodic}

\em In the proof   of  Lemmas \ref{Lemma1} and \ref{Lemma1b}, if the starting recurrent point $x_0$ were a periodic point of period $p$, then the  periodic shrinking  box $K$  so constructed would  contain $x_0$ in its interior and have the same period $p$.
\end{remark}

\begin{lemma}
\label{Lemma2}  Assume  $m \geq 2$. Fix $\delta>0$ and $\Phi \in {\mathcal H} \cap \rm{Emb}(D^m)$ (recall Definition \ref{DefinitionModel}).
Each generic map   $f \in C^0(M)$ has a periodic shrinking box $K$ with ${{\rm diam}}(K)< \delta$ such that the return map $f^p|_K$ is topologically conjugated to a model map in ${\mathcal H}_\Phi$ (recall Definition \ref{Definition HsubPhi}). \em
\end{lemma}
\begin{proof} Let $K \subset M$ be a periodic shrinking box for $f$. Fix a homeomorphism   $\phi: K \rightarrow D^m$.

To prove the $G_{\delta}$ property, assume that $f \in C^0(M)$ has a periodic shrinking box $K$ with ${{\rm diam}}(K)< \delta$, such that
$\phi \circ f^p|_K \circ \phi^{-1} \in {\mathcal H}_{\Phi}$ (recall Definition \ref{Definition HsubPhi} and Lemma \ref{LemmaModelHEmbNonempty}).   From Definition \ref{DefinitionPerShrBox}, the same box $K$ is also periodic shrinking with period $p$  for all  $g \in {\mathcal N}$, where ${\mathcal N}\subset C^0(M)$ is an open neighborhood of $f$.
From Lemma \ref{LemmaModelHEmbNonempty}, ${\mathcal H}_{\Phi}$   is a nonempty $G_{\delta}$-set in $C^0(D^m)$, i.e., it is the nonempty countable intersection of open families ${\mathcal H}_n \subset C^0(D^m)$. We define
$${\mathcal V}_n:= \{g \in {\mathcal N} \colon \phi \circ g^p|_{K} \circ \phi^{-1} \in {\mathcal H}_n\}. $$
Since the restriction to  $K$ of a continuous map $g$, and the composition of continuous maps,  are continuous operations in $C^0(M)$, we deduce that  ${\mathcal V}_n $ is an open family in $  C^0(M)$. Besides
\begin{equation}
\label{eqn13}
\phi \circ g^p|_K \circ \phi^{-1}   \in {\mathcal H} = \bigcap_{n \geq 1} {\mathcal H}_n \ \  \mbox{ if  } \ \  g \in \bigcap_{n \geq 1}  {\mathcal V}_n    \subset C^0(M) . \end{equation}
In other words,  the set of   maps $g\in C^0(M)$   that have   periodic shrinking box $K$ with ${{\rm diam}}(K)< \delta$, such that the return map  $g^p|_K$ coincides, up to a conjugation, with a model map in ${\mathcal H}_\Phi$, is a $G_{\delta}$-set in $C^0(M)$.

To show the denseness fix $f \in C^0(M)$ and $\e >0$. Applying Lemma \ref{Lemma1},  it is not restrictive to assume that $f$ has a periodic shrinking box $K$ with ${{\rm diam}}(K) < \min\{\delta, \e\} $, such that   $f^p|_K$ is a homeomorphism onto its image. We will construct $g \in C^0 (M)$ to be $\e$-near $f$ and  such that  $\phi \circ g^p|_K \circ \phi^{-1}   \in {\mathcal H}$.

Choose a box $W $ such that $ f^{p-1}(K) \subset {{\rm int}}(W)$. If $p \geq 2$, take $W$ disjoint with $f^j(K)$ for all $j\in \{0,1, \ldots, p-2\}$ (Figure \ref{FigureLemma2}).
Let us see that we can assume that $W$ has  an arbitrarily small diameter. It is enough to prove that $f$ can be chosen such that $ f^{p-1}(K)$ has an arbitrarily small diameter. In fact, in the construction of $f$ in the proof of Lemma \ref{Lemma1}, we can choose the box $U$ (see Figure \ref{FigureLemma1}), after choosing $K$, as small as needed. So, we choose $U$ small enough such that the  $(p-1)$-th. image  of $U$  by the map before the perturbation, has a small diameter. (Note that we do not change $p$). After that, we construct the perturbed map, which we are calling $f$ here, as in the proof of Lemma \ref{Lemma1}: the image $f^{p-1}(K)$ of the new map $f $ coincides with the $(p-1)$-th. image of $U$ by the map before the perturbation (Figure \ref{FigureLemma1}). So, it has an arbitrarily small diameter, as required.

\begin{figure}
 \begin{center}
\includegraphics[scale=.5]{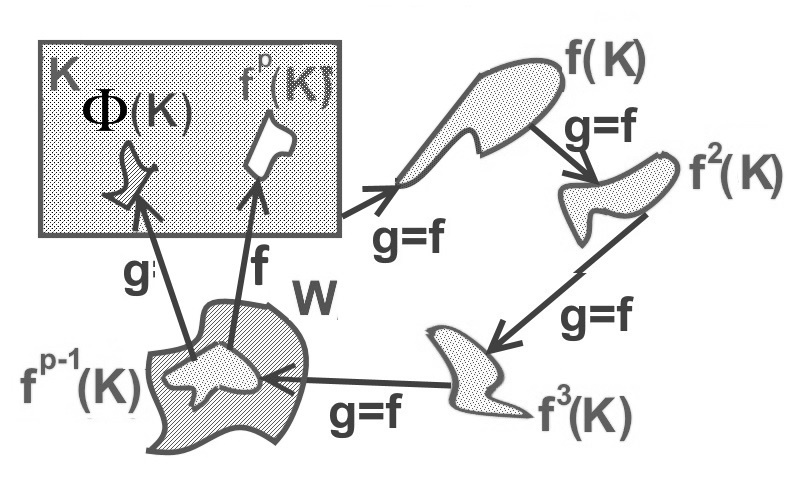}
\caption{Perturbation $g$ of $f$ such that $g^p|_K = \Phi$. \label{FigureLemma2}
}
\end{center}
\end{figure}

To construct $g \in C^0(M)$  (see Figure \ref{FigureLemma2}) we consider the chosen  $\Phi \in {\mathcal H}$ in the hypothesis, and let $g(x) := f(x)  $ if $ x \not \in W$~and
$$g(x) := \phi^{-1} \,  \circ \, \Phi \,  \circ  \, \phi\, \circ \, (f^p|_{K})^{-1} \circ      f (x), \qquad \forall \ x \in f^{p-1}(K).$$
This defines a continuous map $g: f^{p-1}(K) \cup (M \setminus W) \rightarrow M$ such that     $|g(x) - f(x)| < {{\rm diam}}(K) < \e $ for all $x \in f^{p-1}(K) \subset W  $ and $ g(x) = f(x) $ for all $ x \in M \setminus W.$
Applying the  Tietze Extension Theorem,   there exists a continuous extension of $g$ to the whole compact box $W$, hence to  $M$, such that $\|g - f\|_{C^0} < \e.$
Finally, by construction we obtain\ $$\displaystyle g^p|_K = g|_{f^{p-1}(K)} \circ f^{p-1}|_{K}=
\phi^{\text{-}1} \,  \circ \, \Phi \,  \circ  \, \phi\, \circ \, (f^p|_{K})^{\text{-}1} \circ      f  \circ f^{p-1}|_{K}= \phi^{\text{-}1} \,  \circ \, \Phi \,  \circ  \, \phi. \eqno{\qedhere}$$
  \end{proof}

\begin{lemma}
\label{Lemma2b}  Let $\delta>0$. Choose and fix $\Phi \in {\mathcal H} \cap \rm{Emb}(D^m)$.
A generic homeomorphism  $f \in {{\rm Hom}}(M)$ has a periodic shrinking box $K$ with ${{\rm diam}}(K) < \delta$,
such that the return map $f^p|_K$ is topologically conjugated to a model  embedding in ${\mathcal H}_{\Phi}$.
\end{lemma}
\begin{proof}
 We repeat the proof of the $G_{\delta}$-set property of Lemma \ref{Lemma2}, using ${\mathcal H} \cap {{\rm Emb}}(D^m)$ instead of ${\mathcal H}$, and ${{\rm Hom}}(M)$ instead of $C^0(M)$.

To show the denseness fix  $f \in {{\rm Hom}}(M)$ and $\e >0$. Applying Lemma \ref{Lemma1}, it is not restrictive to assume that $f$ has periodic shrinking boxes of arbitrarily small diameters. Let $\delta \in (0,\e)$ be smaller the the $\e$-modulus of continuity of $f$ and $f^{-1}$. Consider a periodic shrinking box $K$ with $\mbox{diam}(K) < \delta$ (Lemma \ref{Lemma1b}).
Fix a homeomorphism  $\phi: K \rightarrow D^m$. We will construct $g \in {{\rm Hom}} (M)$ to be $\e$-near $f$ in ${{\rm Hom}}(M)$, with  $\phi \circ g^p|_K \circ \phi^{-1} =\Phi \in {\mathcal H} \cap {{\rm Emb}}(D^m)$.

From Definition \ref{DefinitionPerShrBox} we know that the boxes $K, f(K), \ldots, f^{p-1}(K) $ are pairwise disjoint and that $f^p(K) \subset {{\rm int}}(K)$.   Denote $W := f^{-1}(K)$. Since $f$ is a homeomorphism, we deduce that $W$ is a box as in Figure \ref{FigureLemma2}, such that
 $ W \cap f^j( K) = \emptyset $ for all $  j  = 0,1, \ldots, p-2$ if $ p \geq 2,$ and $  f^{p-1}(K) \subset {{\rm int}}(W)$.  Since ${{\rm diam}}(K) < \delta$  we have
 ${{\rm diam}}(W) < \e.$

  Consider $\phi \circ f^p|_K \circ \phi^{-1} \in {{\rm Emb}}(D^m)$.  { Applying Lemma \ref{LemmaConstruccionModelPsifisPhi},  }  there exists a homeomorphism  $\psi: D^m \rightarrow D^m$   such that
 $$\psi|_{\partial D^m} = \mbox{id}|_{\partial D^m}, \ \ \ \ \ \psi \circ \phi \circ f^p|_K \circ \phi^{-1} = \Phi \in {\mathcal H} \cap {{\rm Emb}}(D^m).$$
 So, we can construct $g \in {{\rm Hom}}(M)$ such that $g(x) := f(x) $ for all $ x \not \in W$, and $ g(x) := \phi^{-1}\circ \psi \circ \phi \circ          f (x) $ for all $x \in W.$
Since $\psi|_{\partial D^m} $ is the identity map, we obtain $g|_{\partial W} = f|_{\partial W}$. Thus, the above equalities define a continuous map $g : M \rightarrow M$. Moreover $g$ is invertible because     $g|_W : W \rightarrow K$  is a composition of homeomorphisms, and $g|_{M \setminus W} = f|_{M \setminus W}: M \setminus W  \rightarrow M \setminus K$  is also a homeomorphism. So, $g \in {{\rm Hom}}(M)$.  Moreover, by construction we have    $|g(x) - f(x)| < {{\rm diam}}(K) <  \e $ for all  $  x \in W , $ and $  g(x) = f(x) $ for all $x \not \in  W. $ Also the inverse maps satisfy
 $|g^{-1}(x) - f^{-1}(x)| < {{\rm diam}}(f^{-1}(K)) =  {{\rm diam}}(W)  < \e$ for all $ x \in K, $ and $g^{-1}(x) = f^{-1}(x)$ for all $ x \not \in  K.  $
Therefore   $\|g -f\|_{{{\rm Hom}}} < \e.$

Finally, let us check that    $g^p|_K$ is topologically conjugated to $\Phi$:
  $$g^p|_K = g|_{f^{p-1}(K)} \circ f^{p-1}|_ K  = g|_{W} \circ f^{p-1}|_ K =$$
  $$   \phi^{-1}\circ \psi \circ \phi \circ          f \circ f^{p-1}|_K = $$ $$ \phi^{-1}\circ (\psi \circ \phi \circ          f ^p|_K \circ  \phi^{-1}) \circ \phi   =  \phi^{-1} \circ \Phi \circ \phi.\hfill \eqno{\qedhere}$$
  \end{proof}

  \begin{remark}
  \label{RemarkPerturbInsideK}
  \em In the   proof of the dense property in  Lemmas \ref{Lemma2} and \ref{Lemma2b}, once a periodic shrinking box $K$ is constructed with period $p \geq 1$, we only need to perturb the map $f$  inside  $W \cup \bigcup_{j=0}^{p-1} f^j(K)$, where $W= f^{-1}(K)$ if $f$ is a homeomorphism, and ${{\rm int}}(W) \supset f^{p-1}(K)$ otherwise. In both cases, by reducing the set $U$ of  Figure \ref{FigureLemma1}  from the very beginning,  we  can construct $W$ such that  ${{\rm diam}}(W)< \e$ for a previously specified small $\e>0$.
  \end{remark}

\begin{proof}[Proof of Theorems \ref{Theorem1} and \ref{Theorem3}]
 From Lemmas \ref{Lemma2} and \ref{Lemma2b}, a generic map $f \in C^0(M)$ and also a generic $f \in {{\rm Hom}}(M)$, has a periodic shrinking box $K$ such that the return map
  $f^p|_K: K \rightarrow {{\rm int}}(K) $ is conjugated to a model map $ \Phi \in {\mathcal H}.$  We consider the homeomorphism $ \phi ^{-1}:  K \to D^m$
 such that
  $\phi^{-1} \circ f^p \circ \phi   = \Phi \in {\mathcal H}$.
 Lemma \ref{LemmaMain} states that  every map $\Phi \in {\mathcal H}$ has a $\Phi$-invariant mixing measure $\nu$ with infinite metric entropy for $\Phi$. Consider the   push-forward  measure $\phi _*  \nu$,  defined by
 $(\phi_*  \nu) (B):= \nu (\phi^{-1}(B \cap K))$ for all the Borel sets $B \subset M$.
By construction,  $\phi_*  \nu$ is supported  on $K \subset M$.
 Since $\phi$ is a conjugation between $\Phi$ and $f^p|_K$, the push-forward measure $\phi_* \nu$ is $f^p$-invariant and mixing  for $f^p$  and moreover $h_{\phi_* \nu}(f^p) = + \infty$.

 From $\phi_*\nu$, we will construct an $f$-invariant and $f$-ergodic measure $\mu$  supported on $\bigcup_{j=0}^{p-1} f^j(K)$, with infinite metric entropy for $f$. Precisely, for each Borel set $B \subset M$, define
 \begin{equation} \label{eqn15} \mu(B):= \frac{1}{p} \sum_{j= 0}^{p-1} (f^j)_*   (\phi_* \nu) (B \cap f^j(K)).\end{equation}
 Applying Equality (\ref{eqn15}), and the fact that $\phi_* \nu$  is $f^p$-invariant and  $f^p$-mixing,  it is standard to check that $\mu$ is $f$-invariant and $f$-ergodic.
From the convexity of the metric entropy function,  we deduce that $$h_{\mu}(f^p) = \frac{1}{p}\sum_{j=0}^{p-1} h_{(f^j)_* (\phi_* \nu)} (f^p) = + \infty.$$
Finally, recalling that  $ h_{\mu}(f^p) \leq p \,  h_{\mu}(f)$ for any $f$-invariant measure $\mu$ and any natural number $p \geq 1$, we conclude that $h_{\mu}(f)= + \infty$.
 \end{proof}

 \section{Good sequences of periodic shrinking boxes} \label{sectionCleanSequences}
\noindent In this section we  prove  Theorems \ref{Theorem2} and \ref{Theorem4}. Throughout this section we assume that ${{\rm dim}}(M) \geq 2$. In the case that $M$ is  a 1-dimensional manifold,  Theorem \ref{Theorem2} can be proven repeating the proof of the 2-dimensional case, after  replacing Definition \ref{DefinitionModel}   by Definition \ref{definitionModelDim1}.

\begin{definition} \em
\label{definitionCleanSequence}  Let $f \in C^0(M)$ and let
 $K_1, K_2, \ldots, K_n, \ldots$  be a sequence of periodic shrinking boxes for $f$.
We call  $\{K_n\}_{n \ge 1}$  \em good \em if it has the following properties (see Figure \ref{FigureLemma3}):

\noindent $\bullet$ $\{K_n\}_{n \geq 1}$ is composed of pairwise disjoint boxes.

\noindent $\bullet$ There exists a natural number $p \geq 1$, independent of $n$, such that $K_n$ is a periodic shrinking box for $f$ whose period $p_n$ is a multiple of $p$.

\noindent $\bullet$ There exists a sequence $\{H_n\}_{n \geq 0}$ of periodic shrinking boxes, all with period $p$, such that $K_n \cup H_n \subset  H_{n-1} , \ K_n \cap H_n = \emptyset$ for all $n \geq 1$, and ${{\rm diam}}(H_n) \rightarrow 0$ as $n \rightarrow + \infty$.
 \end{definition}

\noindent {\bf Remark.}   Definition \ref{definitionCleanSequence} implies  that $ \bigcap_{n \geq 1} H_n = \{x_0\}$,
where $x_0$ is periodic with period $p$.  Furthermore, for any $j \ge 0$ we have
$$d(f^j(K_n), f^j(x_0)) \leq {{{\rm diam}}}(f^j(H_{n-1})) \leq \displaystyle \max_{0 \leq k \leq p-1} {{{\rm diam}}}(f^k(H_{n-1})) \stackrel{\scriptscriptstyle n \to \infty}{\to} 0,$$
 and thus
\begin{equation} \label{eqnUniformLimitCleanSequence}  \lim_{n \rightarrow + \infty}\sup_{j \geq 0} d(f^j(K_n), f^j(x_0))= 0.\end{equation}

We will construct a good sequence of   periodic shrinking boxes for maps that are arbitrarily near a given $f$. We start by constructing the zeroth level boxes as follows:

\begin{lemma} \label{lemmaLevel0}
Let $f \in C^0(M)$ (resp.\ $f \in {\rm Hom}(M)$) and $\e, \delta >0$. Then, there exists  $g_1 \in C^0(M)$  (resp.\ $g_1 \in {\rm Hom}(M)$ ),  periodic shrinking boxes $H_0$ and $ K_1$ for $g_1$  with periods $p$ and $p_1$  respectively, where $p_1$ is multiple of $p$, and a periodic point $x_0 \in {{{\rm int}}}(H_0)$ for $g_1$ such  that  $K_1 \subset H_0\setminus \{x_0\}$,  $$ g_1^{p_1}|_{K_1} \mbox { is topologically conjugated to }\Phi_1 \in {\mathcal H}  \ \ \mbox{ and } $$ $${{{\rm diam}}}(H_0)<   {\delta} , \ \ \ \|g_1 - f\| <  \frac{\e}{2}.$$
\end{lemma}

\begin{figure}
 \begin{center}
\includegraphics[scale=.5]{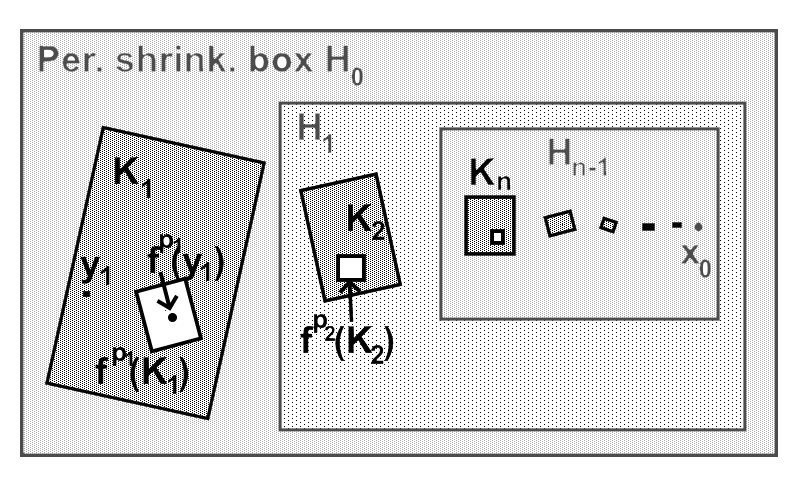}
\caption{Construction of a good sequence of periodic shrinking boxes. \label{FigureLemma3}}
\end{center}
\end{figure}

\begin{proof}
A generic  map  $\widetilde f\in C^0(M)$ (resp.\   $\widetilde f\in {{\rm Hom}}(M)$) has a periodic shrinking box $H_0$
with   period $p \geq 1$,  such that ${{{\rm diam}}}(H_0)<   {\delta} $ and $\widetilde f^p|_{H_0} $ is conjugate to a model map  $ \Phi\in {\mathcal H}$
(Lemmas \ref{Lemma2}, resp.\  \ref{Lemma2b}). Fix such an $\widetilde f$ in the $(\e/6)$-neighborhood of $f$. The same box $H_0$ will be a shrinking periodic box for the map $g_1$ to be constructed.

Since $\widetilde f^p: H_0 \rightarrow {{{\rm int}}}(H_0) \subset H_0$ is continuous,  by the Brouwer Fixed Point Theorem
there exists a periodic point $x_0 \in {{{\rm int}}}(H_0)$ of period  $p$.
 Lemma \ref{LemmaMain} and the argument at the end of the proofs of Theorems \ref{Theorem1} and \ref{Theorem3},   show that the map $\widetilde f$ has an ergodic measure $\mu$ supported on $\bigcup_{j=0}^{p-1} \widetilde f^j(H_0)$ such that $h_{\mu}(\widetilde f) = + \infty$. Therefore, by Poincar\'{e} Recurrence Lemma, there exists some recurrent point $y_1   \in {{{\rm int}}}(H_0) $ for $\widetilde f$. We can choose  such recurrent point $y_1 \neq x_0$ (see Figure \ref{FigureLemma3}) because $\mu$ is not supported on the orbit of the periodic point $x_0$ (recall that $\mu$ has infinite entropy and by construction its support is a perfect set).

Choose  $\delta_1>0$ small enough and
construct a box $B_1$ such that $y_1 \in {{\rm int}}(B_1)$,    ${{\rm diam}}(B_1) < \delta_1$,  the $\widetilde f$-orbit of  $x_0 $ (which is finite) does not intersect the finite piece of the $\widetilde f$-orbit of   $ B_1 $ (until the first iterate of $y_1$ is in $H_0$)   and $  B_1 \subset {{{\rm int}}} (H_0).   $  We repeat the proofs of the dense property of
 Lemmas \ref{Lemma1} and \ref{Lemma1b}, using the recurrent point $y_1$ instead of $x_0$, and the box $B_1$ instead of $B$ (see Figure \ref{FigureLemma1}).
We  deduce that  there exists an $(\e/6)$-perturbation ${\hat f}$ of $\widetilde f$, and a periodic shrinking box $K_1 \subset B_1$ for ${\hat f}$, with some period  $p_1\geq p$  (see Figure \ref{FigureLemma3}). Moreover, ${\hat f}$ coincides with $ \widetilde f$ in $M \setminus  {{{\rm int}}}(B_1)$
 (recall Remark  \ref{RemarkPerturbInsideB}).   Therefore, the same periodic point $x_0$ of $\widetilde f$ survives for ${\hat f}$. Besides, by the openness of the existence of the periodic shrinking box $H_0$, the same initial box $H_0$ is still periodic shrinking with period $p$ for ${\hat f}$, provided that ${\hat f}$ is near enough $\widetilde f$. So, the compact sets of the family $\{{{\hat f}}^j(H_0)\}_{j= 0,1,\ldots, p-1} $ are pairwise disjoint, and ${{\hat f}}^p(H_0) \subset {{{\rm int}}}(H_0)$. This implies that the period $p_1$ of the new periodic shrinking box $K_1$  for ${\hat f}$  is a multiple of $p$.

 Now, we apply the proofs of the dense property of
 Lemmas \ref{Lemma2} and \ref{Lemma2b}, using the shrinking box $K_1$ instead of $K$ (see Figure \ref{FigureLemma2}).
We  deduce that  there exists an $(\e/6)$-perturbation $g_1$ of ${\hat f}$, such that $K_1$ is still a periodic shrinking box  for $g_1$  with  the same period  $p_1$,  but moreover, the return map $g_1^{p_1}|_{K_1}$ is now topologically conjugated to $\Phi_1 \in {\mathcal H} $.

Consider a box  $W_1 $ satisfying ${\hat f}^{p_1-1}(K_1) \subset W_1 \subset K_1$,
  small enough so its ${\hat f}$-orbit is disjoint from the ${\hat f}$-orbit of the periodic point $x_0$.
  Taking into account Remark \ref{RemarkPerturbInsideK}, we can construct $g_1$ to coincide  with $ {\hat f}$ in the complement of $W_1 \bigcup \Big( \bigcup_{j=0}^{p_1-1}{{\hat f}}^j(K_1)   \Big)$.   Then, if $g_1$ is sufficiently near $ {\hat f}$, the point $x_0$  is still  periodic   of period $p$ for $g_1$, and besides $H_0  $  is still a  periodic shrinking box  of period $p$ for $g_1$ (recall that such property is open).
Finally,
$$\|g_1 - f\| < \|g_1 - {\hat f}\| + \|{\hat f} - \widetilde f\|+  \|\widetilde f - f\| < \frac{\e}{6} +\frac{\e}{6} + \frac{\e}{6} = \frac{\e}{2}. \eqno{\qedhere}$$
 \end{proof}

Assume that we  have constructed the $j$-th  level of periodic shrinking boxes  for all  $0 \le j \leq n-1$ of a good sequence. We will construct the   periodic shrinking boxes of the $n$-th  level  by perturbing the given map once more. Let us first define the following family  of maps.

\begin{definition}
\label{definitionUpToLeveln}
\em Fix $\delta>0$, and let $p,n$ be natural numbers such that $p,n \geq 1$. We denote by ${\mathcal G}_{p, n, \delta} \subset C^0(M)$ the family of all the maps $g  \in C^0(M)$ such that there exists $n$   boxes $K_1,   \ldots, K_n$   satisfying the following properties:

\noindent $\bullet$ $\{K_j\}_{1 \leq j \leq n }$ is composed of pairwise disjoint boxes.

\noindent $\bullet$ For all $1 \leq j \leq n$ the box $K_j$ is a periodic shrinking  for $g$ with   period $p_j$ that is a  multiple of $p$, and
$$g^{p_j}|_{K_j} \mbox{ is topologically conjugated to }\Phi_j \in {\mathcal H}.$$

\noindent $\bullet$ There exists  a sequence $\{H_j\}_{0 \leq j \leq n-1}$ of  periodic shrinking boxes   for $g$, all of period $p$, and a periodic point $x_{n-1}\in{{\rm int}}(H_{n-1})$ of period $p$, such that $K_j \cup H_j \subset  H_{j-1} , \ K_j \cap H_j = \emptyset$ for all $1 \leq j \leq n-1 $, $K_n \subset H_{n-1}\setminus {x_{n-1}}$ and ${{\rm diam}}(H_j) < \delta/2^j$ for all $0 \le j \le n-1$ (see Figure
\ref{FigureLemma3}).
\end{definition}

\begin{lemma} \label{lemmaLeveln} Fix $\e>0$, $\delta>0$ and the
    natural numbers $n,p \geq 1$. Assume that    $g_n \in {\mathcal G}_{p,n, \delta}$  or $g_n \in {\mathcal G}_{p,n, \delta} \cap {{\rm Hom}}(M)$ .

Then, there exists an $\e/2^{n+1}$-perturbation $g_{n+1}$ of $g_n$,  such that $g_{n+1} \in {\mathcal G}_{p,n+1, \delta}$  or $g_{n+1} \in {\mathcal G}_{p,n+1, \delta} \cap {{\rm Hom}}(M)$,   respectively.

Moreover,
 for all $j=1, \ldots, n$ the same boxes $K_1, K_2, \ldots, K_n$ and $H_0, H_1, \ldots, H_{n-1}$ are shrinking periodic for the new map $g_{n+1}$  and for the given map $g_n$, with the same periods, and
$$ g_n^i|_{K_{j }}= g_{n+1}^i|_{K_{j }}\ \ \forall \ i= 1, \ldots, p_j. $$
\end{lemma}

\begin{proof} All the  perturbations  of  $g_n$ that we will construct  are sufficiently close to $g_n$  so that the same boxes $H_0, H_1, \ldots, H_{n-1}$ and $K_1, K_2, \ldots, K_n$ that are periodic shrinking for $ g_n$ are still periodic shrinking with the same periods for the perturbed maps. This is possible because the   periodic shrinking property of a box and its period, are open conditions.
Besides, we will only consider perturbations of $g_n$ that coincide with $g_n$ except in the interior of a finite number of boxes $B,W,$ etc, whose $g_n$-iterates, up to the  $(\max_{1\leq j \leq n}p_j)$-th iterate, are disjoint with all the boxes  of the family $\{g_n^i(K_j): 1 \leq j \leq n, \ 0 \leq i \leq p_j-1\}$.
Therefore, if such a perturbation $\widetilde g$ of $g_n$ is near enough $g_n$, then the iterates by $\widetilde g$ of the boxes $B, W,$ etc (where $\widetilde g$ differs from $g_n$) are still disjoint with the $g_n$-iterates of $K_j$. This implies that for all $1 \leq j \leq n$,
$$ g_n^i|_{K_{j }}= \widetilde g ^i|_{K_{j }}\ \ \forall \ i= 1, \ldots, p_j   \ \  \mbox{ and then} $$
$$ \widetilde g  ^{p_j}|_{K_{j }}=   g_n ^{p_j}|_{K_{j }}  \mbox{ is topologically conjugated to } \Phi_j \in {\mathcal H}. $$

Now let us perturb $g_n$ as above, in several steps, to construct the boxes $H_n$ and $K_{n+1}$.

 By   hypothesis,    $g_n$ has a periodic shrinking box $H_{n-1}$ of period $p$, a periodic point $x_{n-1} \in {{{\rm int}}} (H_{n-1})$ of  period  $p $,  and a periodic shrinking box $K_n \subset H_{n-1}\setminus \{x_{n-1}\}$ of period $p_n$,   multiple of $p$.  It also has periodic shrinking boxes $K_1, \ldots, K_{n-1}, K_n$ whose $g_n$-orbits are disjoint with the periodic orbit of $x_{n+1}$. So, we can construct a   box ${\widetilde B_{n}} \subset H_{n-1}$ containing the periodic point $x_{n-1}$ in its interior, whose $g_n$-orbit up to the $(\max_{1 \leq j \leq n} p_j)$-th iterate is disjoint from all the sets of the family $\{f^i(K_{j}): \ 1 \leq j \leq n, \ 0 \leq i \leq p_j-1\}$. Besides, we construct the box ${\widetilde B_{n}}$ such that ${{\rm diam}}({\widetilde B_{n}}) < \widetilde \delta/2^n$.
Repeating the proof of the density properties in Lemmas \ref{Lemma1} and \ref{Lemma1b} (putting $x_{n-1}$ instead of $x_0$),
 we construct an $\e/(3 \cdot 2^{n+1})$-perturbation $\widetilde g_n$ of $g_n$, near enough $g_n$
and a periodic shrinking box
$H_n \subset  {{{\rm int}}} (\widetilde B_{n})$
 for $\widetilde g_n$. Moreover, since $x_{n-1}$
 is a periodic point with period $p$ for $g_n$, the period of $H_n$ for $\widetilde g_n$ can be made equal to $p$
(see Remark \ref{RemarkRecurrentPointIsPeriodic}). By construction $H_n \subset \widetilde B_n \subset H_{n-1}$ is disjoint from $K_n$.
  To construct $\widetilde g_n$ we only needed to modify $g_n$ inside $\widetilde B_n$
(recall Remark \ref{RemarkPerturbInsideB}). Therefore, if $\widetilde g_n$ is near enough $g_n$, as  observed at the beginning, the same periodic shrinking boxes $H_0, H_1, \ldots, H_{n-1}$
and $K_1, K_2, \ldots, K_n$ of
$g_n$, are preserved for $\widetilde g_n$  with the same periods, and $\widetilde g_n$ coincides with $g_n$ on the $g_n$-orbit of the boxes $K_1, \ldots, K_n$.

Now, as in the proof of  Lemmas \ref{Lemma2} and \ref{Lemma2b}, we construct a new $\e/(3 \cdot 2^{n+1}) $-perturbation $  {\hat g}_n$ of $\widetilde g_n$, such that    $ {{\hat g}_n}^{p}|_{H_n}$ is conjugated  to a map in ${\mathcal H}$.  To construct $ {\hat g}_n$ we only need to modify $\widetilde g_n$ in    $  \widetilde W_n \cup \bigcup_{j=0}^{p-1} \widetilde g_n^{j} (H_n)$, where $ \widetilde W_n $ is a small neighborhood of $\widetilde g_n^{p-1}(H_n)$  (see Remark \ref{RemarkPerturbInsideK}). Since the $\widetilde g_n$-orbit of $H_n$ is disjoint from the $\widetilde g_n$ orbits of $K_j$ for all $1 \leq j \leq n$ (because $H_n$ and $K_j$ are disjoint periodic shrinking boxes for $\widetilde g_j$, we can choose   $W_n$ near enough  $\widetilde g_n^{p-1}(H_n)$ and ${\hat g}_n$ near enough $\widetilde g_n$ so ${\hat g}_n$ coincides with $\widetilde g_n$ on the orbit of the boxes $K_j$, as observed at the beginning.

  We conclude that the same shrinking boxes $K_1, \ldots, K_n; H_0, \ldots, H_{n-1}$ for $\widetilde g_n$ and $g_n$, are still periodic shrinking for ${\hat g}_n$, with the same periods and that ${{\hat g}_n}^{p_j}|_{K_j}= \widetilde g_n^{p_j}|_{K_j} $  which is conjugated to  $\Phi_j \in {\mathcal H}$ for all $j=1, \ldots, n$.

When modifying  $g_n$   to obtain $\widetilde g_n$  and   $  {\hat g}_n$, the periodic point $x_{n-1} \in {{\rm int}}(H_{n-1}) $ of period $p$ for $g_n$, may not be preserved as periodic for ${\hat g}_n$. But since $H_n \subset H_{n-1} \setminus K_n$ is a periodic shrinking box with period $p$ for $ {\hat g}_n$, by the Brouwer Fixed Point Theorem,  there   exists a    periodic point  $x_n \in {{{\rm int}}}(H_n)\setminus K_n$ for ${\hat g}_n$, with the same period $p$.

Since the return map ${{\hat g}_n}^p|_{H_n}$  is conjugated to a model map, there exists an ergodic measure $\mu$  with infinite entropy  for ${\hat g}_n$ (see Lemma \ref{LemmaMain}), supported on the ${\hat g}_n$-orbit of $H_n$. Therefore, there exists a recurrent point $y_n \in {{\rm {int}}}(H_n)$. We can choose such recurrent point $y_n \neq x_{n}$, because $\mu$  is not supported on the periodic orbit of $x_n$ (in fact, $\mu$ has infinite entropy).

We now argue as in the proof of Lemma \ref{lemmaLevel0}, (using ${\hat g}_n$, $H_{n}$ and $x_n$ in the role of $\tilde f$, $H_0$ and $x_0$) to construct an $\e/(3 \cdot 2^n)$- perturbation $g_{n+1}$ of ${\hat g}_n$,   and a box $K_{n+1} \subset H_n\setminus \{x_n\}$ that is periodic shrinking for $g_{n+1}$ of period $p_{n+1}$ which is a multiple of $p$, and such that
$g_{n+1}^{p_{n+1}}|_{ K_{n+1}}$ is topologically conjugated to a model map.

As observed at the beginning, if choosing $g_{n+1}$ near enough ${\hat g}_n$, the boxes $H_0, \ldots, H_n$ and $K_1, \ldots, K_n$ are still periodic shrinking for $g_{n+1}$ with the same periods, and
$$g_{n+1}^{p_{j}}|_{ K_{j}} =  {{\hat g}}_{n }^{p_{j}}|_{ K_{j}} = g _{n }^{p_{j}}|_{ K_{j}}$$ is topologically conjugated to a model map, for all $1 \leq j \leq n$.

By construction we have $g_{n+1}  \in {\mathcal G}_{p,n, \delta} $ and
$$\|g_{n+1}- g_n\| \leq \|g_{n+1}- {\hat g}_n\| + \|{\hat g}_n - \widetilde g_n\| + \|\widetilde g_n - g_n\| < 3 \cdot \frac{\e}{3 \cdot 2^{n+1}} = \frac{\e}{2^{n+1}},$$
 as required.
\end{proof}
\begin{definition}
\label{definitionGdelta} \em Fix $\delta >0$.
We denote by ${\mathcal G}_\delta \subset C^0(M)$ the family of all the   maps $g \in \bigcup_{p\geq 1}\bigcap_{n \geq 1} {\mathcal G}_{p,n, \delta}$ such that for all $n \geq 1$, the  boxes $H_0, \ldots, H_{n-1}$ and $K_1, \ldots, K_n$ of Definition \ref{definitionUpToLeveln} for $g$ as belonging to ${\mathcal G}_{p,n, \delta}$ coincide with the boxes   for $g$ as belonging to ${\mathcal G}_{p,n+1, \delta}$.
\end{definition}
\begin{lemma}
\label{lemmaDensity} Fix $\delta >0$.
The family $ {\mathcal G}_{ \delta}$ is dense in $C^0(M)$ and its intersection with ${\rm Hom}(M)$ is dense in ${\rm Hom}(M)$.
\end{lemma}
\begin{proof}
Let $f \in C^0(M)$ or $f \in {{\rm Hom}}(M)$, and $\e>0$. We will construct $g \in {\mathcal G}_{\delta}$ such that  ${\rm dist}(g, f) \leq \e$.

Applying Lemma \ref{lemmaLevel0}, there exists $p \geq 1$ and $g_1 \in {\mathcal G}_{p,1,\delta}$ such that ${\rm dist}(g_1, f) \leq \e/2$. Denote by $H_0, K_1 \subset H_0$ the periodic shrinking boxes for $g_1$ as a map of ${\mathcal G}_{p,1,\delta}$ (recall Definition \ref{definitionUpToLeveln} for $n=1$). By continuity, there exists $0<\e _1< \e$ such that for all $g$ in the $\e_1$-neighborhood of $g_1$, $H_0$ is still a periodic shrinking box of period $p$ for $g$.

By induction on $n \geq 1$, (Lemma \ref{lemmaLeveln} provides the inductive step), there exists a sequence of maps $g_1, g_2, \ldots, g_n, \ldots$ and a strictly decreasing sequence of positive real numbers $\e > \e_1> \e_2 > \ldots > \e_n> \ldots$ such that, for all $n \geq 1$,
$g_n  \in    {\mathcal G}_{p,n,\delta}$, ${\rm dist}(g_{n+1}, g_n) \leq \e_n/2^n$, the boxes $H_0, H_1, \ldots, H_{n-1}$ and $K_1, K_2, \ldots, K_n$ are still periodic shrinking for $g_{n+1}$ with the same periods $p, p_1, p_2, \ldots, p_n$ as for $g_n$, and $g_{n+1}= g_n$ when restricted to the $g_n$-orbits of the boxes $K_j$ for $j=1,\ldots, n$. Besides, for all $g$ in the $\e_n$-neighborhood of $g_n$,  $ H_{n-1}$ is still a periodic shrinking box of period $p$ for $g$.

 Since
 $\|g_{n+1} - g_{n}\|  \leq  {\e}/{2^{n+1}}$ for all $n \geq 1$
the sequence  $\{g_n\}_{n \geq 1}$ is   Cauchy in $C^0(M)$ or ${\rm Hom}(M)$, let $g$ be the limit map.  Since   $g_n$ is an $\e$-perturbation of $f$ for all $n \geq 1$, the limit map $g$ satisfies ${\rm dist}(g, f) \leq \e$.

Besides, by construction $g_k (x) = g_n(x)$ for all $x \in \bigcup_{j=0}^{p_n} g_n^j(K_n)$, for all $k \geq n \geq 1$. So $g_k^{p_n}|_{K_n} = g_n^{p_n}|_ {K_n}$
is topologically conjugated to  $ \Phi_n \in {\mathcal H}$ for all $n \geq 1$ and for all $k \geq n$ (recall that $g_n \in {\mathcal G}_{p,n,\delta}$ and Definition \ref{definitionUpToLeveln}).  Thus
$K_n$ is still a periodic shrinking box for $g$ of period $p_n$, and $g^{p_n}|_{K_n} = g_n^{p_n}|_ {K_n}   $
is topologically conjugated to  a model map for all $n \geq 1$. Finally, for all $k > n \geq 1$
we have, by construction, ${\rm dist}(g_k, g_n) < \e _n (1/2^{n+1} + 1/2^{n+2} + \cdots + 1/2^k)\leq \e_n $. So, taking limit as $k \rightarrow + \infty$, we obtain ${\rm dist}(g , g_n) \leq \e_n $. This implies that
 $H_{n-1}$ is still a periodic shrinking box of period $p$ for $g$  as it was for $g_n$. We have proved that $g \in {\mathcal G}_{\delta}$, as required.
\end{proof}

\begin{lemma} \label{lemma3}  \label{lemma4}
For $m \geq 1$  a generic  map  $f \in C^0(M)$,  and for $m \geq 2$ a generic homeomorphism $f$  has a  good sequence $\{K_n\}$ of  boxes, such that the return map $f^{p_n}|_{K_n}$  is topologically conjugated to a   model $\Phi_n \in {\mathcal H}$. 
\end{lemma}

\begin{proof} To see the $G_{\delta}$ property assume that $f $ has a good sequence $\{K_n\}_n$ of periodic shrinking boxes. For each fixed $n$, the boxes $K_n$ and $H_n$ are also periodic shrinking with  periods $p_n $ and $p$ respectively, for all $g$ in an open set  in $C^0(M)$ or in ${{\rm Hom}}(M)$ (see Definition \ref{DefinitionPerShrBox}).  Taking the intersection of such open sets  for all $n \geq 1$, we deduce that the same  sequence $\{K_n\}$  is also a  good sequence of periodic shrinking boxes for all $g$ in a $G_{\delta}$-set. Now,   assume that besides $f^{p_n}|_{K_n}$ is topologically conjugated to a model map for all $n \geq 1$. From Lemmas \ref{Lemma2} and \ref{Lemma2b}, for each fixed $n \geq 1$,  the family of continuous maps  $g$ such that  the return map $g^{p_n}|_{K_n}$ is topologically conjugated to a model, is a $G_{\delta}$-set in  $C^0(M)$ or in ${{\rm Hom}}(M)$. The (countable) intersection of  these $G_{\delta}$-sets, produces a $G_{\delta}$-set, as required.

To prove denseness, recall Definitions \ref{definitionUpToLeveln} and \ref{definitionGdelta}. Observe that the family of continuous maps or homeomorphisms that have a good sequence $\{K_n\}_{n \geq 1}$ of periodic shrinking boxes such that
the return map  to each $K_n$ is topologically conjugated to a model map, contains the family
 $ {\mathcal G}_{ \delta}$ (or the intersection of this family with ${\rm Hom}(M)$), for any value of $\delta>0$. Applying Lemma \ref{lemmaDensity}, this latter family is dense.
 \end{proof}

\begin{remark}
\label{RemarkLemma4} \em   As a consequence of Lemmas \ref{lemma4} and \ref{LemmaMain} (after applying the same arguments at the end of the proof  of Theorems \ref{Theorem1} and \ref{Theorem3}),    generic  continuous maps and homeomorphisms  $f$ have a sequence  of ergodic measures $\mu_n$,  each one supported on the $f$-orbit of a box $K_n$ of a good sequence $\{K_n\}_{n \geq 1}$ of periodic shrinking boxes for $f$,   satisfying $h_{\mu_n}(f) = + \infty $ for all $n \geq 1$.
\end{remark}

Let ${\mathcal M}$ denote the  metrizable space of Borel probability measures on a compact metric space $M$, endowed with  the weak$^*$ topology. Fix a metric  $ {{{\rm dist}}}^*$ in $\mathcal M$.

\begin{lemma}
\label{LemmaMeasureDistance}For all $\e >0$ there exists $\delta>0$ that satisfies the following property:
if $\mu, \nu \in {\mathcal M}$ and  $\{ B_1,  B_2, \ldots,  B_r\} $ is a finite family of pairwise disjoint compact balls $ B_i \subset M$, and if  \em
${\rm supp}(\mu) \cup {\rm supp}(\nu) \subset \bigcup_{i=1}^r  B_i ,$
and $ \mu( B_i) = \nu( B_i) , $ $  {{\rm diam}}( B_i) < \delta $ for all $ i= 1,2, \ldots, r, $
\em then \em  $ {{{\rm dist}}}^*(\mu, \nu) < \e.$
\end{lemma}
\begin{proof}
If $M= [0,1]$ the proof is in \cite[Lemma 4]{CT2017}. If $M$ is any other compact manifold of finite dimension $m \geq 1$, with or without boundary, just copy the proof of \cite[Lemma 4]{CT2017} by substituting the pairwise disjoint compact intervals $I_1, I_2, \ldots, I_r \subset [0,1]$ in that proof, by the family of pairwise disjoint compact boxes $ B_1,  B_2, \ldots,  B_r \subset M$.
\end{proof}

\begin{proof}[Proofs of Theorems \ref{Theorem2} and \ref{Theorem4}]
Fix  $\e>0$, let  $\delta>0$ satisfy Lemma \ref{LemmaMeasureDistance}.
Applying Lemma \ref{lemma4}, generic continuous maps or homeomorphisms $f $ have a good sequence of periodic shrinking boxes $\{K_n\}_{n \geq 1}$, and a sequence $\{\mu_n\}$ of ergodic $f$-invariant measures  such that $h_{\mu_n}(f) = + \infty$ (see Remark \ref{RemarkLemma4}) and such that
${\rm supp}(\mu_n) \subset \bigcup_{j= 0}^{p_n-1} f^j(K_n),$
where $p_n  = l_n \cdot p$, multiple of $p$, is the period of the shrinking box $K_n$.
Taking into account that  $\{f^j({K_n})\}_{0 \leq j \leq p_n-1}$ is a family of pairwise disjoint compact sets, and $f^{p_n}(K_n) \subset {{{\rm int}}}(K_n)$, we obtain for each $j \in \{0,1,\dots,p_n\}$
$$ \mu_n(f^j(K_n)) = \mu_n(f^{\shortminus j}(f^j(K_n)))= \mu_n(f^{\shortminus j}(f^j(K_n))\cap {\rm supp}(\mu_n)   )=
\mu_n(K_n).$$
Since
$1= \sum_{j=0}^{p_n-1}   \mu_n(f^j(K_n)) = p_n \cdot \mu_n (K_n);$ we obtain
$$\mu_n(f^j(K_n)) = \mu_n(K_n) = \frac{1}{p_n} = \frac{1}{l_n \, p}, \qquad \forall \  j= 0,1, \ldots, p_n.$$

From Definition \ref{definitionCleanSequence}, there exists a periodic point $x_0$ of   period $p $ such that  $\lim _{n \rightarrow + \infty} \sup_{j \geq 0} {{\rm H{dist}}}(f^j(K_n), f^j(x_0)) = 0,$ where
${{\rm Hdist}}$ denotes the Hausdorff distance. Therefore, there exists $n_0 \geq 1$ such that
 $\displaystyle d  (f^j  (K_n), f^j(x_0)  )  <   {\delta'}  $ for all $ j \geq 0 $ and for all $ n \geq n_0,$ where $\delta' < \delta/2$ is chosen such that
  the family of compact balls $  B_0$, $  B_1$, \ldots, $  B_{p-1}$,  centered at the points $ f^j(x_0)$ and with radius $\delta'$,  are pairwise disjoint.  We obtain
$f^j(K_n) \subset  B_{j \pmod{p}}$ for all $ j \geq 0 $ and for all $  n \geq 0. $
Therefore,
$$\mu_n (  B_j) = \frac{1}{p}, \qquad \forall \ j=0,1, \ldots, p-1, \ \ \ \  \forall \ n \geq n_0.$$
Finally, applying Lemma \ref{LemmaMeasureDistance}, we conclude
 ${{\rm dist}}^*(\mu_n, \mu_0) < \e $ for all $n \geq n_0,$
where  $\mu_0 := (1/p) \sum_{j= 0}^{p-1} \delta_{f^j(p)}$
is the $f$-invariant probability measure supported on the periodic orbit of $x_0$, which has zero entropy.
\end{proof}

\section{Open questions}
Lipschitz maps have finite topological entropy and thus can not have infinite entropy invariant measures.
The following question
arises: do Theorems  \ref{Theorem1} and \ref{Theorem3}   hold  also for maps with more regularity than continuity but lower regularity than Lipschitz? For instance, do they   hold for  H\"{o}lder-continuous maps?

A priori  there is a chance to answer this question positively in situations where the topological entropy is generically infinite, for example 
for one-dimensional H\"{o}lder continuous endomorphisms and  also   for  bi-H\"{o}lder homeomorphisms  on manifolds of dimension  2 or larger. In both case generic infinite entropy is known  \cite{FHT}, \cite{FHT1}.  This is a good question for further reasearch.

Theorems \ref{Theorem1} and \ref{Theorem3}    are proved for compact manifolds, we wonder if some of the results also hold in   other compact metric spaces  that  are not manifolds? Do they hold if the space   is a Cantor set $K$?

If the aim were just to construct    $f  \in {{\rm Hom}}(K)$ with ergodic measures  with infinite metric entropy, the answer is positive. Theorem  \ref{Theorem3}  holds  for the 2-dimensional square $D^2:=[0,1]^2$. One of the steps of the proof  consists in constructing  a Cantor set   $\Lambda \subset D^2$,  and a homeomorphism  $\Phi$ on $M$  that leaves  $\Lambda$ invariant, and  possesses an  $\Phi$-invariant ergodic measure supported on $\Lambda$ with   infinite metric entropy (see Lemma \ref{LemmaMain} and Remark \ref{RemarkMainLemma}).
Since any pair of Cantor sets $K$ and $\Lambda$ are homeomorphic, we deduce that any Cantor set $K$ supports a homeomorphism $f$ and an $f$-ergodic measure  with infinite metric entropy.

If the purpose were to prove that such homeomorphisms are generic in ${{\rm Hom}}(K)$, the answer is negative.
 On the one hand, there also exists   homeomorphisms on $K$ with finite, and even zero,  topological entropy, for example  $f \in {{\rm Hom}}(K)$  conjugated to the homeomorphism  on the attractor  of a Smale horseshoe, or     to the attractor of the $C^1$- Denjoy example on the circle. On the other hand,  it is known  that each homeomorphism on a Cantor set $K$ is topologically locally unique; i.e., it is conjugated to any of its small perturbations \cite{AGW}. Therefore, the topological entropy is locally constant in ${{\rm Hom}}(K)$. We conclude that the homeomorphisms on the Cantor set $K$ with infinite metric entropy, that do exist, are not  dense  in $\hbox{{{\rm Hom}}}(K)$; hence they are not generic.

\end{document}